\documentclass[a4paper, DIV=12, 11pt]{amsart}
\usepackage[latin1]{inputenc}
\usepackage{amsmath}
\usepackage{amsfonts}
\usepackage{amssymb}
\usepackage{amsthm}
\usepackage{graphicx}
\usepackage{thmtools}
\usepackage{mathtools}
\usepackage{bbm}
\usepackage{hyperref}
\usepackage[bottom, marginal]{footmisc}
\usepackage{todonotes}

\usepackage{tikz}
\usetikzlibrary{matrix}

\declaretheoremstyle[headfont=\normalfont]{normalhead}
\newtheorem{lemma}{Lemma}[section]
\newtheorem{theorem}[lemma]{Theorem}
\newtheorem{proposition}[lemma]{Proposition}
\newtheorem{corollary}[lemma]{Corollary}
\newtheorem{definition}[lemma]{Definition}
\newtheorem{remark}[lemma]{Remark}

\newcounter{mt}

\newtheorem{maintheorem}[mt]{Theorem}

\newcommand{\R}{\mathbb{R}}
\newcommand{\C}{\mathbb{C}}

\DeclareMathOperator{\Val}{Val}
\DeclareMathOperator{\VConv}{VConv}
\DeclareMathOperator{\Conv}{Conv}
\DeclareMathOperator{\vol}{vol}

\DeclareMathOperator{\supp}{supp}

\DeclareMathOperator{\diam}{diam}

\DeclareMathOperator{\Hess}{Hess}
\DeclareMathOperator{\GL}{GL}

\DeclareMathOperator{\GW}{\mathrm{GW}}

\DeclareMathOperator{\Gr}{\mathrm{Gr}}

\DeclareMathOperator{\Aff}{\mathrm{Aff}}
\DeclareMathOperator{\Sym}{\mathrm{Sym}}

\DeclareMathOperator{\SO}{\mathrm{SO}}

\DeclareMathOperator{\MAVal}{\mathrm{MAVal}}

\DeclareMathOperator{\F}{\mathbb{F}}
\DeclareMathOperator{\MA}{\mathrm{MA}}
\DeclareMathOperator{\Mat}{\mathrm{Mat}}

\renewcommand{\Re}{\operatorname{Re}}
\renewcommand{\Im}{\operatorname{Im}}
\renewcommand{\O}{\mathcal{O}}
\renewcommand{\P}{\mathcal{P}}
\newcommand{\M}{\mathcal{M}}

\newcommand{\I}{\mathcal{I}}

\newcommand{\diag}{d}

\author{Jonas Knoerr}
\title[Paley--Wiener--Schwartz Theorem for valuations]{A Paley--Wiener--Schwartz Theorem for smooth valuations on convex functions}
\date{}

\newcommand{\Addresses}{{
		\bigskip
		\footnotesize
		
		Jonas Knoerr, \textsc{Institute of Discrete Mathematics and Geometry, TU Wien, Wiedner Hauptstrasse 8-10, 1040 Wien, Austria}\par\nopagebreak
		\textit{E-mail address}: \texttt{jonas.knoerr@tuwien.ac.at}
		
		\medskip
	}}
	
\makeatletter
\def\blfootnote{\xdef\@thefnmark{}\@footnotetext}
\makeatother

\makeindex
\begin{document}
\maketitle
\begin{abstract}
	Continuous dually epi-translation invariant valuations on convex functions are characterized in terms of the Fourier--Laplace transform of the associated Goodey--Weil distributions. This description is used to obtain integral representations of the smooth vectors of the natural representation of the group of translations on the space of these valuations. As an application, a complete classification of all closed and affine invariant subspaces is established, yielding density results for valuations defined in terms of mixed Monge-Amp\`ere operators.
\end{abstract}

\section{Introduction}

Let $\mathcal{K}(\R^n)$ denote the space of convex bodies in $\R^n$, that is, the set of all nonempty, convex and compact subsets of $\R^n$ equipped with the Hausdorff metric. A functional $\mu:\mathcal{K}(\R^n)\rightarrow \C$ is called a valuation if it is finitely additive in the following sense:
\begin{align*}
	\mu(K\cup L)+\mu(K\cap L)=\mu(K)+\mu(L)
\end{align*}
for all $K,L\in\mathcal{K}(\R^n)$ such that $K\cup L$ is convex. The class of valuations contains many well known geometric functionals, for example Euler characteristic, measures, and mixed volumes, and consequently, valuations play an important role in various areas of mathematics. Over the last 30 years, the theory of valuations has developed rapidly, which has led to a variety of new classification results and applications in convex, integral, discrete, and differential geometry. We refer to \cite{AleskerDescriptiontranslationinvariant2001,AleskerFaifmanConvexvaluationsinvariant2014,BernigHadwigertypetheorem2009,BernigBroeckerValuationsmanifoldsRumin2007,BernigEtAlCurvaturemeasurespseudo2022,BernigEtAlHardLefschetztheorem2024,FaifmanHofstaetterConvexvaluationsWhitney2023,FreyerEtAlUnimodularValuationsEhrhart2024,FuStructureunitaryvaluation2006,KotrbatyWannererharmonicanalysistranslation2023,LudwigReitznerclassification$SLn$invariant2010,Wannerermoduleunitarilyinvariant2014} for some of these advances.
Much of this progress is rooted in Alesker's solution of the following conjecture by McMullen \cite{McMullenContinuoustranslationinvariant1980}. Let $\Val(\R^n)$ denote the space of all continuous and translation invariant valuations $\mu:\mathcal{K}(\R^n)\rightarrow\C$. This space is a complete locally convex vector space (in fact, a Banach space) with respect to the topology of uniform convergence on compact subsets of $\mathcal{K}(\R^n)$.
\begin{theorem}[Alesker \cite{AleskerDescriptiontranslationinvariant2001}]
	\label{theorem:McMullenConjecture}
	Linear combinations of mixed volumes span a dense subspace of $\Val(\R^n)$.
\end{theorem}
Alesker obtained Theorem~\ref{theorem:McMullenConjecture} as a consequence of a much stronger result, which is now known as the \emph{Irreducibility Theorem}. Note that there is a natural continuous representation $\pi$ of $\GL(n,\R)$ on $\Val(\R^n)$ given for $g\in \GL(n,\R)$ and $\mu\in\Val(\R^n)$ by
\begin{align*}
	[\pi(g)\mu](K)=\mu(g^{-1}K)\quad\text{for}~K\in\mathcal{K}(\R^n).
\end{align*}
By a result of McMullen \cite{McMullenValuationsEulertype1977}, $\Val(\R^n)$ decomposes into a direct sum of the spaces $\Val_k(\R^n)=\{\mu\in\Val(\R^n):\mu(tK)=t^k\mu(K),t\ge 0,K\in\mathcal{K}(\R^n)\}$:
\begin{align*}
	\Val(\R^n)=\bigoplus_{k=0}^n\Val_k(\R^n).
\end{align*}
Furthermore, these spaces decompose further into even and odd valuations: $\Val_k(\R^n)=\Val_k^+(\R^n)\oplus\Val_k^-(\R^n)$, where $\mu\in\Val^\pm_k(\R^n)$ if and only if it satisfies $\mu(-K)=\pm\mu(K)$ for $K\in\mathcal{K}(\R^n)$. Obviously these spaces are $\GL(n,\R)$-invariant. The Irreducibilty Theorem states that these are the only nontrivial, $\GL(n,\R)$-invariant closed subspaces of $\Val(\R^n)$.
\begin{theorem}[Alesker \cite{AleskerDescriptiontranslationinvariant2001}]
	\label{theorem:IrreducibilityTheorem}
	The natural representation of $\GL(n,\R)$ on $\Val_k^\pm(\R^n)$ is irreducible, that is, every nontrivial $\GL(n,\R)$-invariant subspace is dense.
\end{theorem}
Theorem~\ref{theorem:McMullenConjecture} is simple consequence of this result, since linear combinations of mixed volumes form a $\GL(n,\R)$-invariant subspace that intersects the spaces $\Val^\pm_k(\R^n)$ nontrivially. In fact, it is sufficient to take mixed volumes with simplices as reference bodies (or ellipsoids in the even case), as pointed out in \cite{AleskerDescriptiontranslationinvariant2001}. These results can also be used to show that linear combinations of mixed volumes with a  bounded (but sufficiently large) number of terms are dense \cite{KnoerrSmoothvaluationsconvex2024a}.\\

The valuation property admits a natural extension to spaces $X$ of (extended) real-valued functions. Here, a map $\mu: X\rightarrow F$ into a vector space $F$ is called a valuation if
\begin{align*}
	\mu(f\vee h)+\mu(f\wedge h)=\mu(f)+\mu(h)
\end{align*}
for all $f,g\in X$ such that the pointwise maximum $f\vee h$ and minimum $f\wedge h$ belong to $X$. Note that this recovers the notion of a valuation on convex bodies if $X$ denotes the corresponding set of indicator functions or support functions. There has been a considerable effort to identify many classical functionals as the unique (continuous) valuations satisfying additional invariance properties (see \cite{ColesantiEtAlclassinvariantvaluations2020,ColesantiEtAlContinuousvaluationsspace2021,KnoerrMongeAmpereoperators2024,LiMaLaplacetransformsvaluations2017,LudwigFisherinformationmatrix2011,LudwigValuationsSobolevspaces2012,VillanuevaRadialcontinuousrotation2016}) and to understand the general structure of these functionals on a variety of spaces (see \cite{ColesantiEtAlhomogeneousdecompositiondually2024,TradaceteVillanuevaRadialcontinuousvaluations2017,TradaceteVillanuevaContinuityrepresentationvaluations2018,TradaceteVillanuevaValuationsBanachlattices2020}).\\
Due to their intimate relation to convex bodies, spaces of convex functions have been at the center of this part of valuation theory. This includes a variety of characterization results \cite{AleskerValuationsconvexfunctions2019,ColesantiEtAlMinkowskivaluationsconvex2017,ColesantiEtAlHadwigertheoremconvex2022,ColesantiEtAlHadwigertheoremconvex2023,ColesantiEtAlHadwigertheoremconvex2024,HofstaetterKnoerrEquivariantValuationsConvex2024,Knoerrgeometricdecompositionunitarily2024,MouamineMussnigVectorialHadwigerTheorem2025,MussnigVolumepolarvolume2019,Mussnig$SLn$InvariantValuations2021} as well as functional versions of a number of structural results \cite{ColesantiEtAlhomogeneousdecompositiontheorem2020,Knoerrsupportduallyepi2021,KnoerrUlivelliPolynomialvaluationsconvex2024}.\\

In this article, we consider valuations on the space $\Conv(\R^n,\R)$ of all convex functions $f:\R^n\rightarrow\R$, i.e. all convex functions on $\R^n$ that take finite values. This space is naturally equipped with the metrizable topology induced by locally uniform convergence (which in this case coincides with the topology induced by epi-convergence or pointwise convergence, see Section \ref{section:Preliminaries_convexFunctions}). A valuation $\mu$ on $\Conv(\R^n,\R)$ is called \emph{dually epi-translation invariant} if
\begin{align*}
	\mu(f+\ell)=\mu(f)\quad\text{for all}~f\in\Conv(\R^n,\R),~\ell:\R^n\rightarrow\R~\text{affine}.
\end{align*}
Valuations with this property are intimately related to translation invariant valuations on convex bodies \cite{AleskerValuationsconvexfunctions2019,Knoerrsupportduallyepi2021,KnoerrSmoothvaluationsconvex2024,KnoerrUlivellivaluationsconvexbodies2024} as well as to Monge-Amp\`ere-type operators \cite{AleskerValuationsconvexsets2005,AleskerValuationsconvexfunctions2019,KnoerrMongeAmpereoperators2024}. For example, the real Monge-Amp\`ere operator $\MA$ may be considered as a continuous valuation on $\Conv(\R^n,\R)$ with values in the space $\mathcal{M}(\R^n):=C_c(\R^n)'$ of (complex) Radon measures, considered as the continuous dual space of $C_c(\R^n)$ equipped with the weak-* topology. For $f\in\Conv(\R^n,\R)\cap C^2(\R^n)$, $\MA(f)$ is given by integrating the determinant of the Hessian $D^2f$ over a given Borel set, which extends by continuity to arbitrary convex functions. More generally, one can associate to functions $f_1,\dots,f_n\in\Conv(\R^n,\R)$ their mixed Monge-Amp\`ere measure defined by
\begin{align*}
	\MA(f_1,\dots,f_k):=\frac{1}{n!}\frac{\partial^n}{\partial\lambda_1\dots\partial\lambda_n}\Big|_0\MA\left(\sum_{j=1}^n\lambda_j f_j\right),
\end{align*}
which is well defined since $(\lambda_1,\dots,\lambda_n)\mapsto \MA\left(\sum_{j=1}^n\lambda_j f_j\right)$ is a polynomial in $\lambda_1,\dots,\lambda_n\ge 0$. From these, a variety of scalar valued valuations can be obtained, as first shown by Alesker \cite{AleskerValuationsconvexfunctions2019} and Colesanti, Ludwig, and Mussnig \cite{ColesantiEtAlHessianvaluations2020}. The constructions in \cite{AleskerValuationsconvexfunctions2019} rely on the fact that for $B\in C_c(\R^n)$ and functions $A_1,\dots,A_{n-k}\in C_c(\R^n,\Sym^2(\R^n))$ with values in the space $\Sym^2(\R^n)$ of symmetric $n\times n$ real matrices, the functional defined on $\Conv(\R^n,\R)\cap C^2(\R^n)$ given by
\begin{align}
	\label{equation:ValuationsMAOperators}
	f\mapsto \int_{\R^n} B(x)\det(D^2f(x)[k],A_1(x),\dots, A_{n-k}(x))dx,
\end{align}
where $\det:(\Sym^2(\R^n))^n\rightarrow\R$ denotes the mixed discriminant, extends uniquely by continuity to an element of $\VConv_k(\R^n)$. 
Using the multilinearity of the mixed discriminant, it is not difficult to see that the valuations defined in \eqref{equation:ValuationsMAOperators} can all be expressed in terms of the Monge-Amp\`ere operators $f\mapsto \MA(f[k],Q_1,\dots,Q_{n-k})$, where $Q_1,\dots,Q_{n-k}$ are positive semi-definite quadratic forms and $f$ is used $k$-times as an argument.\\

In this article we will be interested in scalar-valued valuations. Let $\VConv(\R^n)$ denote the space of all continuous valuations $\mu:\Conv(\R^n,\R)\rightarrow\C$ that are dually epi-translation invariant. As shown by Colesanti, Ludwig, and Mussnig \cite{ColesantiEtAlhomogeneousdecompositiontheorem2020}, this space admits a homogeneous decomposition mirroring the homogeneous decomposition of $\Val(\R^n)$: If we denote by $\VConv_k(\R^n)$ the subspace of all valuations $\mu$ that are $k$-homogeneous, i.e. that satisfy $\mu(tf)=t^k\mu(f)$ for all $f\in\Conv(\R^n,\R)$, $t\ge 0$, then 
\begin{align*}
	\VConv(\R^n)=\bigoplus_{k=0}^n\VConv_k(\R^n).
\end{align*}
For $k=0$, $\VConv_0(\R^n)$ consists of constant valuations, while it was shown in \cite{ColesantiEtAlhomogeneousdecompositiontheorem2020} that the valuations in the top component $\VConv_n(\R^n)$ are all of the form \eqref{equation:ValuationsMAOperators}. For the intermediate degrees, no full classification is known, however, there exists a description of a certain dense subspace based on Theorem~\ref{theorem:IrreducibilityTheorem} (see \cite{KnoerrSmoothvaluationsconvex2024}).

\subsection{Main results}
	In this article we establish a version of the Irreducibility Theorem~\ref{theorem:IrreducibilityTheorem} for the spaces $\VConv_k(\R^n)$. Let $\Aff(n,\R)$ denote the group of invertible affine transformations of $\R^n$. We have a natural representation $\pi$ of $\Aff(n,\R)$ on $\VConv(\R^n)$, given for $g\in\Aff(n,\R)$, $\mu\in\VConv(\R^n)$ by
	\begin{align*}
		[\pi(g)\mu](f)=\mu(f\circ g)\quad \text{for}~f\in\Conv(\R^n,\R).
	\end{align*}
	We equip $\VConv(\R^n)$ with the topology of uniform convergence on compact subsets of $\Conv(\R^n,\R)$. Then it is not difficult to see that the representation above is continuous (compare Section \ref{section:Valuations_Support_Topology}).\\
	The homogeneous components $\VConv_k(\R^n)$ are obviously $\Aff(n,\R)$-invariant subspaces, however, they are far from the only ones. In Section \ref{section:affineInvSubspaces} we provide a complete classification of these subspaces as the closure of certain spaces of valuations defined in terms of mixed Monge-Amp\`ere operators. The construction implies the following description of these spaces.
	\begin{maintheorem}
		\label{maintheorem:WeakIrreducibility}
		\begin{enumerate}
			\item Every affine invariant closed subspace of $\VConv_k(\R^n)$ has finite codimension.
			\item The set of affine invariant closed subspaces of $\VConv_k(\R^n)$ is totally ordered by inclusion.
			\item If $W\subset \VConv_k(\R^n)$ is an affine invariant subspace such that there exists a valuation $\mu\in W$ with $\mu(q)\ne 0$ for a positive semi-definite quadratic form $q$ on $\R^n$, then $W$ is dense in $\VConv_k(\R^n)$. 
		\end{enumerate}
	\end{maintheorem}
	Note that (3) applies in particular to the space spanned by valuations of the form \eqref{equation:ValuationsMAOperators}. Thus, the valuations considered by Alesker in \cite{AleskerValuationsconvexfunctions2019} span a dense subspace of $\VConv(\R^n)$. Let us remark that this fact can also be obtained by combining the main results of \cite{KnoerrMongeAmpereoperators2024,KnoerrSmoothvaluationsconvex2024}, however, Theorem~\ref{maintheorem:WeakIrreducibility} (3) applies to much more general spaces of valuations constructed from mixed Monge-Amp\`ere operators, and we list some of these examples in Section \ref{section:denseSubspaces}.\\
	
	The proof of Theorem~\ref{maintheorem:WeakIrreducibility} is based on the interplay between three different notions of regularity for elements in $\VConv_k(\R^n)$, which turn out to coincide.	Let us discuss the three notions.
	\subsubsection{Representation theoretic/analytic perspective}
	The first notion mirrors the definition of smooth valuations on convex bodies as the smooth vectors of the $\GL(n,\R)$-representation $\Val(\R^n)$, compare \cite{AleskerTheoryvaluationsmanifolds.2006}. It turns out that the action of the subgroup of $\Aff(n,\R)$ given by translations is sufficient for our purposes.
	\begin{definition}
		\label{definition:SmoothValuations}
		A valuation $\mu\in\VConv_k(\R^n)$ is called a smooth valuation if the map 
		\begin{align*}
			\R^n&\rightarrow \VConv_k(\R^n)\\
			x&\mapsto \left[f\mapsto \mu(f(\cdot+x))\right]
		\end{align*}
		is smooth.
	\end{definition}
	A standard convolution argument shows that smooth valuations are dense in $\VConv_k(\R^n)$, compare Corollary~\ref{corollary:ApproxTranslationSmoothValuations}. 
	
	Let us remark that this terminology differs from the notion of smooth valuation used in \cite{KnoerrSmoothvaluationsconvex2024}, however, we show in Theorem~\ref{maintheorem:DescriptionsSmoothValuations} below that these two different notions are equivalent. The benefit of Definition \ref{definition:SmoothValuations} is the greater flexibility of this notion in approximation arguments. In particular, it enables approximation in translation invariant subspaces. On the other hand, it does a priori not provide an explicit representation of a given smooth valuation. 
	
	\subsubsection{Geometric perspective}
		In \cite{KnoerrSmoothvaluationsconvex2024}, the following geometric construction of valuations on convex functions was investigated: If $f\in\Conv(\R^n,\R)$ is sufficiently smooth, the graph of its differential is an $n$-dimensional submanifold of the cotangent bundle $T^*\R^n=\R^n\times (\R^n)^*$, and thus an integral current. As shown by Fu \cite{FuMongeAmperefunctions.1989}, this construction extends to arbitrary convex functions, which gives rise to the \emph{differential cycle}, an integral current $D(f)$ on $T^*\R^n$ associated to any $f\in\Conv(\R^n,\R)$.  By the main results of \cite{KnoerrSmoothvaluationsconvex2024}, any differential form $\omega\in \Omega^{n-k}_c(\R^n)\otimes \Lambda^{k}((\R^n)^*)^*$ induces a continuous valuation in $\VConv_k(\R^n)$ by $
		\mu(f)=D(f)[\omega]\quad\text{for}~f\in\Conv(\R^n,\R)$. In order to differentiate this class of valuations from smooth valuations, we will call valuations of this type \emph{representable by integration with respect to the differential cycle}.\\
		
		It follows from the results in \cite{KnoerrMongeAmpereoperators2024} that any such valuation can be written as a finite linear combination of valuations given by integrating elements of $C^\infty_c(\R^n)$ with respect to the mixed Monge-Amp\`ere operators
		\begin{align}
			\label{equation:MAquadratic}
			f\mapsto \MA(f[k],Q_1,\dots,Q_{n-k}),
		\end{align}
		where $Q_1,\dots,Q_{n-k}$ are positive semi-definite quadratic forms on $\R^n$. More precisely, the results in \cite{KnoerrMongeAmpereoperators2024} provide different equivalent characterizations of a certain space of measure-valued valuations $\Psi:\Conv(\R^n,\R)\rightarrow\mathcal{M}(\R^n)$, which we denote by $\MAVal_k(\R^n)$ (see Section \ref{section:MAOperators} for the definition and the precise statement). One of these characterizations identifies $\MAVal_k(\R^n)$ with linear combinations of the Monge-Amp\`ere operators in \eqref{equation:MAquadratic}. In particular $\MAVal_k(\R^n)$ is a finite dimensional space.

	\subsubsection{Fourier analytic perspective}
	In order to relate the two previous notions, we turn to a characterization of smoothness for distributions in terms of the properties of their Fourier--Laplace transform. The classical Paley--Wiener--Schwartz Theorem characterizes the entire functions on $\C^n$ that can be obtained as the Fourier--Laplace transform of a compactly supported distribution on $\R^n$
	in terms of the decay properties of these functions. For smooth functions with compact support, it entails the following characterization (see \cite[Theorem 7.3.1]{Hoermanderanalysislinearpartial2003}).	
	\begin{theorem}
		\label{theorem:PaleyWienerSchwartz}
		Let $A\subset \R^n$ be compact and convex. If $F$ is an entire function on $\C^n$ such that for every $N\in\mathbb{N}$ there exists a constant $C_N>0$ with
		\begin{align*}
			|F(z)|\le C_N (1+|z|)^{-N}e^{h_A(\Im z)} \quad\text{for}~z\in\C^n,
		\end{align*} 
		then $F$ is the Fourier--Laplace transform of a function in $C^\infty_c(\R^n)$ with support contained in $A$. Conversely, the Fourier--Laplace transform of any smooth function with support contained in $A$ satisfies an estimate of this form for every $N\in\mathbb{N}$.
	\end{theorem}
	Here, $h_A(y):=\sup_{x\in A}\langle y,x\rangle$ denotes the support function of $A$.\\
	
	The main goal of this article is a characterization of smooth valuations in terms of the Fourier--Laplace transform of certain distributions associated to homogeneous valuations in order to show that these valuations are representable by integration with respect to the differential cycle. It turns out that in contrast to the classical Paley--Wiener--Schwartz Theorem a simple decay condition is not sufficient to describe the relevant space of entire functions. Instead, the space of these functions is heavily restricted by additional geometric conditions.\\
	
	Let us make this more precise. It was shown in \cite{Knoerrsupportduallyepi2021} that one can associate to any $\mu\in\VConv_k(\R^n)$ a unique symmetric distribution $\GW(\mu)$ on $(\R^n)^k$ with compact support such that
	\begin{align}
		\label{eq:CharacterizingPropertyGW}
		\mu(f)=\GW(\mu)[f^{\otimes k}]\quad\text{for all}~f\in\Conv(\R^n,\R)\cap C^\infty(\R^n).
	\end{align}
	Here $f^{\otimes k}$ denotes the smooth function on $(\R^n)^k$ defined by $f^{\otimes k}(x_1,\dots,x_k)=f(x_1)\dots f(x_k)$ for $x_1,\dots,x_k\in\R^n$. This construction goes back to ideas of Goodey and Weil \cite{GoodeyWeilDistributionsvaluations1984} and we call $\GW(\mu)$ the Goodey--Weil distribution associated to $\mu$. Since smooth functions are dense in $\Conv(\R^n,\R)$, \eqref{eq:CharacterizingPropertyGW} implies that a valuation is uniquely determined by its Goodey--Weil distribution.\\
	
	One of the key properties of these distributions is that their support is contained in the diagonal $\{(x,\dots,x)\in(\R^n)^k:x\in\R^n\}$. Standard facts from distribution theory imply that the Fourier--Laplace transform of any compactly supported distribution with this property is a polynomial in the directions normal to the diagonal. Let us encode this property in a convenient way: We identify $(\C^n)^k$ with the space $\Mat_{n,k}(\C)$ of $(n\times k)$-matrices with complex entries and denote by $w=(w_1,\dots,w_k)$ the matrix with column vectors $w_1,\dots,w_k\in \C^n$. Let $\mathcal{O}_{\Mat_{n,k}(\C)}$ denote the space of entire functions on $\Mat_{n,k}(\C^n)$. We consider $\mathcal{O}_{\Mat_{n,k}(\C)}$ as a module over $\mathcal{O}_{\C^n}$, where $g\in \mathcal{O}_{\C^n}$ acts on $F\in \mathcal{O}_{\Mat_{n,k}(\C)}$ by the diagonal action
	 \begin{align*}
	 	(g\bullet F)(w):=g\left(\sum_{j=1}^kw_j\right)F(w), \quad w\in\Mat_{n,k}(\C).
	 \end{align*}
 	Then the Fourier--Laplace transforms of Goodey--Weil distributions belong to an $\mathcal{O}_{\C^n}$-submodule of $\mathcal{O}_{\Mat_{n,k}(\C)}$ generated by polynomials.\\
 	Let $\mathcal{P}(\Mat_{n,k}(\C))$ denote the space of polynomials on $\Mat_{n,k}(\C)$ and consider the subspace $\M^2_k\subset \mathcal{P}(\Mat_{n,k}(\C))$ spanned by quadratic products of the $k$-minors, that is, all quadratic products of the determinants of $(k\times k)$-submatrices of $w\in\Mat_{n,k}(\C)$. We denote by $\widehat{\M^2_k}$ the $\mathcal{O}_{\C^n}$-submodule of $\mathcal{O}_{\Mat_{n,k}(\C)}$ generated by $\M^2_k$.\\
 	The following result encodes the geometric restrictions satisfied by these distributions.
	\begin{maintheorem}
		\label{maintheorem:FourierTransformBelongsToModule}
		For any $\mu\in\VConv_k(\R^n)$, $\mathcal{F}(\GW(\mu))\in \widehat{\M^2_k}$.
	\end{maintheorem}
	For $w\in \Mat_{n,k}(\C)$ we define its diagonal component to be the matrix $\diag(w):=\frac{1}{k}\left(\sum_{j=1}^kw_j,\dots,\sum_{j=1}^kw_j\right)\in\Mat_{n,k}(\C)$. We establish the following version of the Paley--Wiener--Schwartz Theorem for smooth valuations. We refer to Section \ref{section:Valuations_Support_Topology} for the notion of support of elements of $\VConv(\R^n)$.
	\begin{maintheorem}
		\label{maintheorem:PWSSmoothValuations}
		Let $A\subset\R^n$ be compact and convex. 
		A valuation $\mu\in\VConv_{k}(\R^n)$ is smooth and satisfies $\supp\mu\subset A$ if and only if the following holds: For every $N\in\mathbb{N}$ there exists a constant $C_N>0$ such that
		\begin{align}
			\label{eq:PWSconditionValuations}
			\begin{split}
				&|\mathcal{F}(\GW(\mu))[w]|\\
				&\quad \le C_N\left(1+\left|\diag(w)\right|\right)^{-N} e^{h_A\left(\sum_{j=1}^k\Im(w_j)\right)} |w-\diag(w)|^{2(k-1)}.
			\end{split}
		\end{align}
		Moreover, if an entire function belongs to $\widehat{\M^2_k}$ and satisfies estimates of the form \eqref{eq:PWSconditionValuations} for every $N\in\mathbb{N}$, then it is the Fourier--Laplace transform of the Goodey--Weil distribution of a unique smooth valuation in $\VConv_k(\R^n)$ with support contained in $A$.
	\end{maintheorem}

	Theorem~\ref{maintheorem:FourierTransformBelongsToModule} and Theorem~\ref{maintheorem:PWSSmoothValuations} are  the key results that connect smooth valuations with valuations that are representable by integration with respect to the differential cycle. We establish the following equivalence between these notions. Recall that $\MAVal_k(\R^n)$ is a finite dimensional space which is spanned by the elements in \eqref{equation:MAquadratic}.
	\begin{maintheorem}
		\label{maintheorem:DescriptionsSmoothValuations}
		The following are equivalent for $\mu\in \VConv_k(\R^n)$:
		\begin{enumerate}
			\item $\mu$ is a smooth valuation in the sense of Definition \ref{definition:SmoothValuations}.
			\item There exists a differential form $\omega\in \Omega^{n-k}_c(\R^n)\otimes \Lambda^k((\R^n)^*)^*$ such that
			\begin{align*}
				\mu(f)=D(f)[\omega]\quad\text{for all}~f\in\Conv(\R^n,\R).
			\end{align*}
			\item For any basis $\Psi_j$, $1\le j\le \dim\MAVal_k(\R^n)$, of $\MAVal_k(\R^n)$ there exist $\phi_j\in C^\infty_c(\R^n)$, $1\le j\le N_{n,k}$, such that
			\begin{align*}
				\mu(f)=\sum_{j=1}^{\dim\MAVal_k(\R^n)}\int_{\R^n} \phi_jd\Psi_j(f)\quad\text{for all}~f\in\Conv(\R^n,\R).
			\end{align*}
		\end{enumerate}
	\end{maintheorem}

	\subsection{Plan of the article}
		
	The main constructions in this article are centered around the proofs of Theorem \ref{maintheorem:FourierTransformBelongsToModule} and Theorem \ref{maintheorem:PWSSmoothValuations}, from which the remaining results are deduced. Let us briefly sketch the idea behind the constructions. \\
	In order to relate the Fourier--Laplace transform of the Goodey--Weil distributions to the space $\M^2_k$ generated by quadratic products of $k$-minors, we first show that the homogeneous terms in the power series expansion all belong to the corresponding module of polynomials. Since the restriction of the Fourier--Laplace transform of the Goodey--Weil distributions encodes a certain restriction procedure for homogenous valuations, we can recover the restrictions of the homogeneous terms in the power series expansion from the classification of valuations of top degree from \cite{ColesantiEtAlhomogeneousdecompositiontheorem2020}.\\
	This reduces the problem to a characterization of a certain submodule of polynomials generated by $\M^2_k$ in terms of the restriction to lower dimensional subspaces, which we obtain using some basic tools from the representation theory of $\GL(n,\C)$.\\
	The next step consists in obtaining a suitable representation of an entire function on $\Mat_{n,k}(\C)$ with the property that the homogeneous terms of its power series expansion belongs to this module. In Section \ref{section:DivisionAlg}, we use some basic tools from complex analysis to obtain a division algorithm that expresses an element with this property as an $\mathcal{O}_{\C^n}$-linear combination of suitable generators such that the coefficients can be estimated in terms of the original function. Combining this result with estimates for the Fourier--Laplace transform of Goodey--Weil distributions of smooth valuations provides a decomposition into a sum of entire functions that admit an interpretation in terms of valuations constructed from elements in $\MAVal_k(\R^n)$, which establishes the desired integral representations.\\
		
		The article is structured as follows:\\
		In Section \ref{section:Preliminaries} we discuss some background on convex functions and establish a description of a certain module of polynomials generated by $\M^2_k$ in terms of the behavior of these polynomials under restrictions. Section \ref{section:DivisionAlg} establishes some results for the decomposition of entire functions belonging to a specific class of submodules generated by polynomials into multiples of generators of these modules. These results are based on some simple properties of Gröbner bases, which we briefly recall.\\
		Section \ref{section:DuallyEpiValuations} contains the necessary background on dually epi-translation invariant valuations. We also discuss some basic properties of the differential cycle and its relation to the space $\MAVal_k(\R^n)$, and we establish basic properties of smooth valuations. This section also contains some approximation results.\\
		In Section \ref{section:FourierGW} we examine the Fourier--Laplace transform of Goodey--Weil distributions and prove Theorem~\ref{maintheorem:FourierTransformBelongsToModule} based on the results in Section \ref{section:DivisionAlg}. These results are combined with an additional estimate for smooth valuations in Section \ref{section:PWSValuations} in order to prove Theorem~\ref{maintheorem:PWSSmoothValuations} and Theorem~\ref{maintheorem:DescriptionsSmoothValuations}.\\
		Section \ref{section:AffInariantSubspaces} contains the proof of Theorem \ref{maintheorem:WeakIrreducibility} as well as some additional density results that follow from this result.

 	\subsection*{Acknowledgments} I want to thank Georg Hofstätter for his comments on the first draft of this article.
 	
\section{Preliminaries}
	\label{section:Preliminaries}
	\subsection{Convex functions}
		\label{section:Preliminaries_convexFunctions}
		We refer to \cite{RockafellarConvexanalysis1997,RockafellarWetsVariationalanalysis1998} for a general background on convex functions and only collect the results needed for the constructions in this article. Recall that we equip $\Conv(\R^n,\R)$ with the topology induced by locally uniform convergence, which coincides on $\Conv(\R^n,\R)$ with pointwise convergence and epi-convergence, compare \cite[Theorem 7.17]{RockafellarConvexanalysis1997}. This implies in particular that the topology on $\Conv(\R^n,\R)$ is metrizable. \\
					
		Let $\Aff(n,\R)$ denote the group of invertible affine transformations of $\R^n$. The following is a simple consequence of the characterization of the topology on $\Conv(\R^n,\R)$ in terms of pointwise convergence and the fact that finite convex functions are continuous.
		\begin{lemma}
			\label{lemma:continuityActionAffonConv}
			The map
			\begin{align*}
				\Aff(n,\R)\times\Conv(\R^n,\R)&\rightarrow\Conv(\R^n,\R)\\
				g&\mapsto f\circ g^{-1}
			\end{align*}
			is continuous.
		\end{lemma}
		Recall that the spaces of valuations considered in this article are equipped with the topology of uniform convergence on compact subsets of $\Conv(\R^n,\R)$. These sets may be characterized in the following way.
		\begin{proposition}[\cite{Knoerrsupportduallyepi2021} Proposition 2.4]
			\label{proposition:compactnessConv}
			A subset $K\subset \Conv(\R^n,\R)$ is relatively compact if and only if it is bounded on compact subsets of $\R^n$, that is, if for any compact subset $A\subset \R^n$ there exists a constant $C(A)>0$ such that
			\begin{align*}
				\sup\limits_{x\in A}|f(x)|\le C(A)\quad \forall f\in K.
			\end{align*}
		\end{proposition}

\subsection{Polynomials and $k$-minors on $\Mat_{n,k}(\C)$}
\label{section:Prelim_Minors}
The goal of this section is a characterization of a certain module of polynomials generated by quadratic products of $k$-minors in terms of the behavior of these polynomials under restrictions to lower dimensional subspaces. Recall that we identify $(\C^n)^k\cong \Mat_{n,k}(\C)$ with the space of complex $(n\times k)$-matrices and aim to investigate certain $\mathcal{O}_{\C^n}$-modules of holomorphic functions on $\Mat_{n,k}(\C)$, where $\mathcal{O}_{\C^n}$ is identified with a subring of $\mathcal{O}_{\Mat_{n,k}(\C)}$ of functions constant along the off-diagonal $\{w\in\Mat_{n,k}(\C):\sum_{j=1}^k w_j=0\}$. It will be convenient to use the change of coordinates
\begin{align*}
	\Mat_{n,k}(\C)&\rightarrow\Mat_{n,k}(\C)\\
	w&\mapsto\left(\frac{w_1+w_k}{k},\dots,\frac{w_{k-1}+w_k}{k},\frac{w_k}{k}-\sum_{j=1}^{k-1}\frac{w_j}{k}\right),
\end{align*}
which maps matrices of the form $(w_1,\dots,w_{k-1},0)$ to the off-diagonal and matrices of the form $(0,\dots,0,w_k)$ to the diagonal $\{(z,\dots,z)\in\Mat_{n,k}(\C): z\in \C^n\}$. In these coordinates, functions that are constant along the off-diagonal correspond to functions depending on the variable $w_k$ only.\\
In order to be consistent with the usual notation for the entries of an $(n\times k)$-matrix, we will denote the components of the column vector $w_j\in \C^n$ by $w_j=(w_{1,j},\dots,w_{n,j})$.\\

For a finite dimensional complex vector space $V$ let $\P(V)$ denote the space of complex polynomials on $V$. We call $P\in\mathcal{P}(\Mat_{n,k}(\C))$ homogeneous of degree $h\in\mathbb{N}^k$ if
\begin{align*}
	P(t_1w_1,\dots,t_kw_k)=t_1^{h_1}\dots t_k^{h_k}P(w)
\end{align*}
for $t_1\dots,t_k\in \C$, $w=(w_1,\dots,w_k)\in\Mat_{n,k}(\C)$, i.e. $P$ is homogeneous with respect to rescaling the columns of the matrix $w$.\\

We denote by $\M_k\subset \P(\Mat_{n,k}(\C))$ the space spanned by $k$-minors. Note that for any $k$-dimensional complex subspace $F\subset \C^n$, we may consider $F^k\subset (\C^n)^k$ as a subset of $\Mat_{n,k}(\C)$. If we choose a basis of $F$, then the restriction of any $k$-minor to $F^k$ is a multiple of the determinant of the coordinate matrix of $w\in F^k$ with respect to the given basis. In particular, the restriction of elements in $\M_k$ to $F^k$ defines a $1$-dimensional subspace in $\P(F^k)$.\\
Let ${}^\C\Gr_k(\C^n)$ denote the Grassmannian of $k$-dimensional complex subspaces of $\C^n$, $\Gr_k(\R^n)$ the Grassmannian of $k$-dimensional real subspaces of $\R^n$. For $\Delta\in \M_k$ set
\begin{align*}
	U_\Delta:=\{E\in {}^\C\Gr_k(\C^n): \Delta|_{E^k}\ne 0\}.
\end{align*}
Note that this is a Zariski dense subset of ${}^\C\Gr_k(\C^n)$ unless $\Delta$ vanishes identically. In the rest of this section, we consider polynomials that restrict to certain multiples of $\Delta^2$ on $U_\Delta$. The first lemma shows that it is sufficient to consider restrictions to complexifications of real subspaces.
\begin{lemma}
	\label{lemma:PolynomialFromComplexifiedSpaces}
	Let $P\in\P(\Mat_{n,k}(\C))$ be a polynomial with the following property: For every $\Delta\in \M_k$ and every $E\in U_\Delta\cap \{E_0\otimes \C: E_0\in\Gr_k(\R^n)\}$ there exists a polynomial $P_{\Delta,E}\in \P(E)$ such that for all $w_1,\dots,w_k\in E$
	\begin{align}
		\label{eq:restrictionRealSubspaces}
		P(w_1,\dots,w_k)=\Delta^2(w_1,\dots,w_k) P_{\Delta,E}(w_k).
	\end{align}
	Then there exist polynomials $Q_{\Delta,E}\in \mathcal{P}(E)$ for all $E\in U_\Delta$ such that \begin{align*}
		P(w_1,\dots,w_k)=\Delta^2(w_1,\dots,w_k) Q_{\Delta,E}(w_k).
	\end{align*}
	for all $w_1,\dots,w_k\in E$.
\end{lemma}
\begin{proof}
	First note that this is vacuously true for $\Delta=0$ since $U_0=\emptyset$, so we may assume that $\Delta\ne 0$. We may further assume that $P$ is a homogeneous polynomial. Since the set
	\begin{align*}
		\{(w_1,\dots,w_k):w_1,\dots,w_k\in E\in U_\Delta\cap \{E_0\otimes \C: E_0\in \Gr_k(\R^n)\}\}
	\end{align*} is Zariski dense in $\Mat_{n,k}(\C)$, \eqref{eq:restrictionRealSubspaces} implies that $P$ is homogeneous of degree $(2,\dots,2,d)$, $d\ge 2$. Set $W_\Delta=\{w\in \Mat_{n,k}(\C):\Delta(w)\ne 0\}$. Consider the regular function on $W_\Delta$ given by
	\begin{align*}
		Q(w_1,\dots,w_k):=\frac{P(w_1,\dots,w_k)}{\Delta^2(w_1,\dots,w_k)}.
	\end{align*}
	We claim that the value of $Q(w_1,\dots,w_k)$ only depends on $w_k$ and the at most $k$-dimensional subspace $\mathrm{span}(w_1,\dots,w_k)$. To see this, it is sufficient to show that 
	\begin{align*}
		P(w_1,\dots,w_k)\Delta^2(w'_1,\dots,w'_{k-1},w_k)-P(w'_1,\dots,w'_{k-1},w_k)\Delta^2(w_1,\dots,w_k)
	\end{align*}
	vanishes if $\mathrm{span}(w'_1,\dots,w'_{k-1},w_k)=\mathrm{span}(w_1,\dots,w_{k-1},w_k)$. For $c=(c_{jl})_{jl}\in\C^{(k-1)\times k}$, $w_1,\dots,w_k\in\C^n$ set
	\begin{align*}
		\tilde{Q}(c,w_1,\dots,w_k):=&P(w_1,\dots,w_k)\Delta^2\left(\sum_{l=1}^kc_{1l}w_l,\dots,\sum_{l=1}^kc_{k-1,l}w_l,w_k\right)\\
		&-P\left(\sum_{l=1}^kc_{1l}w_l,\dots,\sum_{l=1}^kc_{k-1,l}w_l,w_k\right)\Delta^2(w_1,\dots,w_k).
	\end{align*}
	Then it is enough to to show that $\tilde{Q}$ vanishes identically. If $w_1,\dots,w_k\in \R^n$ are fixed, then the polynomial $c\mapsto \tilde{Q}(c,w_1,\dots,w_k)$ vanishes for $c\in \R^{(k-1)\times k}$ due to \eqref{eq:restrictionRealSubspaces} and thus vanishes identically on $\C^{(k-1)\times k}$. Thus for fixed $c\in \C^{(k-1)\times k}$, the polynomial $(w_1,\dots,w_k)\mapsto \tilde{Q}(c,w_1,\dots,w_k)$ vanishes on $(\R^n)^k$ and thus on $(\C^n)^k$, which shows the claim.\\
	Let us construct the polynomial $Q_{\Delta,E}$ for a given subspace $E\in U_\Delta$. For $w_k\in E\setminus\{0\}$ we may choose $w_1,\dots,w_{k-1}\in E$ such that $\Delta(w_1,\dots,w_k)\ne 0$ and define
	\begin{align*}
		Q_E(w_k):=Q(w_1,\dots,w_k)=\frac{P(w_1,\dots,w_k)}{\Delta^2(w_1,\dots,w_k)}.
	\end{align*}
	Then $Q_E$ is well defined and defines a regular function on $E\setminus \{0\}$. Now note that for $w_k\ne 0$,
	\begin{align*}
		Q_E(w_k)=|w_k|^{d-2}\frac{P\left(\frac{w_1}{|w_1|},\dots,\frac{w_k}{|w_k|}\right)}{\Delta^2\left(\frac{w_1}{|w_1|},\dots,\frac{w_k}{|w_k|}\right)}=|w_k|^{d-2}Q_E\left(\frac{w_k}{|w_k|}\right),
	\end{align*}
	because $P$ is homogeneous of degree $(2,\dots,2,d)$, so for $w_k\ne 0$ we obtain
	\begin{align*}
		|Q_E(w_k)|\le |w_k|^{d-2}\sup_{|z|=1}|Q_E(z)|.
	\end{align*}
	As $d\ge 2$, the holomorphic function $Q_E:E\setminus\{0\}\rightarrow\C$ is thus bounded on a neighborhood of $0$ and thus extends uniquely to a holomorphic function on $E$. Moreover, $Q_E$ is bounded by a polynomial of degree $d-2\ge 0$ and thus a polynomial of degree at most $d-2$ itself. Thus
	\begin{align*}
		P(w_1,\dots,w_k)=\Delta^2(w_1,\dots,w_k)Q_E(w_k)
	\end{align*}
	for all $w_1,\dots,w_k\in E$, which shows the claim with $Q_{\Delta,E}:=Q_E$.
\end{proof}
Next, we are going to show that polynomials satisfying the restriction property in Lemma~\ref{lemma:PolynomialFromComplexifiedSpaces} belong to a certain module generated by $\M^2_k$, the space spanned by quadratic products of $k$-minors. Here, we consider the space $\P(\Mat_{n,k}(\C))$ as a module over $\P(\C^n)$ by letting $p\in \P(\C^n)$ act on $P\in \P(\Mat_{n,k}(\C))$ by
\begin{align*}
	(p\cdot P)[w]=p(w_k)P(w)\quad\text{for}~w=(w_1,\dots,w_k)\in\Mat_{n,k}(\C).
\end{align*}
The proof requires some notions from the representation theory of $\GL(n,\C)$ and we refer to \cite{GoodmanWallachSymmetryrepresentationsinvariants2009} for a general background. We consider $\P(\Mat_{n,k}(\C))$ as a representation of $\GL(n,\C)$ with respect to the action given by
\begin{align*}
	(g\cdot P)[w_1,\dots,w_k]=P(g^Tw_1,\dots,g^Tw_k)
\end{align*}
for $g\in\GL(n,\C)$, $P\in\P(\Mat_{n,k}(\C))$, i.e. $g$ acts on the matrix in the argument by left multiplication with its transpose.
\begin{proposition}
	\label{proposition:CharacterizationModuleSquareMinors}
	Let $P\in \P(\Mat_{n,k}(\C))$ be a polynomial with the following property:
	For every $\Delta\in\M_k$ and every subspace $E\in U_\Delta$ there exists a polynomial $P_{\Delta,E}\in \P(E)$ such that
	\begin{align}
		\label{eq:PropositionRestrictionMinors}
		P(w_1,\dots,w_k)=\Delta^2(w_1,\dots,w_k) P_{\Delta,E}(w_k)\quad\text{for all }w_1,\dots,w_k\in E.
	\end{align}
	Then $P$ is contained in the $\P(\C^n)$-submodule generated by $\M^2_k$.
\end{proposition}
\begin{proof}
	Note that the polynomials with these properties form a $\GL(n,\C)$-invariant subspace of $\P(\Mat_{n,k}(\C))$ with respect to the action defined above. Similarly, the module generated by $\M^2_k$ is a $\GL(n,\C)$-invariant subspace. As the space of complex polynomials on $\Mat_{n,k}(\C)$ is a regular representation of $\GL(n,\C)$, it decomposes into a direct sum of irreducible representations. Any such representation is generated by a unique highest weight vector, so it is sufficient to show that any highest weight vector of $\P(\Mat_{n,k}(\C))$ satisfying \eqref{eq:PropositionRestrictionMinors} is contained in the module generated by $\M^2_k$. We will show that any such vector is a product of the square of the $k$th principal minor $\Delta_k$ with a polynomial in $\P(\C^n)$.\\
	
	Let $P_\lambda$ be a highest weight vector with weight $\lambda\in\mathbb{Z}^n$ and the properties above. Then $P_\lambda$ is invariant under the group $N_n^+$ of upper triangular matrices with $1$ on the diagonal and for any diagonal matrix $h=\mathrm{diag}(h_1,\dots,h_n)$ we have
	\begin{align*}
		h\cdot P_\lambda=h^\lambda P_\lambda.
	\end{align*}
	Let us consider the restriction of $P_\lambda$ to the dense open set of all elements $w\in \Mat_{n,k}(\C)$ with $\Delta_i(w)\ne 0$ for $1\le i\le k$, where $\Delta_i$ denotes the $i$th principal minor. Using the Gauss decomposition (compare e.g. \cite[Chapter~11.6]{GoodmanWallachSymmetryrepresentationsinvariants2009}), we can write any such matrix as $w=\nu  h u$, where $\nu\in N^-_n$ is a lower triangular $(n\times n)$-matrix with $1$ on the diagonal, $h\in D_{n,k}$ is a diagonal $(n\times k)$-matrix, and $u\in N^+_k$. Thus
	\begin{align*}
		P_\lambda(w)= P_\lambda((\nu^T)^Thu)=P_\lambda(hu),
	\end{align*}
	since $\nu^T\in N^+_n$.	In particular, $P_\lambda$ is uniquely determined by its restriction to upper triangular matrices. If $w=\nu U$ for an upper triangular matrix $U=(u_1,\dots,u_k)\in \Mat_{n,k}(\C)$ and $\nu\in N^-_n$, then $u_k\in \C^k\times\{0\}$ belongs to a $k$-dimensional subspace, and we obtain
	\begin{align*}
		P_\lambda(w)=P_\lambda((\nu^T)^T\cdot U)=P_\lambda(U)=\Delta^2_k(U)Q(u_k)
	\end{align*} 
	for a unique polynomial $Q:=P_{\Delta_k,\C^k\times\{0\}}$ on $\C^k\times\{0\} $ by assumption. In particular, since $\Delta_k$ is invariant under the operation of $N_n^+$,
	\begin{align*}
		P_\lambda(w)=\Delta_k^2(w)Q(u_k),
	\end{align*}
	and so $P_\lambda$ is uniquely determined by the polynomial $Q$. Since we assume that $P_\lambda$ is nontrivial, $Q$ does not vanish identically. Let us examine the polynomial $Q$ on $\C^k\times\{0\}$ as well as the weight $\lambda$. First note that for any $u_k\in \C^k\times\{0\}$ with $(u_k)_k\ne 0$ there exist an upper triangular matrix $w_0\in \Mat_{n,k}(\C)$ with $\Delta_i(w_0)\ne0$ for $1\le i\le k$ such that $u_k$ is the last column of $w_0$. Given such a matrix $w_0\in\Mat_{n,k}(\C)$ and a diagonal matrix $h\in D_{n,n}$, we thus have
	\begin{align*}
		h^{\lambda} P_\lambda(w_0)=P_\lambda(h^T\cdot w_0)=h_1^2\dots h_k^2 \Delta^2_k(w_0)Q(\mathrm{diag}(h_1,\dots,h_k)u_k),
	\end{align*}
	where we consider $u_k$ as an element of $\C^k\cong \C^k\times \{0\}$. If $u_k\in \C^k\times\{0\}\cong\C^k$ belongs to the complement of the zero set of $Q$, then the right hand side of this equation does not depend on $h_{k+1},\dots,h_n$, so  $\lambda_{k+1}=\dots=\lambda_n=0$ since this set is dense in $\C^k$ and $\Delta_k(w_0)\ne 0$. Moreover, this implies that $Q$ is a weight vector with weight $(\lambda_1,\dots,\lambda_k)-(2,\dots,2)\in \mathbb{Z}^k$ of the $\GL(k,\C)$-representation $\P(\C^k)$ with action 
	\begin{align*}
		g\cdot P(z)=P(g^Tz)\quad \text{for}~g\in\GL(k,\C),~P\in \P(\C^k).
	\end{align*} If $\nu\in N^+_k\subset \GL(k,\R)$ is an upper triangular unidiagonal matrix, then, since $P_\lambda$ and $\Delta_k$ are invariant under $N^+_n$,
	\begin{align*}
		\Delta_k^2(w_0)Q(u_k)=P_\lambda(w_0)=P_\lambda\left(\begin{pmatrix}
			\nu &0\\
			0 & Id_{n-k}
		\end{pmatrix}^Tw_0\right)=\Delta^2_k(w_0)Q(\nu^T u_k)
	\end{align*}
	In other words, $Q$ is invariant under $N^+_k$, so $Q$ is a sum of highest weight vectors with weight $(\lambda_1-2,\dots,\lambda_k-2)$. As $\mathcal{P}(\C^k)$ is a multiplicity free representation of $\GL(k,\C)$ and all highest weights are of the form $(d',0,\dots,0)$ for $d'\ge0$, we thus obtain $c\in \C$ such that  $Q(z)=cz_1^{d'}$ for $z\in \C^k$. Comparing the degrees of homogeneity, $d'=d-2$. In particular, $(\lambda_1,\dots,\lambda_k)=(d,2,\dots,2)$, and 
	\begin{align*}
		P_\lambda(w)=c\cdot\Delta_k^2(w) w_{1,k}^{d-2}
	\end{align*}
	which belongs to the $\P(\C^n)$-submodule spanned by $\M^2_k$. 
\end{proof}
For later use, note that the proof provides the following characterization of the highest weight vectors.
\begin{corollary}
	\label{corollary:lowestWeightvectorsquaredMinors}
		Let $\GL(n,\C)$ act on $\P(\Mat_{n,k}(\C))$ by 
		\begin{align*}
			(g\cdot P)[w_1,\dots,w_k]= P(g^Tw_1,\dots,g^Tw_k).
		\end{align*} Then any $\GL(n,\C)$-highest weight vector of $\P(\C^n)\M^2_k$ is of the form 
		\begin{align*}
			c\Delta_k^2(w) w_{1,k}^d
		\end{align*} for some $d\in\mathbb{N}$, $c\in \C\setminus\{0\}$.
\end{corollary}

\subsection{Some modules of entire functions generated by polynomials}
	\label{section:DivisionAlg}
	In this section we will consider certain rather specific modules of holomorphic functions generated by a finite set of polynomials and the problem of decomposing a given element into a $\mathcal{O}_{\C^n}$-linear combination of suitable generators. Since we are mostly interested in modules generated by homogeneous polynomials, the constructions are based around the power series expansion of a given holomorphic function in $0$, which we will just call its power series expansion for short, and will refer to the terms in the power series expansion as the homogeneous terms of the power series expansion (of a given order).\\
	Since we identify $\Mat_{n,k}(\C)\cong (\C^n)^k$, we will use the following notation: For $\alpha\in\mathbb{N}^n=\mathbb{N}^{n\times 1}$, $|\alpha|=\alpha_1+\dots+\alpha_n$ denotes the usual degree. Abusing notation, given a multiindex $\alpha=(\alpha_1,\dots,\alpha_k)\in\mathbb{N}^{n\times k}$, we will set $|\alpha|:=(|\alpha_1|,\dots,|\alpha_k|)\in\mathbb{N}^k$, where $|\alpha_j|$ denotes the usual degree of $\alpha_j\in\mathbb{N}^n$. \\ 
	
	For a multiindex $\alpha\in \mathbb{N}^m$, we denote by $z^\alpha:=z_1^{\alpha_1}\dots z_m^{\alpha_m}$ the corresponding monomial on $\C^m$. We start with the following estimate involving division with remainders by monomials.
	\begin{lemma}
		\label{lemma:monomialDivisionEstimate}
		For $\alpha\in \mathbb{N}^n$ and $F\in\mathcal{O}_{\C^n}$ there exists a unique function $g_\alpha\in \mathcal{O}_{\C^n}$ such that no term of the power series expansion of 
		\begin{align*}
			F-g_\alpha z^\alpha 
		\end{align*}
		contains a monomial divisible by $z^\alpha$. Moreover, this function is given by
		\begin{align}
			\label{equation:defG_alpha}
			g_\alpha(z)=\frac{1}{(2\pi i)^n}\int_{\mathcal{C}_1}\dots\int_{\mathcal{C}_n}\frac{F(\xi)}{\xi^\alpha}\frac{d\xi_1\dots d\xi_n}{(\xi_1-z_1)\dots(\xi_n-z_n)}.	
		\end{align}
		where $\mathcal{C}_j$ is any positively oriented simple closed curve in $\C$ such that the interior of the domain enclosed by $C_j$ contains $0$ and $z_j$, and satisfies for any $\delta>0$
		\begin{align*}
			|g_\alpha(z)|\le \frac{2^n}{\delta^{|\alpha|}}\left(\prod_{j=1}^n\frac{|z_j|+\delta}{\delta}\right)\sup_{\substack{|\Re(\xi_j)| \le |\Re(z_j)|+\delta,\\|\Im(\xi_j)|\le |\Im(z_j)|+\delta,\\1\le j\le n}}|F(\xi)|.
		\end{align*}
	\end{lemma}
	\begin{proof}
		The uniqueness is a standard argument: If $g_\alpha,g_\alpha'\in\mathcal{O}_{\C^n}$ are two functions with this properties, then 
		\begin{align*}
			(g_\alpha-g_\alpha')z^\alpha=(F-g_\alpha'z^\alpha)-(F-g_\alpha z^\alpha)
		\end{align*}
		is a holomorphic function such that no homogeneous term in its power series expansion contains a monomial divisible by $z^\alpha$. This is only possible if $g_\alpha=g_\alpha'$.\\
		Let us establish that a function with the required properties exists. First, note that the value of the integral \eqref{equation:defG_alpha} does not depend on the choice of curves $\mathcal{C}_j$ containing the poles $0$ and $z_j$ of the integrand in the interior of the domain enclosed by them. In particular, we may define $g_\alpha$ by \eqref{equation:defG_alpha}. Then the representation \eqref{equation:defG_alpha} holds for all elements in a neighborhood of a given $z\in \C^n$, as all poles of the integrand are contained in the domains enclosed by the curves $\mathcal{C}_1,\dots,\mathcal{C}_n$, and defines a holomorphic function. Thus $g_\alpha$ is well defined.\\
		In order to obtain the desired estimate, fix $z\in\C^n$ and let $\mathcal{C}_j$ denote the curve in $\C$ given by the positively oriented boundary of $\{\xi\in\C: |\Re(\xi)| \le |\Re(z_j)|+\delta,|\Im(\xi)|\le |\Im(z_j)|+\delta\}$ for some $\delta>0$. Since $|z_j-\xi_j|\ge \delta$, $|\xi_j|\ge \delta$ for $\xi_j\in\mathcal{C}_j$,
		\begin{align*}
			|g_\alpha(z)|\le \frac{1}{\delta^{|\alpha|+n}}\left(\prod_{j=1}^n\frac{\mathrm{length}(\mathcal{C}_j)}{2\pi }\right)\sup_{\substack{|\Re(\xi_j)| \le |\Re(z_j)|+\delta,\\|\Im(\xi_j)|\le |\Im(z_j)|+\delta,\\1\le j\le n}}|F(\xi)|,	
		\end{align*}
		which implies the desired estimate. In order to see that no term in the power series expansion of $F-g_\alpha z^\alpha$ contains a monomial that is divisible by $z^\alpha$, note that
		the coefficient  of $z^{\alpha+\beta}$ of $F-g_\alpha z^\alpha $ is given by a multiple of $\partial^{\alpha+\beta}[f-g_\alpha z^\alpha](0)$.
		We have
		\begin{align*}
			\xi^\alpha=[(\xi-z)+z]^\alpha=\sum_{\beta\le\alpha} \binom{\alpha}{\beta} (\xi-z)^{\alpha-\beta} z^{\beta}.
		\end{align*}
		with $\binom{\alpha}{\beta}=\prod_{j=1}^n \binom{\alpha_j}{\beta_j}$. Cauchy's integral formula thus implies		
		\begin{align*}
			&F(z)-g_\alpha(z)z^\alpha=\frac{1}{(2\pi i)^n}\int_{\mathcal{C}_1}\dots\int_{\mathcal{C}_n}F(\xi)\frac{d\xi_1\dots d\xi_n}{(\xi_1-z_1)\dots(\xi_n-z_n)}-g_\alpha(z)z^\alpha\\
			&\qquad=\sum_{\beta<\alpha}\binom{\alpha}{\beta}\frac{1}{(2\pi i)^n}\int_{\mathcal{C}_1}\dots\int_{\mathcal{C}_n}\frac{F(\xi)}{\xi^\alpha}(\xi-z)^{\alpha-\beta}\frac{d\xi_1\dots d\xi_n}{(\xi_1-z_1)\dots(\xi_n-z_n)} z^{\beta}.
		\end{align*}
		Note that since $\alpha-\beta$ has at least one nontrivial component, the integrands of the terms in the sum have the property that the singularity $\xi_j=z_j$ is canceled for at least one $1\le j\le n$. In particular, they are polynomials in at least one of the variables $z_j$ of degree strictly less than $\alpha_j$. Thus for any $\beta'\ge 0$,
		\begin{align*}
			\partial^{\alpha+\beta'}[F-g_\alpha z^\alpha](0)=0,
		\end{align*}
		so no term in the power series expansion of $F-g_\alpha z^\alpha$ contains a nontrivial multiple of a monomial divisible by $z^\alpha$.
	\end{proof}

	We will be interested in modules over holomorphic functions that only depend on a selection of the variables. For a splitting $\C^{m+n}=\C^m\times \C^n$, $z=(z_1,z_2)\in \C^m\times\C^n$, and $\delta>0$, define
	\begin{align}
		\label{equation:polydiskSimple}
		\begin{split}
			D_\delta(z_1,z_2):=\{(w_1,w_2)\in \C^m\times \C^n:&|\Re(w_j)|_1\le |\Re(z_j)|_1+\delta,\\
			&|\Im(w_j)|_1\le |\Im(z_j)|_1+\delta\},
		\end{split}
	\end{align}
	where $|\cdot|_1$ denotes the $l^1$-norm. We will also say that the degree of $P\in\P(\C^m\times \C^n)$ is bounded by $(h_1,h_2)\in\mathbb{N}^2$ if all monomials of $P(w_1,w_2)$ are of smaller degree than $h_i$ in $w_i$ for $i=1,2$.

	\begin{corollary}
		\label{corollary:DivisionMonomialMultiHom}
		Let $F\in \mathcal{O}_{\C^m\times\C^n}$ and $\alpha\in \mathbb{N}^m$ be an index such that no term of the power series expansion of $F$ contains a monomial that is divisible by $w_1^{\alpha'}$ for some $\alpha'\ne\alpha$ with $\alpha'\ge\alpha$. Then the following holds:
		\begin{enumerate}
			\item For every $\beta\in \mathbb{N}^n$ there exists a unique function $g_{(\alpha,\beta)}\in \mathcal{O}_{\C^n}$ such that no term of the power series expansion of 
			\begin{align*}
				F(w_1,w_2)-g_{(\alpha,\beta)}(w_2) w_1^\alpha w_2^\beta
			\end{align*}
			contains a monomial divisible by $w_1^{\alpha} w_2^\beta$.
			\item The function $g_{(\alpha,\beta)}$ satisfies
				\begin{align*}
					|g_{(\alpha,\beta)}(z)|\le \frac{\max(n,m)^{|\alpha|+|\beta|+n}2^{n+m}}{\delta^{|\alpha|+|\beta|+n}}\left(|z|_1+\delta\right)^{n}\sup_{\zeta\in D_\delta(0,z)}|F(\zeta)|.
				\end{align*}
				for every $\delta>0$.
		\end{enumerate}
	\end{corollary}
	\begin{proof}
		We apply Lemma~\ref{lemma:monomialDivisionEstimate} to the index $(\alpha,\beta)\in \mathbb{N}^m\times\mathbb{N}^n$ and obtain a unique function $g_{(\alpha,\beta)}\in \mathcal{O}_{\C^m\times\C^n}$ such that no term of the power series expansion of $F(w)-g_{(\alpha,\beta)}(w) w_1^\alpha w_2^\beta$ contains a monomial divisible by $w_1^{\alpha} w_2^\beta$. In order to obtain (1), we have to show that $g_{(\alpha,\beta)}$ is independent of $w_1$.
		If this is not the case, then there exists $\gamma\in \mathbb{N}^m$, $\gamma\ne 0$, such that some monomial in the power series expansion of $g_{(\alpha,\beta)}$ is a nontrivial multiple of $w_1^\gamma$. By assumption, no term in the power series expansion of $F$ contains a monomial divisible by $w_1^{\alpha+\gamma}$. But then the power series expansion of
		\begin{align*}
			F(w_1,w_2)-g_{(\alpha,\beta)}(w_1,w_2) w_1^\alpha w_2^\beta
		\end{align*}
		contains a monomial divisible by $w_1^{\alpha+\gamma}w_2^\beta$, which is a contradiction. Thus $g_{(\alpha,\beta)}$ does not depend on $w_1$.\\
		Using the estimate in Lemma~\ref{lemma:monomialDivisionEstimate}, we thus obtain
		\begin{align*}
			|g_{(\alpha,\beta)}(0,z)|\le \frac{2^{n+m}}{\delta^{|\alpha|+|\beta|}}\left(\frac{|z|_1+\delta}{\delta}\right)^n\sup_{\xi\in D_{\max(n,m)\delta}(0,z)}|F(\xi)|.
		\end{align*}
		Replacing $\delta$ by $\frac{\delta}{\max(n,m)}$ yields the desired inequality.
	\end{proof}

	We are going to apply these estimates to obtain suitable presentations of holomorphic functions belonging to $\mathcal{O}_{\C^n}$-submodules of $\mathcal{O}_{\C^m\times\C^n}$ generated by a finite family of polynomials satisfying rather specific properties. The argument relies on properties of Gröbner bases for an associated $\P(\C^n)$-submodule of polynomials. We refer to \cite[Chapter 15]{EisenbudCommutativealgebra1995} for a general background on Gröbner bases and only collect the necessary terminology we will use.\\
	
	We consider $\P(\C^m\times\C^n)$ as a module over $\P(\C^n)$ using the obvious inclusion $\P(\C^n)\subset \P(\C^m\times \C^n)$. Note that if $M\subset \P(\C^m\times \C^n)$ is a finitely generated $\P(\C^n)$-submodule, then $M$ is contained in the finite free $\P(\C^n)$-module $F$ generated by the monomials in $\P(\C^m)$ that divide a monomial in a generating set of $M$.\\
	In order to be consistent with the notation for matrices, we consider $\C^m\times \C^n$ as two column vectors (of potentially different length), and denote the components of $w_1\in \C^m$ by $w_1=(w_{1,1},\dots,w_{m,1})$ and of $w_2\in \C^n$ by $w_2=(w_{1,2},\dots,w_{n,2})$. We fix a monomial order on this free module $F$ by restricting the lexicographic order on $\P(\C^m\times \C^n)$ induced from the variable order
	\begin{align*}
		w_{1,1}>\dots >w_{m,1}>w_{1,2}>\dots >w_{n,2}
	\end{align*}
	to the monomials in $F$. Given an element $P\in F$, we denote by $\mathrm{in}(P)$ its initial term with respect to this order, i.e. the largest monomial occurring in $P$ with a non-zero coefficient. Similarly, if $M\subset F$ is a submodule, we denote by $\mathrm{in}(M)$ the corresponding initial module, i.e. the monomial submodule generated by $\mathrm{in}(P)$ for $P\in M$. Recall that a set $P_1,\dots,P_N$ that generates $M$ as a $\P(\C^n)$-module is called a Gröbner basis if the initial terms $\mathrm{in}(P_1),\dots,\mathrm{in}(P_N)$ generate the initial module $\mathrm{in}(M)$. If the elements $P_1,\dots,P_N$ are in addition ordered with respect to the chosen monomial order, we will say that this Gröbner basis is ordered.\\
	Note that every finitely generated $\P(\C^n)$-submodule of $\P(\C^n\times \C^m)$ has a Gröbner basis, however, this basis is in general far from unique.\\
	
	In the cases we will be interested in, the relevant $\P(\C^n)$-submodules of $\P(\C^m\times\C^n)$ will be generated by homogeneous elements, where we call $P\in\P(\C^m\times\C^n)$ homogeneous of degree $(h_1,h_2)\in\mathbb{N}^2$ if $P(t_1w_1,t_2w_2)=t_1^{h_1}t_2^{h_2}P(w_1,w_2)$ for $t_1,t_2\in \C$, $(w_1,w_2)\in\C^{m}\times \C^n$.
	\begin{theorem}
		\label{theorem:DivisionAlg}
		Let $M\subset \mathcal{P}(\C^m\times\C^n)$ be a finitely generated homogeneous module over $\P(\C^n)$ and $P_1,\dots,P_N\in M$ an ordered Gröbner basis consisting of homogeneous elements. Assume that this basis has the property that no multiple of $P_i-\mathrm{in}(P_i)$ contains a monomial that is divisible by the initial term of $P_i$ for every $1\le i\le N$. Then the following holds:
		\begin{enumerate}
			\item If $F\in\mathcal{O}_{\C^m\times\C^n}$ is a function such that every homogeneous term in its power series expansion belongs to $M$, then there exist unique functions $g_j\in \mathcal{O}_{\C^n}$ with the following property: No monomial of the homogeneous terms in the power series expansion of $F-\sum_{i=1}^{j}g_iP_i$ is divisible by the initial term of $P_i$ for $1\le i\le j$.
			\item For every $\delta>0$ there exists a constant $C_{\delta,P_1,\dots,P_N}$ depending on $\delta$ and $P_1,\dots,P_N$ only, such that the functions in (1) satisfy
				\begin{align*}
					|g_j(z)|\le& C_{n,\delta,P_1,\dots,P_N} (1+|z|_1+\delta)^{N(h_2+n)}\sup_{\zeta\in 	D_{\delta}(0,z)}|F(\zeta)|,
				\end{align*}
			where $(h_1,h_2)\in\mathbb{N}^2$ is a bound on the degree of $P_j$ for $1\le j\le N$. 
		\end{enumerate} 
		In particular, $F=\sum_{i=1}^Ng_iP_i$ belongs to the $\O_{\C^n}$-module generated by the polynomials $P_1,\dots,P_N$.
	\end{theorem}
	\begin{remark}
		\label{remark:specialPropertyGroebner}
		The property stated in (1) is rather restrictive. For example, for $m=0$, only Gröbner bases consisting of monomials satisfies this property.
	\end{remark}
	\begin{proof}[Proof of Theorem \ref{theorem:DivisionAlg}]
		First note that if we find functions $g_j\in\mathcal{O}_{\C^n}$ such that no term in the power series expansion of $\tilde{F}:=F-\sum_{j=1}^Ng_j P_j$ contains a monomial that is divisible by the initial terms of $P_1,\dots,P_N$, then its power series expansion vanishes identically, since the homogeneous terms belong to $M$ and $P_1,\dots,P_N$ is a Gröbner basis for this module. Thus $\tilde{F}$ has to vanish identically and $F= \sum_{i=1}^Ng_iP_i$ belongs to $\O_{\C^n}M$. In particular, the last claim follows from (1).\\
		
		We will show (1) and (2) by induction on the $N$ number of elements in the chosen Gröbner basis for the module $M$, the case $N=0$ being vacuously true.\\
		
		In order to simplify some notation in the proof, we define a norm on the space of polynomials of degree at most $h=(h_1,h_2)\in\mathbb{N}^2$ in $\mathcal{P}(\C^m\times\C^n)$ by
		\begin{align*}
			\|P\|_{h}:=\sup_{(w_1,w_2)\in\C^m\times\C^n} \frac{|P(w)|}{(1+|w_1|_1)^{h_1}(1+|w_2|_1)^{h_2}}.
		\end{align*}
		Assume that the claim holds for all such $\P(\C^n)$-modules with a Gröbner basis of at most $N-1$ elements.\\
		Let $M\subset\P(\C^m\times\C^n)$ be a homogeneous $\P(\C^n)$-module generated in degree at most $(h_1,h_2)$ and $P_1,\dots,P_N$ an ordered Gröbner basis with normalized initial terms satisfying the condition in (1). Omitting terms, we may assume that the initial terms are mutually indivisible, i.e. that the Gröbner basis is minimal. Let $(\alpha,\beta)\in\mathbb{N}^{m}\times \mathbb{N}^n$ be the index corresponding to the monomial $\mathrm{in}(P_1)$. Since $\mathrm{in} (P_1)$ is the maximal monomial in the chosen Gröbner basis, no term in the power series expansion of $F$ is divisible by a monomial $w_1^{\alpha'}$ for $\alpha'>\alpha$. Let $g_1\in\mathcal{O}_{\C^n}$ be the unique function in Corollary~\ref{corollary:DivisionMonomialMultiHom} for the index $(\alpha,\beta)$. Then there exists a constant $C_{\delta,1}$ depending on $\delta>0$, $n$, and $P_1$ only such that
		\begin{align}
			\label{equation:estimate_g1_simplified}
			|g_1(z)|\le C_{\delta,1}(1+|z|_1+\delta)^n\sup_{\zeta\in D_\delta(0,z)}|F(\zeta)|,
		\end{align}
		so $g_1$ satisfies the desired estimates. Moreover, we have
		\begin{align*}
			F-g_1P_1=[F-g_1\mathrm{in}(P_1)]+g_1[\mathrm{in}(P_1)-P_1].
		\end{align*}
		By Corollary~\ref{corollary:DivisionMonomialMultiHom}, the first term on the right hand side has the property that no monomial in the homogeneous components in its power series expansion are divisible by $\mathrm{in}(P_1)$. Since we assume that no multiple of $P_1-\mathrm{in}(P_1)$ contains a monomial divisible by $\mathrm{in}(P_1)$, the second term has the same property. In particular, no monomial of the homogeneous terms in the power series expansion of $\tilde{F}:=F-g_1P_1$ is divisible by the initial term of $P_1$, which shows that (1) holds for $j=1$. Moreover, since $P_1,\dots,P_N$ form a Gröbner basis for $M$, $\tilde{F}$ has the property that every homogeneous term in its power series expansion belongs to the homogeneous module $M'$ generated by $P_2,\dots,P_N$. Since $P_2,\dots,P_N$ form an ordered Gröbner basis for $M'$ such that no multiple of $P_i-\mathrm{in}(P_i)$ contains a monomial divisible by $\mathrm{in}(P_i)$ for $2\le i\le N$, we may apply the induction assumption to $M'$ with the Gröbner basis $P_2,\dots,P_N$ and $\tilde{F}$. Thus there exist constants $C_{\delta,2}$ depending on $\delta>0$, $n$, and $P_2,\dots,P_N$ only as well as unique functions $g_2,\dots,g_N\in\mathcal{O}_{\C^n}$ such that for any $\delta>0$
		\begin{align}
			\label{eq:estimate_Induction_gj}
			\begin{split}
				|g_j(z)|\le&C_{\delta,2}(1+|z|_1+\delta)^{(N-1)(h_2+n)}\sup_{\zeta\in D_{\delta}(0,z)}|\tilde{F}(\zeta)|
			\end{split}
		\end{align}
		and such that no monomial of the homogeneous terms in the power series expansion of $\tilde{F}-\sum_{i=2}^jg_iP_i$ is a multiple of $\mathrm{in}(P_2),\dots,\mathrm{in}(P_j)$. Since the Gröbner basis is ordered and the initial terms are mutually indivisible by assumption, no monomial in the homogeneous terms in the power series expansion of this function is divisible by $\mathrm{in}(P_1)$, so in total we see that $F-\sum_{i=1}^jg_iP_i=\tilde{F}-\sum_{i=2}^jg_iP_i$ has the same property, which completes the induction step for (1). It remains to see that $g_j$ satisfies the desired estimates for $2\le j\le N$. Using \eqref{equation:estimate_g1_simplified}, we have
		\begin{align*}
			|\tilde{F}(w)|=&|F(w)-g_1(w_2)P_1(w)|\\
			\le&|F(w)|+|g_1(w_2)|\cdot \|P_1\|_{h} (1+|w_1|_1)^{h_1}(1+|w_2|_1)^{h_2}\\
			\le&\left(1+C_{\delta,1}\|P_1\|_h(1+|w_1|_1)^{h_1}(1+|w_2|_1+\delta)^{h_2+n}\right)\sup_{\zeta\in D_\delta(w)}|F(\zeta)|.
		\end{align*}
		In particular,
		\begin{align*}
			\sup_{\zeta\in D_\delta(0,z)}|\tilde{F}(\zeta)|\le & \left(1+C_{\delta,1}\|P_1\|_h(1+\delta)^{h_1}\right)(1+|w_2|_1+2\delta)^{h_2+n}
			\sup_{\zeta\in D_{2\delta}(0,z)}|F(\zeta)|
		\end{align*}
		for any $\delta>0$, which together with \eqref{eq:estimate_Induction_gj} implies
		\begin{align*}
			|g_j(z)|\le& C_{\delta,2}\left(1+C_{\delta,1}\|P_1\|_h(1+\delta)^{h_1}\right)(1+|z|_1+2\delta)^{N(h_2+n)}\sup_{\zeta\in D_{2\delta}(0,z)}|F(\zeta)|
		\end{align*}
		for any $\delta>0$. Replacing $\delta$ by $\delta/2$ yields the desired estimate.
	\end{proof}
	\begin{corollary}
		In the situation of Theorem~\ref{theorem:DivisionAlg}, there exist constants $C_\delta$ depending on $\delta$, $n$, and $P_1,\dots,P_N$ only, such that the functions $g_j\in\mathcal{O}_{\C^n}$ satisfy
		\begin{align*}
			|g_j(z)|\le& C_{\delta}(1+|z|)^{N(h_2+n)}\sup_{\zeta\in D_{\delta}(0,z)}|F(\zeta)|
		\end{align*}
		for every $\delta>0$.
	\end{corollary}
	\begin{proof}
		Use the estimate $(1+|z|_1+\delta)\le (1+\delta)(1+|z|_1)$ and $|z|_1\le \sqrt{n}|z|$.
	\end{proof}
	Recall that we identify $(\C^n)^k$ with the space $\Mat_{n,k}(\C)$ of $(n\times k)$-matrices. Similar to the previous section, we consider $\mathcal{O}_{\Mat_{n,k}(\C)}$ as a module over $\mathcal{O}_{\C^n}$ by letting $g\in \mathcal{O}_{\C^n}$ act on $F\in \mathcal{O}_{\Mat_{n,k}(\C)}$ by
	\begin{align*}
		(g\cdot F)[w]=g(w_k)F(w)\quad\text{for}~w=(w_1,\dots,w_k)\in\Mat_{n,k}(\C).
	\end{align*}
	We are going to apply the results of this section to the module generated by $\M^2_k$, the space spanned by quadratic products of $k$-minors. By \cite[Corollary 6.19]{KnoerrMongeAmpereoperators2024}, this is a vector space of dimension $N_{n,k}=\dim\MAVal_k(\R^n)=\binom{n}{k}^2-\binom{n}{k-1}\binom{n}{n-k-1}$ (compare also Lemma~\ref{lemma:QBijectiveMA} below).\\	
	In order to apply Theorem \ref{theorem:DivisionAlg}, we require a suitable Gröbner basis for this submodule. Recall that we consider $w=(w_1,\dots,w_k)$ as a matrix of column vectors and denote the coordinates of $w_j\in\C^n$ by $(w_{1,j},\dots,w_{n,j})$.\\	
	For an ordered $k$-tuple $I=(i_1,\dots,i_k)\in\{1,\dots,n\}^k$, $i_j<i_{j+1}$, we denote by
	\begin{align*}
		[I](w)=\det\nolimits_k\begin{pmatrix}
			w_{i_1,1} & \dots &w_{i_1,k}\\
			\vdots & &\vdots\\
			w_{i_k,1} & \dots &w_{i_k,k}
		\end{pmatrix}
	\end{align*}
	the $(k\times k)$-minor of the matrix $w$ obtained from the rows $i_1,\dots,i_k$. To any ordered $k$-tuple $I$ we associate its index $\alpha(I)\in \mathbb{N}^{n\times k}$, $\alpha(I)=(e_{i_1},\dots,e_{i_k})$, where $e_j\in\mathbb{N}^n$ denotes the standard column vectors.
	\begin{lemma}
		\label{lemma:productsMinorsGroebnerBasis}
		The set consisting of $[I,I']:=[I][I']$ for ordered $k$-tuples $I,I'$ with $I\ge I'$ is a Gröbner basis for the $\P(\C^n)$-module generated by $\M^2_k$ with respect to the lexicographic term order induced from the variable order
		\begin{align*}
			&w_{1,1}>w_{2,1}>\dots>w_{n,1}\\
			>&w_{1,2}>w_{2,2}>\dots>w_{n,2}>\dots\\
			>&w_{1,k}>w_{2,k}>\dots>w_{n,k}.
		\end{align*}
		Moreover, no multiple of $[I,I']-\mathrm{in}([I,I'])$ contains a monomial that is divisible by $\mathrm{in}([\tilde{I},\tilde{I'}])$ for any ordered $k$-tuples $\tilde{I},\tilde{I}'$ with $\tilde{I}\ge \tilde{I}'$.
	\end{lemma}
	\begin{proof}
		Obviously, the polynomials $[I,I']$ span the vector space $\M^2_k$, so they generate $\P(\C^n)\M^2_k$. 	In order to show that these products form a Gröbner basis for $\P(\C^n)\M^2_k$, we use Buchberger's Criterion.\\
		Note that the initial term of $[I,I']$ is 
		\begin{align*}
			\left(\prod_{j=1}^kw_{i_j,j}\right)\left(\prod_{j=1}^kw_{i'_j,j}\right)=w^{\alpha(I)}w^{\alpha(I')}.
		\end{align*}
		Let $J,J'$ be two additional ordered $k$-tuples with $J\ge J'$. Note that two different monomials $w^{\alpha(I)+\alpha(I')}$ and $w^{\alpha(J)+\alpha(J')}$ are only related by multiplication with monomials in $\P(\C^n)$ if $\alpha(I')_j=\alpha(J')_j$ and $\alpha(I)_j=\alpha(J)_j$ for $1\le j\le k-1$. Then $\alpha(I')_k+\alpha(I)_k$ and $\alpha(J')_k+\alpha(J)_k$ differ by either one or two entries. If they differ by one entry, then without loss of generality $\alpha(I)_k+\alpha(I')_k=e_i+e_j$ and $\alpha(J)_k'+\alpha(J')_k'=e_i+e_l$ for $j>l$ and so
		\begin{align}
			\notag
			&w_{l,k}[I,I']-w_{j,k}[J,J']\\
			&\label{equation:CasesKMinors}\qquad=\begin{cases}
				[I](w_{\alpha(J')_k,k}[I']-w_{\alpha(I')_k,k}[J']),& i\ge j>l,~\text{i.e. for}~I=J,\\
				[I'](w_{\alpha(J')_k,k}[I]-w_{\alpha(I)_k,k}[J']),&j>i> l,~\text{i.e. for}~I'=J,\\
				[I'](w_{\alpha(J)_k,k}[I]-w_{\alpha(I)_k,k}[J]),&j>l\ge i,~\text{i.e. for}~I'=J'.
			\end{cases}
		\end{align}
		Since the elements $[\tilde{I}]$ for ordered index sets $\tilde{I}$ form a Gröbner basis for the ideal of $\P(\Mat_{n,k}(\C))$ generated by $k$-minors with respect to the lexicographic extension of any variable order respecting the columns of rows (see for example \cite[Theorem 1]{SturmfelsGrobnerbasesStanley1990} and the remark after its proof), the second factor in these expressions can be written as a sum of products of polynomials $P_{\tilde{I}}\in\mathcal{P}(\Mat_{n,k}(\C))$ and the $k$-minors $[\tilde{I}]$, where the initial term of $P_{\tilde{I}} [\tilde{I}]$ is smaller than the initial term of the second factor in \eqref{equation:CasesKMinors}. Comparing the degrees of homogeneity, we see that $P_{\tilde{I}}$ is linear and only depends on $w_k$. Thus $w_{l,k}[I,I']-w_{j,k}[J,J']$ can be written as a $\P(\C^n)$-linear combination of $[\tilde{I},\tilde{I}']$ with strictly smaller initial terms.\\
		If $\alpha(I')_k+\alpha(I)_k$ and $\alpha(J')_k+\alpha(J)_k$ differ by two entries, then we have to consider
		\begin{align*}
			&w_{\alpha(J')_k,k}w_{\alpha(J)_k,k}[I,I']-w_{\alpha(I')_k,k}w_{\alpha(I)_k,k}[J,J']\\
			&\quad=w_{\alpha(J)_k,k}[I]\left(w_{\alpha(J')_k,k}[I']-w_{\alpha(I')_k,k}[J']\right)\\
			&\quad\phantom{=}+w_{\alpha(I')_k,k}[J']\left(w_{\alpha(J)_k,k}[I]-w_{\alpha(I)_k,k}[J]\right).
		\end{align*}
		With the same reasoning as before, the expressions in brackets can be written as a $\P(\C^n)$-linear combination of $[\tilde{I}]$ with strictly smaller initial terms. Thus $w_{\alpha(J')_k,k}w_{\alpha(J)_k,k}[I'][I]-w_{\alpha(I')_k,k}w_{\alpha(I)_k,k}[J'][J]$ can be written as $\P(\C^n)$-linear combination of $[\tilde{I},\tilde{I'}]$ with strictly smaller initial terms. By Buchberger's criterion, the elements $[I,I']$, $I\ge I'$, form a Gröbner basis of $\P(\C^n)\M^2_k$.\\
		It remains to show that this Gröbner basis satisfies the additional condition. For $k=1$, $[I,I']=\mathrm{in}([I,I'])$, so there is nothing to prove. For $k>1$, note that any monomial in $[I,I']$ corresponds to a permutation of the $k$ nontrivial rows of $\alpha(I)$, $\alpha(I')$, or equivalently to a permutation of the columns. In particular, the monomials in $[I,I']-\mathrm{in}([I,I'])$ correspond to permutations of the columns of $\alpha(I)$, $\alpha(I')$, where at least one of these permutations is proper. If $\beta\in\mathbb{N}^{n\times k}$ denotes the index of any monomial in this expression, then $\beta< \alpha(I)+\alpha(I')$, and at least one of the columns $1\le j\le k-1$ contains an entry that is strictly smaller than the corresponding entry in $\alpha(I)+\alpha(I')$. In particular, any product of this monomial with a monomial in $\P(\C^n)$ has the property that the matrix consisting of the first $(k-1)$ columns of its index is strictly smaller than the matrix consisting of the first $(k-1)$ columns of the matrix $\alpha(I)+\alpha(I')$. Thus, this term is not divisible by $\mathrm{in}([I,I'])$.
		\end{proof}

	Similar to \eqref{equation:polydiskSimple}, we consider the following subsets of $\Mat_{n,k}(\C)$ for $\delta>0$, $w\in\Mat_{n,k}(\C)$:
	\begin{align*}
		D_\delta(w)=\{\zeta\in \Mat_{n,k}(\C):&|\Re(\zeta_j)|_1\le |\Re(w_j)|_1+\delta,\\
		&|\Im(\zeta_j)|_1\le |\Im(w_j)|_1+\delta,~1\le j\le k\}.
	\end{align*}
	\begin{theorem}
	\label{theorem:GroebnerAlgModuleMinor}
	Let $F\in \mathcal{O}_{\Mat_{n,k}(\C)}$ be a function such that every homogeneous term in the power series expansion of $F$ belongs to the $\P(\C^n)$-module generated by $\M^2_k$. For every basis $Q_j$, $1\le j\le N_{n,k}$, of $\M^2_k$ there exist functions $g_{j}\in\mathcal{O}_{\C^n}$ such that 
	\begin{align}
		\label{eq:presentationTildeM2k}
		F(w)=\sum_{j=1}^{N_{n,k}}g_{j}(w_k)Q_j(w).
	\end{align}
	In particular, $F\in \mathcal{O}_{\C^n}\M^2_k$.
	Moreover, these functions can be chosen such that for every $\delta>0$, 
	\begin{align*}
		|g_{j}(z)|\le C_{\delta}\left(1+|z|\right)^{(2+n)N_{n,k}}\sup_{\zeta\in D_\delta(0,\dots,0,z)}|F(\zeta)|,
	\end{align*}
	for $N_{n,k}=\binom{n}{k}^2-\binom{n}{k-1}\binom{n}{n-k-1}$ and suitable constants $C_\delta>0$ depending on the chosen basis and $\delta>0$, $n$, and $k$ only.
\end{theorem}
\begin{proof}
	Obviously the claim holds as soon as it holds for one basis of $\M^2_k$. We choose the Gröbner basis of $\P(\C^n)\M^2_k$ in Lemma \ref{lemma:productsMinorsGroebnerBasis}. Then the conditions in Theorem \ref{theorem:DivisionAlg} are satisfied. Since the degree of these element is bounded by $2$ in the last argument and the number of elements in the chosen Gröbner basis is $N_{n,k}$, the claim follows from Theorem \ref{theorem:DivisionAlg}.
\end{proof}
\begin{remark}
	The $\O_{\C^n}$-module generated by $\M^2_k$ is in general not freely generated by $\M^2_k$. In particular, a given function $F\in \O_{\C^n}\M^2_k$ can admit different presentations \eqref{eq:presentationTildeM2k}.
\end{remark}

\begin{corollary}
	\label{corollary:EstimateFourierDiagonalCoordinates}
	For every $\delta>0$ there exists a constant $C_\delta>0$ such that every $F\in \O_{\C^n}\M^2_k$ satisfies
	\begin{align*}
		|F(w_1,\dots,w_k)|\le C_\delta \left(\prod_{j=1}^{k}|w_j|^2\right)\left(1+|w_k|\right)^{(2+n)N_{n,k}}\sup_{\zeta\in D_\delta(0,\dots,0,w_k)}|F(\zeta)|,
	\end{align*}		
\end{corollary}
\begin{proof}
	Let $F(w)=\sum_{j=1}^{N_{n,k}} g_j(w_k)Q_j(w)$ be a decomposition as in Theorem~\ref{theorem:GroebnerAlgModuleMinor}. Then
	\begin{align*}
		|F(w)|\le \sum_{j=1}^{N_{n,k}}|Q_j(w)| \cdot |g_{j}(w_k)|,
	\end{align*}
	so the claim follows by estimating each term in the sum: $|Q_j(w)|$ is bounded by a constant multiple of $\prod_{j=1}^k|w_j|^2$ since it is homogeneous of degree $2$ in each argument, and $g_{j}$ satisfies the estimate in Theorem~\ref{theorem:GroebnerAlgModuleMinor} by assumption.
\end{proof}

\section{Dually epi-translation invariant valuations}
	\label{section:DuallyEpiValuations}
	\subsection{Support and topology}
	\label{section:Valuations_Support_Topology}
	
	 The homogeneous decomposition of $\VConv(\R^n)$ implies that the \emph{polarization} of $\mu$, given for $f_1,\dots,f_k\in\Conv(\R^n,\R)$ by
	\begin{align*}
		\bar{\mu}(f_1,\dots,f_k)=\frac{1}{k!}\frac{\partial^k}{\partial\lambda_1\dots\partial \lambda_k}\Big|_0\mu\left(\sum_{j=1}^{k}\lambda_jf_j\right),
	\end{align*}
	is well defined since $(\lambda_1,\dots,\lambda_k)\mapsto \mu\left(\sum_{j=1}^{k}\lambda_jf_j\right)$ is a polynomial in $\lambda_j\ge 0$. The polarization is essentially a multilinear functional, which can be extended to a distribution on $(\R^n)^k$. 
	
	\begin{theorem}[\cite{Knoerrsupportduallyepi2021} Theorem 2]
		\label{theorem:GWDistributions}
		For every $\mu\in\VConv_k(\R^n)$ there exists a unique distribution $\GW(\mu)$ on $(\R^n)^k$ with compact support which satisfies the following property: If $f_1,...,f_k\in \Conv(\R^n,\R)\cap C^\infty(\R^n)$, then
		\begin{align*}
			\GW(\mu)[f_1\otimes...\otimes f_k]=\bar{\mu}(f_1,...,f_k),
		\end{align*}
		where $\bar{\mu}$ is the polarization of $\mu$.\\
		Moreover, the support of $\GW(\mu)$ is contained in the diagonal in $(\R^n)^k$.
	\end{theorem}
	If $\phi_1,\dots,\phi_k\in C^\infty_c(\R^n)$, then $\GW(\mu)[\phi_1\otimes\dots\otimes \phi_k]$ can be calculated in the following way, compare \cite[Section~5]{Knoerrsupportduallyepi2021}: Choose convex functions $f_1,\dots,f_k\in \Conv(\R^n,\R)$ such that $h_j:=f_j+\phi_j$ is convex for every $1\le j\le k$. Then
	\begin{align}
		\label{equation:formulaGWElementaryTensor}
		\begin{split}
		&\GW(\mu)[\phi_1\otimes\dots\otimes \phi_k]\\
		&\quad =\sum\limits_{j=0}^k(-1)^{k-j}\frac{1}{(k-j)!l!}\sum_{\sigma\in S_k}\bar{\mu}(h_{\sigma(1)},\dots,h_{\sigma(j)},f_{\sigma(j+1)},\dots,f_{\sigma(k)}),
	\end{split}
	\end{align}
	where $S_k$ denotes the symmetric group. Note that this corresponds to the multilinear extension of $\bar{\mu}$ evaluated in $\phi_j=h_j-f_j$.\\

	The support of these distributions gives rise to the following notion of support for valuations: If $\mu=\sum_{k=0}^n\mu_k$ is the decomposition of $\mu\in\VConv(\R^n)$ into its homogeneous components, then 
	\begin{align*}
		\supp\mu:=\bigcup_{k=1}^n \Delta_k^{-1}(\supp\GW(\mu_k)),
	\end{align*}
	where $\Delta_k:\R^n\rightarrow(\R^n)^k$ is the diagonal embedding. Note that the support of a valuation is a compact subset of $\R^n$. It is characterized by the following property.
	\begin{proposition}[\cite{Knoerrsupportduallyepi2021} Proposition 6.3]
		\label{proposition:characterizationSupport}
		The support of $\mu\in\VConv(\R^n)$ is minimal (with respect to inclusion) among the closed sets $A\subset \R^n$ with the following property: If $f,g\in\Conv(\R^n,\R)$ satisfy $f=g$ on an open neighborhood of $A$, then $\mu(f)=\mu (g)$.
	\end{proposition}
	Given a closed set $A\subset\R^n$, we denote by $\VConv_A(\R^n)$ the subspace of valuations with support contained in $A$. It is easy to see that these spaces are closed.
	\begin{proposition}[\cite{Knoerrsupportduallyepi2021} Proposition 6.8 and Corollary 6.10]
		\label{proposition:NormsSubspaces}
		Let $A\subset V$ be compact and convex, $\delta>0$. For $\mu\in \VConv_{A}(\R^n)$ define
		\begin{align*}
			\|\mu\|_{A,\delta}:=\sup\left\{|\mu(f)|  \ : \ f\in\Conv(\R^n,\R), \sup\limits_{x\in A+\delta B_1}|f(x)|\le 1 \right\}.
		\end{align*}
		This defines a continuous norm on $\VConv_{A}(\R^n)$ that metrizes the subspace topology. In particular, $\VConv_A(\R^n)$ is a Banach space with respect to this norm.
	\end{proposition}
	Recall that $g\in \Aff(n,\R)$ acts on $\mu\in \VConv(\R^n)$ by
	\begin{align*}
		\pi(g)\mu(f)=\mu(f\circ g),\quad f\in\Conv(\R^n,\R).
	\end{align*}
	\begin{lemma}
		\label{lemma:continuityActionAff_VConv}
		\begin{enumerate}
			\item $\pi$ defines a continuous representation on $\VConv(\R^n)$, i.e. the following map is continuous:\begin{align*}
				\Aff(n,\R)\times\VConv(\R^n)&\rightarrow\VConv(\R^n)\\
				(g,\mu)&\mapsto \pi(g)\mu
			\end{align*}
			\item For $A\subset\R^n$ compact and $C\subset\Aff(n,\R)$ compact, this map restricts to a continuous map
			\begin{align*}
				C\times \VConv_A(\R^n)\rightarrow\VConv_{CA}(\R^n).
			\end{align*}
		\end{enumerate}
	\end{lemma}
	\begin{proof}
		It is easy to see that $\pi$ defines a representation. In order to see that the map in (1) is continuous, let $K\subset\Conv(\R^n,\R)$ be a compact subset. Since the action of $\Aff(n,\R)$ on $\Conv(\R^n,\R)$ is continuous by Lemma~\ref{lemma:continuityActionAffonConv}, for any compact set $C\subset \Aff(n,\R)$, the set
		\begin{align*}
			K_C:=\{f\circ g: f\in K,~g\in C\}\subset\Conv(\R^n,\R)
		\end{align*}
		is compact. For $g_1,g_2\in C$,
		\begin{align*}
			\|\pi(g_1)\mu_1-\pi(g_2)\mu_2\|_K\le& \|\pi(g_1)\mu_1-\pi(g_1)\mu_2\|_K+\|\pi(g_1)\mu_2-\pi(g_2)\mu_2\|_K,
		\end{align*}
		where
		\begin{align*}
			\|\pi(g_1)\mu_2-\pi(g_2)\mu_2\|_K=\sup_{f\in K}|\mu_2(f\circ g_1)-\mu_2(f\circ g_2)|.
		\end{align*}
		Fix $g_2\in\Aff(n,\R)$ and $\mu_2\in\VConv(\R^n)$ and let $\epsilon>0$ be given. The map
		\begin{align*}
			\Aff(n,\R)\times K&\rightarrow \R\\
			(g_1,f)&\mapsto|\mu_2(f\circ g_1)-\mu_2(f\circ g_2)|
		\end{align*}
		is continuous by Lemma~\ref{lemma:continuityActionAffonConv} and vanishes for $g_1=g_2$. Since $K$ is compact and $\Aff(n,\R)$ is locally compact, this map is locally uniformly continuous, so we find a compact neighborhood $C$ of $g_2$ such that 
		\begin{align*}
			|\mu_2(f\circ g_1)-\mu_2(f\circ g_2)|<\epsilon\quad \text{for all}~g_1\in C,~f\in K.
		\end{align*}
		Thus for $g_1\in C$,
		\begin{align*}
			\|\pi(g_1)\mu_1-\pi(g_2)\mu_2\|_K\le& \|\pi(g_1)\mu_1-\pi(g_1)\mu_2\|_K+\|\pi(g_1)\mu_2-\pi(g_2)\mu_2\|_K\\
			\le&\|\mu_1-\mu_2\|_{K_C}+\epsilon 
		\end{align*}
		Since $K_C$ is compact, the first term is smaller than $\epsilon$ for all $\mu_1$ belonging to a suitable neighborhood of $\mu_2$. Thus this map is continuous.\\
		In order to show (2), it is therefore sufficient to show that the map is well defined, i.e. that for compact sets $C\subset \Aff(n,\R)$, $A\subset\R^n$, $g\in C$, and $\mu\in\VConv_A(\R^n)$, the support of $\pi(g)\mu$ is contained in $CA$. Since $CA$ is compact, this is a simple consequence of Proposition~\ref{proposition:characterizationSupport}.
	\end{proof}
	
	\subsection{Monge--Amp\`ere operators and the differential cycle}
		\label{section:MAOperators}
		Recall that we denote by $\M(\R^n)=(C_c(\R^n))'$ the space of (complex) Radon measures on $\R^n$, where we equip $\M(\R^n)$ with the weak-* topology. Let $\MAVal(\R^n)$ denote the space of all continuous valuations $\Psi:\Conv(\R^n,\R)\rightarrow\M(\R^n)$ that are
		\begin{enumerate}
			\item locally determined: if $f,h\in\Conv(\R^n,\R)$ satisfy $f|_U=h|_U$ for some open set $U\subset\R^n$, then
			\begin{align*}
				\Psi(f)|_U=\Psi(h)|_U;
			\end{align*}
			\item translation equivariant: for $f\in\Conv(\R^n,\R)$, $x\in\R^n$, 
			\begin{align*}
				\Psi(f)[B+x]=\Psi(f(\cdot+x))[B]\quad\text{for}~B\subset\R^n~\text{bounded Borel set};
			\end{align*}
			\item dually epi-translation invariant.
		\end{enumerate}
		Examples of functionals of this type are given by $\MA(\cdot[k],Q_1,\dots,Q_{n-k})$, where $Q_1,\dots,Q_{n-k}$ are positive definite quadratic forms. A second type of example can be constructed from the differential cycle. Recall that we denote by $D(f)$ the differential cycle of a convex function $f\in \Conv(\R^n,\R)$, which is an integral $n$-current on $T^*\R^n$, compare \cite{FuMongeAmperefunctions.1989}. Let $p_1:T^*\R^n\rightarrow\R^n$ denote the natural projection. 	
		\begin{theorem}[\cite{KnoerrMongeAmpereoperators2024} Theorem 4.10]
			\label{theorem:measuresDifferentialCycle}
			Let $\tau\in\Omega^n(T^*\R^n)$ be a smooth differential $n$-form and define $\Psi_\tau(f)\in\M(\R^n)$ for $f\in\Conv(\R^n)$ by 
			\begin{align*}
				\Psi_\tau(f)[B]:=D(f)[1_{p_1^{-1}(B)}\tau]\quad\text{for all bounded Borel sets } B\subset\R^n.
			\end{align*}
			Then $\Psi_\tau:\Conv(\R^n,\R)\rightarrow\M(\R^n)$ is a continuous and locally determined valuation. If $\tau$ is invariant with respect to translations in the second factor of $T^*\R^n=\R^n\times(\R^n)^*$, then $\Psi_\tau$ is dually epi-translation invariant. If $\tau$ is invariant with respect to translations in the first factor, then $\Psi_\tau$ is translation equivariant.
		\end{theorem}	
		As shown in \cite{KnoerrMongeAmpereoperators2024}, $\MAVal(\R^n)=\bigoplus_{k=0}^n\MAVal_k(\R^n)$, where $\MAVal_k(\R^n)$ denotes the subspace of all $k$-homogeneous elements. The main result of \cite{KnoerrMongeAmpereoperators2024} may be summarized as follows.
		\begin{theorem}
			\label{theorem:ClassificationMAVal}
			For a continuous map $\Psi:\Conv(\R^n,\R)\rightarrow\M(\R^n)$ the following are equivalent:
			\begin{enumerate}
				\item $\Psi\in\MAVal_k(\R^n)$.
				\item $\Psi$ is a linear combination of mixed Monge--Amp\`ere operators of the form $\MA(\cdot[k],Q_1,\dots,Q_{n-k})$ for positive definite quadratic forms $Q_1,\dots,Q_{n-k}$.
				\item There exists a differential form $\tau\in\Lambda^{n-k}(\R^n)^*\otimes\Lambda^{n-k}((\R^n)^*)^*$ such that for all $f\in\Conv(\R^n,\R)$, $B\subset\R^n$ bounded Borel set,
				\begin{align*}
					\Psi(f)[B]=D(f)[1_{p_1^{-1}(B)}\tau].
				\end{align*}
			\end{enumerate}
		\end{theorem}
		Due to (2), we consider $\MAVal(\R^n)$ as the space of translation equivariant Monge--Amp\`ere operators. As a byproduct, we obtain that $\MAVal_k(\R^n)$ is a finite dimensional space. More precisely, $\dim\MAVal_k(\R^n)=\binom{n}{k}^2-\binom{n}{k-1}\binom{n}{n-k-1}$, compare \cite[Section 6.1]{KnoerrMongeAmpereoperators2024}.
		Note that any differential form $\omega\in \Omega^{n-k}_c(\R^n)\otimes \Lambda^k((\R^n)^*)^*$ may be written as a finite linear combination of differential forms of the form $\phi(x)\tau$ for $\phi\in C_c^\infty(\R^n)$ and $\tau\in \Lambda^{n-k}(\R^n)\otimes \Lambda^k((\R^n)^*)^*$. In particular, the valuation $\mu=D(\cdot)[\omega]$ may be written as a finite linear combination of valuations of the form
		\begin{align*}
			f\mapsto \int_{\R^n}\phi(x)d\Psi_\tau(f;x).
		\end{align*}
		Recall that the mass of an integral current $T$ defined on an open set $U\subset\R^n$ is given by
		\begin{align*}
			M_U(T):=\sup_{\omega\in \Omega_c^k(U),\|\omega\|_\infty\le 1}|T(\omega)|.
		\end{align*}
		We have the following estimate for the mass of the differential cycle over the open balls $U_R(0)=\{x\in\R^n: |x|<R\}$ for $R>0$:
	\begin{lemma}[\cite{KnoerrSmoothvaluationsconvex2024} Lemma 4.8]
			\label{lemma:massDifferentialCycle}
			For $f\in\Conv(\R^n,\R)$, 
			\begin{align*}
				M_{p_1^{-1}(U_R(0))}(D(f))\le2^n\omega_n \left(\sup_{|x|\le R+1}|f(x)|\right)^n.
			\end{align*}
	\end{lemma}
	This implies the following bound for the measure-valued valuations defined above.	
	\begin{corollary}
		\label{corollary:trivialBoundMAOperators}
		For every $\tau\in \Lambda^n(\R^n\times(\R^n)^*)^*$ and $\phi\in C_c(\R^n)$ with $\supp\phi\subset B_R(0)$:
		\begin{align*}
			\left|\int_{\R^n} \phi(x)d\Psi_\tau(f;x)\right|\le  2^n\omega_n \left(\sup_{|x|\le R+1}|f(x)|\right)^n\|\tau\|_\infty \|\phi\|_\infty.
		\end{align*}
	\end{corollary}

	\subsection{Basic properties of smooth valuations}
		Recall that we call a valuation $\mu\in\VConv(\R^n)$ smooth if the map
		\begin{align*}
			\R^n&\rightarrow \VConv(\R^n)\\
			x&\mapsto [f\mapsto\mu(f(\cdot+x))]
		\end{align*}
		is an infinitely differentiable map.

	Note that for $|x|\le R$, the support of $f\mapsto\mu(f(\cdot+x))$ is contained in the $R$-neighborhood of the support of $\mu$, so on bounded sets of $\R^n$ this is a continuous, Banach space-valued map by Proposition~\ref{proposition:NormsSubspaces}.
	\begin{lemma}
		\label{lemma:MollifiedValuations}
		For $\phi\in C_c(\R^n)$ and $\mu\in\VConv_k(\R^n)$ define $\mu_\phi\in \VConv_k(\R^n)$ by
		\begin{align*}
			\mu_\phi(f):=\int_{\R^n}\phi(x)\mu(f(\cdot-x))dx\quad\text{for}~f\in\Conv(\R^n,\R).
		\end{align*}
		\begin{enumerate}
			\item If $\phi\in C^\infty_c(\R^n)$, then $\mu_\phi$ is a smooth valuation and $\supp\mu_\phi\subset \supp\mu+\supp\phi$. 
			\item If $ W\subset \VConv_k(\R^n)$ is a closed and translation invariant subspace with $\mu\in W$, then $\mu_\phi\in W$.
		\end{enumerate}
	\end{lemma}
	\begin{proof}
		It follows from dominated convergence that $\mu_\phi$ is well defined and continuous. Applying Proposition~\ref{proposition:characterizationSupport}, the inclusion $\supp\mu_\phi\subset \supp\mu+\supp\phi$ follows directly from the definition of $\mu_\phi$.\\
		
		Assume that $A,B\subset\R^n$ are compact and convex with $\supp\mu\subset A$, $\supp\phi\subset B$. 
		Not that if $f\in\Conv(\R^n,\R)$ is bounded by $1$ on $A+B+B_1(0)$, then the function $f(\cdot-x)$ is bounded by $1$ on $A+B_1(0)$ for every $x\in B$. From the definition of the norms, we thus obtain
		\begin{align}
			\label{equation:estimateConvolution}
			\|\mu_\phi\|_{A+B,1}\le \|\phi\|_{L^1(\R^n)} \|\mu\|_{A,1}.
		\end{align}
		We identify $\R^n$ with the subgroup of $\Aff(n,\R)$ generated by translations. For $y\in \R^n$ and $\mu\in\VConv_k(\R^n)$ we have
		\begin{align}
			\label{equation:equivarianceConvolution}
			\begin{split}
				\pi(y)\mu_\phi(f)=&\int_{\R^n}\phi(x)\mu(f(\cdot+y-x))dx\\
				=&\int_{\R^n}\phi(x+y)\mu(f(\cdot-x))dx=\mu_{\phi(\cdot+y)}(f).
			\end{split}
		\end{align}
		In order two show (1), assume that $\phi$ is a smooth function. We claim that the previous map is continuously differentiable with differential in $y$ given by $v\mapsto \mu_{\partial_v\phi(\cdot+y)}$. Note that it follows from \eqref{equation:equivarianceConvolution} that it is sufficient to show that that the map is differentiable in $y=0$ with the given differential, since $(y,v)\mapsto \mu_{\partial_v\phi(\cdot+y)}$ is a continuous map. In addition, $\phi(\cdot+v)$ is supported on $B+B_R(0)$ for every $|v|\le R$. Consequently, we obtain for $v\ne 0$,
		\begin{align*}
			&\left\|\frac{\pi(v)\mu_{\phi}-\mu_{\phi}-\mu_{\partial_v\phi}}{|v|}\right\|_{A+B+B_R(0),1}=\left\|\mu_{\frac{\phi(\cdot+v)-\phi-\partial_v\phi}{|v|}}\right\|_{A+B+B_R(0),1}\\
			&\quad\le \int_{\R^n}\frac{\left|\phi(x+v)-\phi(x)-\partial_v\phi(x)\right|}{|v|}dx \cdot\|\mu\|_{A,1}.
		\end{align*}
		The integrand in the last integral converges uniformly to $0$ for $|v|\rightarrow0$ and is supported on a compact subset independent of $v$, so this integral converges to $0$. Thus $y\mapsto \pi(y)\mu_\phi$ is differentiable in $y=0$ with the desired derivative. Iterating this argument, we obtain that $y\mapsto \pi(y)\mu_\phi$ is a smooth map.\\
		
		In order to show (2), assume that $W\subset\VConv_k(\R^n)$ is a closed and translation invariant subspace with $\mu\in W$. Note that we may write
		\begin{align*}
			\mu_\phi(f)=\int_{\R^n}\phi(x)[\pi(-x)\mu](f)dx,
		\end{align*}
		where $\pi$ denotes the representation of the affine group on $\VConv_k(\R^n)$ defined in Section \ref{section:Valuations_Support_Topology}. The function $x\mapsto \phi(x)\pi(-x)\mu$ is a continuous $\VConv_{A+B}(\R^n)\cap W$-valued map by Lemma~\ref{lemma:continuityActionAff_VConv}, where we used that $W$ is translation invariant by assumption. Since $W$ is closed, $\VConv_{A+B}(\R^n)\cap W$ is a Banach space by Proposition~\ref{proposition:NormsSubspaces}. In particular, $\mu_\phi=\int_{\R^n}\phi(x)\pi(-x)\mu dx$ as a Lebesgue--Bochner integral, so $\mu_\phi\in \VConv_{A+B}(\R^n)\cap W$ as well.	
	\end{proof}
	For $A\subset\R^n$ compact, set $\VConv_{k,A}(\R^n):=\VConv_k(\R^n)\cap \VConv_A(\R^n)$. Lemma \ref{lemma:MollifiedValuations} implies the following approximation results under support restrictions, which we split into two parts for simplicity.
	\begin{corollary}
		\label{corollary:ApproxTranslationSmoothValuations}
		Let $A\subset\R^n$ be compact, $B\subset\R^n $ a compact neighborhood of $A$. Every valuation $\mu\in\VConv_{k,A}(\R^n)$ can be approximated by a sequence $(\mu_j)_j$ of smooth valuations in $\VConv_{k,B}(\R^n)$.
	\end{corollary}
	\begin{proof}
		Let $\phi\in C^\infty_c(\R^n)$ be a nonnegative function with $\supp\phi\subset B_1(0)$, $\int_{\R^n}\phi(x)dx=1$ , and set $\phi_\delta(x):=\delta^{-n}\phi\left(\frac{x}{\delta}\right)$. Then the valuation $\mu_{\phi_\delta}$ defined as in Lemma~\ref{lemma:MollifiedValuations} is a smooth valuation with $\supp\mu_{\phi_\delta}\subset \supp\mu+B_{\delta}(0)$ by Lemma~\ref{lemma:MollifiedValuations}. Since $B$ is a neighborhood of $\supp \mu\subset A$, we can choose $\delta>0$ small enough to obtain the desired support restriction. It remains to see that $\mu_{\phi_\delta}$ converges to $\mu$ for $\delta\rightarrow0$. Note that for $f\in\Conv(\R^n,\R)$,
		\begin{align*}
			|\mu(f)-\mu_{\phi_{\delta}}(f)|=&\left|\int_{\R^n}\phi_\delta(x)\mu(f)dx-\int_{\R^n}\phi_\delta(x)\mu(f(\cdot-x))dx\right|\\
			\le & \int_{\R^n}\phi_\delta(x)|\mu(f)-\mu(f(\cdot-x))|dx.
		\end{align*}
		Let $K\subset\Conv(\R^n,\R)$ be a compact subset. Since $\mu$ is continuous, the map
		\begin{align*}
			K\times \R^n &\rightarrow \R\\
			(f,x)&\mapsto |\mu(f)-\mu(f(\cdot-x))| 
		\end{align*}	
		is uniformly continuous on compact subsets. In particular, for every $\epsilon>0$ there exists $\delta(\epsilon)>0$ such that $|\mu(f)-\mu(f(\cdot-x))|<\epsilon$ for all $f\in K$ and $|x|\le \delta(\epsilon)$. Since $\phi_\delta$ is supported on $B_{\delta}(0)$, this implies for $0<\delta<\delta(\epsilon)$,
		\begin{align*}
			|\mu(f)-\mu_{\phi_{\delta}}(f)|\le \epsilon \quad\text{for all}~f\in K.
		\end{align*}
		Thus $\mu_{\phi_\delta}$ converges to $\mu$ in $\VConv_k(\R^n)$ for $\delta\rightarrow0$.
	\end{proof}
	\begin{theorem}
		\label{theorem:approxBySmoothInConvexSupport}
		Let $W\subset\VConv(\R^n)$ be a closed, translation invariant subspace.
		\begin{enumerate}
			\item Smooth valuations are dense in $W$. 
			\item Let $A\subset \R^n$ be closed and convex. If $A$ has nonempty interior, then every valuation $\mu\in W$ with support contained in the interior of $A$ can be approximated by a sequence of smooth valuations in $W$ with support in the interior of $A$.
			\item For $A\subset \R^n$ closed and convex let $W_A$ denote the subspace of $W$ of all valuations with support contained in $A$. If $W$ is invariant under dilations and $A$ has nonempty interior, then smooth valuations are sequentially dense in $W_A$.
		\end{enumerate}
	\end{theorem}
	\begin{proof}
		Note that (1) is a special case of (2). Let us show (2). 
		Assume that $\supp\mu$ is contained in the interior of $A$, and let $\phi_\delta\in C^\infty_c(\R^n)$ be the functions from the proof of Corollary~\ref{corollary:ApproxTranslationSmoothValuations}. Because $W$ is translation invariant and closed, the valuation 
		\begin{align*}
			\mu_\delta(f):=\int_{\R^n}\phi_\delta(x)\mu(f(\cdot-x))dx
		\end{align*}
		is smooth and belongs to $W$ by Lemma~\ref{lemma:MollifiedValuations}, and it converges to $\mu$ for $\delta\rightarrow0$ by the proof of Corollary~\ref{corollary:ApproxTranslationSmoothValuations}. Moreover, $\supp\mu_\delta$ is contained in $\{x\in\R^n: d(x,\supp \mu)\le \delta\}$. As $\supp\mu$ is compact and contained in the interior of $A$, this set is contained in the interior of $A$ for all $\delta>0$ sufficiently small. Thus $\mu_\delta\in W_A$ for all $\delta>0$ sufficiently small, which completes the proof of the second claim. \\
		
		In order to show (3), assume that $A$ has nonempty interior. Note that since $W$ is translation invariant, we may assume that $0$ is contained in the interior of $A$. Due to (2) it is sufficient to show that any valuation $\mu\in W_A$ with support contained in the interior of $A$ can be approximated by a sequence in $W_A$ of valuations supported in the interior of $A$. Such a sequence of valuations is given by $\mu_j(f):=\mu\left(f\left(\frac{j-1}{j}\cdot\right)\right)$, $j\in\mathbb{N}$. Note that this valuation belongs to $W$ because $W$ is invariant under dilations.
	\end{proof}

	The next result will be used in Section \ref{section:PWSValuations} in the proof of Theorem \ref{maintheorem:DescriptionsSmoothValuations}.
	\begin{lemma}
		\label{lemma:SmoothValTranslationSmooth}
		If $\mu\in\VConv_k(\R^n)$ is given by integration with respect to the differential cycle, then $\mu$ is a smooth valuation.
	\end{lemma}
	\begin{proof}
		Every such valuation can be written as a linear combination of valuations of the form
		\begin{align*}
			\Psi_{\tau}[\phi](f):=\int_{\R^n}\phi(x)d\Psi_\tau(f;x), \quad f\in\Conv(\R^n,\R),
		\end{align*}
		where $\phi\in C^\infty_c(\R^n)$, $\tau\in \Lambda^{n-k}(\R^n)^*\otimes \Lambda^k((\R^n)^*)^*$ is a constant differential form, and $\Psi_\tau$ is defined as in Theorem~\ref{theorem:measuresDifferentialCycle}. We claim that the map 
		\begin{align*}
			\R^n&\rightarrow\VConv_k(\R^n)\\
			y&\mapsto [f\mapsto \Psi_{\tau}[\phi](f(\cdot+y))]
		\end{align*}
		is differentiable in $y$ with derivative $v\mapsto \Psi_{\tau}[-\partial_v\phi(\cdot-y)]$. Note that $\Psi_\tau$ is equivariant with respect to translations, so
		\begin{align*}
			\pi(y)\Psi_\tau[\phi](f)=\Psi_{\tau}[\phi](f(\cdot+y))=\Psi_{\tau}[\phi(\cdot-y)](f).
		\end{align*}
		In particular, it is sufficient to show that claim for $y=0$. Fix $R>0$ such that $\supp \phi\subset U_R(0)$. For $v\in\R^n$ with $|v|\le r$, the support of the valuation
		\begin{align*}
			\pi(v)\Psi_{\tau}[\phi]-\Psi_{\tau}[\phi]-\Psi_{\tau}[-\partial_v\phi]
		\end{align*}
		is contained in $U_{R+r}(0)$. From Corollary~\ref{corollary:trivialBoundMAOperators} we obtain for $v\ne 0$, $|v|\le r$,
		\begin{align*}
			&\frac{\|\pi(v)\Psi_{\tau}[\phi]-\Psi_{\tau}[\phi]-\Psi_{\tau}[-\partial_v\phi]\|_{B_{R+r}(0),1}}{|v|}\\
			&\quad\le 2^n\omega_n\|\tau\|_\infty \frac{\|\phi(\cdot-v)-\phi+\partial_v\phi\|_\infty}{|v|}.
		\end{align*}
		It is easy to see that $\frac{\phi(\cdot-v)-\phi+\partial_v\phi}{|v|}$ converges uniformly to $0$ for $v\rightarrow0$. Thus $y\mapsto\pi(y)\Psi_\tau[\phi]$ is differentiable in $y=0$ with derivative $v\mapsto \Psi_\tau[-\partial_v\phi]$.\\
		As $\phi$ is a smooth function, we can now proceed by induction to obtain that $\Psi_\tau[\phi]$ is a smooth valuation.
	\end{proof}

\section{Structure of the Fourier--Laplace transform of Goodey--Weil distributions}
	\label{section:FourierGW}
	\subsection{Basic properties of the Fourier--Laplace transform of Goodey--Weil distributions}
In this section we establish some necessary estimates for the Fourier--Laplace transform, which rely on the estimates used in \cite{Knoerrsupportduallyepi2021} in the construction of the Goodey--Weil distributions.

\begin{lemma}
	\label{lemma:differenceConvexFunctionsRestrSupport}
	Let $\phi\in C^\infty(\R^n)$, $A\subset \R^n$ compact and convex. Then there exists a convex function $f\in\Conv(\R^n,\R)$ such that $f+\phi$ is convex and such that
	\begin{align*}
		\max(|f(x)|,|f(x)+\phi(x)|)\le \left(\frac{\diam(A)^2}{8}+1\right)\|\phi\|_{C^2(A)}\quad\text{for all } x\in A.
	\end{align*}
\end{lemma}
\begin{proof}
	Since $A$ is convex and compact, we may choose $x_0\in A$ such that $\sup_{x\in A}|x-x_0|=\frac{\diam A}{2}$.
	By \cite[Lemma 5.1]{Knoerrsupportduallyepi2021}, there is a function $\hat{f}\in\Conv(\R^n,\R)$ such that $\hat{f}+\phi$ is convex. Set $c:=\sup_{x\in A}|\hat{f}(x)|$ and consider the convex functions defined for $\delta\ge 0$ by
	\begin{align*}
		f_\delta(x):=\max(\hat{f}(x)-c,0)+\|\phi\|_{C^2(A+\delta B_1(0))}\frac{|x-x_0|^2}{2}.
	\end{align*}
	Assume that $\delta>0$. Then for every $x\in \R^n$ there is a neighborhood $U$ such that $f_\delta+\phi$ is convex on $U$, so $f_\delta+\phi$ is convex. As $f:=f_0$ is the pointwise limit of $f_\delta$ for $\delta\rightarrow0$, $f+\phi$ is also convex. Obviously, $f$ and $f+\phi$ satisfy the desired inequality.
\end{proof}

 The following result is stated in \cite[Lemma 4.13]{Knoerrsupportduallyepi2021} for compact sets only but extends trivially to arbitrary convex sets.
\begin{lemma}
	\label{lemma:normPolarizationBoundSupport}
		There exists a constant $C_{k}>0$ depending on $0\le k\le n$ only such that the following holds: If $K\subset \Conv(\R^n,\R)$ is convex, then 
	\begin{align*}
		\|\bar{\mu}\|_{K}:=\sup_{f_1,\dots,f_k\in K}|\bar{\mu}(f_1,\dots,f_k)|\le C_{k}\|\mu\|_{K}.
	\end{align*}
\end{lemma}

Let $A\subset\R^n$ be compact and convex. We will apply this result to the sets
\begin{align*}
	K_{A,\delta}:=\{f\in \Conv(\R^n,\R): |f(x)|\le 1~\text{for}~x\in A+\delta B_1(0)\}.
\end{align*}
Note that for $\mu\in \VConv_A(\R^n)$,
\begin{align*}
	\|\mu\|_{A,\delta}=\sup_{f\in K_{A,\delta}} |\mu(f)|
\end{align*}
is the norm considered in Proposition~\ref{proposition:NormsSubspaces}.

\begin{proposition}
	\label{proposition:estimateFourierGW}
	Let $A\subset\R^n$ be compact and convex, $\mu\in \VConv_{k,A}(\R^n)$. Then
	\begin{align*}
		&|\mathcal{F}(\GW(\mu))[w_1,\dots,w_k]|\\
		&\qquad \le \tilde{C}_{A,k}\|\mu\|_{A,\delta}\left(\prod_{j=1}^k\left(1+|w_j|\right)\right)^2e^{\sum_j h_{A+\delta B_1(0)}(\Im(w_j))}
	\end{align*}
	for $w_1,\dots,w_k\in\C^n$, where $\tilde{C}_{A,k}:=4^{3k}C_k\left(\frac{(\diam A+2\delta)^2}{8}+1\right)^k$ for the constant $C_k>0$ in Lemma~\ref{lemma:normPolarizationBoundSupport}.
\end{proposition}
\begin{proof}
	Let $\delta>0$ and set $A_\delta:= A+\delta B_1(0)$. By definition
	\begin{align*}
		\mathcal{F}(\GW(\mu))[w_1,\dots,w_k]=\GW(\mu)[\exp\left( -i\langle w_1,\cdot\rangle\right)\otimes\dots\otimes\exp\left( -i\langle w_k,\cdot\rangle\right)].
	\end{align*}
	Set \begin{align*}
		\phi_j(x):=&\Re \exp\left(-i\langle w_j,x\rangle\right)=\exp(\langle\Im w_j,x\rangle) \cos(\langle \Re w_j,x\rangle),\\
		\psi_j(x):=&\Im \exp\left( -i\langle w_j,x\rangle\right)=\exp(\langle \Im w_j,x\rangle) \sin(\langle \Re w_j,x\rangle).
	\end{align*}
	Then
	\begin{align*}
		\max(|D^2\phi_j(x)|,|D^2 \psi_j(x)|)\le 4\exp(\langle\Im w_j,x\rangle)|w_j|^2.
	\end{align*}
	This implies
	\begin{align*}
		\|\phi_j\|_{C^2(A_\delta)},\|\psi_j\|_{C^2(A_\delta)}\le 8(1+|w_j|)^2\exp(h_{A_\delta}(\Im w_j)).
	\end{align*}
	By Lemma~\ref{lemma:differenceConvexFunctionsRestrSupport} there exist functions $f_j,h_j\in\Conv(\R^n,\R)$ such that $f_j+\phi_j$, $h_j+\psi_j$ are convex and such that
	\begin{align*}
		\max(|f_j(x)|,|f_j(x)+\phi_j(x)|)\le& C_{A_\delta}\|\phi_j\|_{C^2(A_\delta)}\quad\text{for all } x\in A_\delta,\\
		\max(|h_j(x)|,|h_j(x)+\psi_j(x)|)\le& C_{A_\delta}\|\psi_j\|_{C^2(A_\delta)}\quad\text{for all } x\in A_\delta,
	\end{align*}
	where $C_{A_\delta}=\frac{(\diam A+2\delta)^2}{8}+1$. By multilinearity,
	\begin{align*}
		&\mathcal{F}(\GW(\mu))[w_1,\dots,w_k]=\GW(\mu)[[\phi_1+i\psi_1]\otimes\dots\otimes[\phi_k+i\psi_k]]\\
		=&\sum_{j=0}^{k}\frac{i^{k-j}}{j!(k-j)!}\sum_{\sigma\in S_k}\GW(\mu)[\phi_{\sigma(1)}\otimes\dots\otimes\phi_{\sigma(j)}\otimes\psi_{\sigma(j+1)}\otimes\dots\otimes\psi_{\sigma(k)}],
	\end{align*}
	where $S_k$ denotes the symmetric group. If $\tilde{\phi}_1,\dots,\tilde{\phi}_k\in \{\phi_j:1\le j\le k\}\cup \{\psi_j:1\le j\le k\}$ is any of the tuples of functions, let $\tilde{f}_j$, denote the corresponding convex function $f_j$ or $h_j$. From \eqref{equation:formulaGWElementaryTensor}, we obtain
	\begin{align*}
		&|\GW(\mu)[\tilde{\phi}_1\otimes\dots\otimes\tilde{\phi}_k]|\\
		=&\left|\sum_{j=0}^{k}\frac{(-1)^{k-j}}{j!(k-j)!}\sum_{\sigma\in S_k}\bar{\mu}(\tilde{f}_{\sigma(1)}+\tilde{\phi}_{\sigma(1)},\dots,\tilde{f}_{\sigma(j)}+\tilde{\phi}_{\sigma(j)},\tilde{f}_{\sigma(j+1)},\dots,\tilde{f}_{\sigma(k)})\right|.
	\end{align*}
	Now note that $\frac{\tilde{f}_j+\phi_j}{C_{A_\delta}\|\tilde{\phi}_{i}\|_{C^2(A_\delta)}}$ and $\frac{\tilde{f}_j}{C_{A_\delta}\|\tilde{\phi}_{i}\|_{C^2(A_\delta)}}$ are both bounded by $1$ on $A_\delta$. Lemma~\ref{lemma:normPolarizationBoundSupport} thus implies
	\begin{align*}
		|\GW(\mu)[\tilde{\phi}_1\otimes\dots\otimes\tilde{\phi}_k]|\le& \left(\sum_{j=0}^{k}\frac{k!}{j!(k-j)!}\right)\|\bar{\mu}\|_{A,\delta}C_{A_\delta}^k\prod_{j=1}^k\|\tilde{\phi}_{j}\|_{C^2(A_\delta)}\\
		\le& 2^kC_k\|\mu\|_{A,\delta}C_{A_\delta}^k\prod_{j=1}^k\|\tilde{\phi}_{j}\|_{C^2(A_\delta)}.
	\end{align*}
	Using our estimates for $\|\phi_j\|_{C^2(A_\delta)}$ and $\|\psi_j\|_{C^2(A_\delta)}$, we obtain
	\begin{align*}
		&|\mathcal{F}(\GW(\mu))[w_1,\dots,w_k]|\\
		&\quad\le
		4^{3k}C_k\|\mu\|_{A,\delta}C_{A_\delta}^k\prod_{j=1}^k(1+|w_j|)^2\exp\left(\sum_{j=1}^{k}h_{A_\delta}(\Im w_j)\right),
	\end{align*}
	which shows the claim.
\end{proof}

\begin{corollary}
	\label{corollary:estimateFourier_noDelta}
	Let $A\subset \R^n$ be compact and convex, $1\le k\le n$. There exists a constant $C(A,k)>0$ depending on $A$, $k$ only such that for $\mu\in\VConv_k(\R^n)$ with $\supp\mu\subset A$
	\begin{align*}
		|\mathcal{F}(\GW(\mu))[w_1,\dots,w_k]|\le C(A,k)\|\mu\|_{A,1}\left(\prod_{j=1}^k\left(1+|w_j|\right)\right)^3e^{\sum_j h_{A}(\Im(w_j))}.
	\end{align*}
\end{corollary}
\begin{proof}
	Fix $w_1,\dots,w_k\in\C^n$ and set $\delta:=\max((1+|w_j|)^{-1}, 1\le j\le k)$. Let $A\subset\R^n$ be a compact and convex subset and $\mu\in\VConv_k(\R^n)$ a valuation with $\supp\mu\subset A$. Since $\delta\le1$, \cite[Corollary 6.9]{Knoerrsupportduallyepi2021} implies
	\begin{align*}
		\|\mu\|_{A,\delta}\le \frac{2^k}{\delta^k}\left(\diam A+3\right)^k\|\mu\|_{A,1}.
	\end{align*}
	We combine this estimate with Proposition~\ref{proposition:estimateFourierGW} to obtain a constant $C(A,k)>0$ depending on $A,k$ only such that
	\begin{align*}
			&|\mathcal{F}(\GW(\mu))[w_1,\dots,w_k]|\\
		&\qquad \le\frac{1}{\delta^k} C(A,k) \|\mu\|_{A,1}\left(\prod_{j=1}^k\left(1+|w_j|\right)\right)^2e^{\sum_j h_{A}(\Im(w_j))}.
	\end{align*}
	Plugging in the definition of $\delta$, we obtain the desired inequality.
\end{proof}
We will also need the following continuity property of this map.
\begin{proposition}
	\label{proposition:ContinuityFourier}
	For every compact subset $A\subset\C^n$ there exists a compact set $K_A\subset\Conv(\R^n,\R)$ such that for every $\mu\in\VConv_k(\R^n)$,
	\begin{align*}
		|\mathcal{F}(\GW(\mu))[w_1,\dots,w_k]|\le 4^kC_k\|\mu\|_{K_A}\quad\text{for all}~w_1,\dots,w_k\in A,
	\end{align*}
	where $C_k$ is the constant from Lemma~\ref{lemma:normPolarizationBoundSupport}. In particular, the map $\mathcal{F}\circ\GW:\VConv_k(\R^n)\rightarrow\mathcal{O}_{\Mat_{n,k}(\C)}$ is continuous if we equip the target space with the topology of uniform convergence on compact subsets of $\Mat_{n,k}(\C)$.
\end{proposition}
\begin{proof}
	As in the proof of Proposition \ref{proposition:estimateFourierGW}, we consider for $z\in \C^n$ the functions
	\begin{align*}
		\phi_z(x):=&\Re \exp\left(-i\langle z,x\rangle\right)=\exp(\langle\Im z,x\rangle) \cos(\langle \Re z,x\rangle),\\
		\psi_z(x):=&\Im \exp\left( -i\langle z,x\rangle\right)=\exp(\langle \Im z,x\rangle) \sin(\langle \Re z,x\rangle),
	\end{align*}
	which satisfy
	\begin{align*}
		\max(|D^2\phi_z(x)|,|D^2 \psi_z(x)|)\le 4\exp(\langle\Im z,x\rangle)|z|^2.
	\end{align*}
	Consider the function $f_z\in\Conv(\R^n,\R)\cap C^\infty(\R^n)$ given by
	\begin{align*}
		f_z(x)= 4\exp\left(|z|\frac{|x|^2+1}{2}\right).
	\end{align*}
	If $I_n$ denotes the identity matrix, then 
	\begin{align*}
		D^2f_z(x)\ge 4|z|^2\exp\left(|z|\frac{|x|^2+1}{2}\right)I_n\ge 4\exp(\langle\Im z,x\rangle)|z|^2 I_n
	\end{align*}
	In particular, $f_z+\phi_z$ and $f_z+\psi_z$ are convex. If $A\subset\C^n$ is compact, set $r(A):=\sup_{z\in A}|z|$. Then the set
	\begin{align*}
		K_A:=\left\{f\in\Conv(\R^n,\R): |f(x)|\le 8\exp\left(r(A)\frac{|x|^2+1}{2}\right)\right\}
	\end{align*}
	is obviously locally uniformly bounded and thus relatively compact by Proposition \ref{proposition:compactnessConv}. Moreover, it contains the functions $f_z,f_z+\phi_z,f_z+\psi_z$ for each $z\in A$. For $w=(w_1,\dots,w_k)\in\Mat_{n,k}(\C)$, we have
	\begin{align*}
		&\mathcal{F}(\GW(\mu))[w]=\GW(\mu)[[\phi_{w_1}+i\psi_{w_1}]\otimes\dots\otimes[\phi_{w_k}+i\psi_{w_k}]]\\
		=&\sum_{j=0}^{k}\frac{i^{k-j}}{j!(k-j)!}\sum_{\sigma\in S_k}\GW(\mu)[\phi_{w_{\sigma(1)}}\otimes\dots\otimes\phi_{w_{\sigma(j)}}\otimes\psi_{w_{\sigma(j+1)}}\otimes\dots\otimes\psi_{w_{\sigma(k)}}].
	\end{align*}
	Let $\tilde{\phi}_1,\dots,\tilde{\phi}_k\in\{\phi_{w_j}:1\le j\le k\}\cup \{\psi_{w_j}:1\le j\le k\}$ be one of the tuples corresponding to one of the permutations, $f_j:=f_{w_j}$ the corresponding function and $h_j:=f_j+\tilde{\phi}_j\in\Conv(\R^n,\R)$. The multilinearity and the defining property of the Goodey--Weil distribution imply
	\begin{align*}
		&\GW(\mu)[\tilde{\phi}_1\otimes\dots\otimes\dots \otimes\tilde{\phi}_k]\\
		=&\sum_{j=0}^{k}\frac{(-1)^{k-j}}{j!(k-j)!}\sum_{\sigma\in S_k}\GW(\mu)[h_{\sigma(1)}\otimes\dots\otimes h_{\sigma(j)}\otimes f_{\sigma(j+1)}\otimes\dots\otimes f_{\sigma(k)}]\\
		=&\sum_{j=0}^{k}\frac{(-1)^{k-j}}{j!(k-j)!}\sum_{\sigma\in S_k}\bar{\mu}(h_{\sigma(1)},\dots, h_{\sigma(j)}, f_{\sigma(j+1)},\dots, f_{\sigma(k)}).
	\end{align*}
	If $w_1,\dots,w_k\in A$, then the two previous equations imply
	\begin{align*}
		|\mathcal{F}(\GW(\mu))[w]|\le& 2^{2k}\sup_{\tilde{f}_1,\dots,\tilde{f}_k\in K_A} |\bar{\mu}(\tilde{f}_1,\dots,\tilde{f}_k)|\le 4^kC_k \|\mu\|_{K_A},
	\end{align*}
	where we used Proposition \ref{lemma:normPolarizationBoundSupport} in the last step. This shows the desired estimate. 
\end{proof}
Recall that the spaces of holomorphic functions we are interested in are modules over $\mathcal{O}_{\C^n}$ with respect to the diagonal action defined in the introduction. In order to simplify the constructions, we will us the coordinate change on $\Mat_{n,k}(\C)$ from Section \ref{section:Prelim_Minors}. Consider the function 
\begin{align*}
	&\F(\mu)[w]:=\frac{k!k^{2k-2}}{(-1)^k}\mathcal{F}(\GW(\mu))\left[\frac{w_1+w_k}{k},\dots,\frac{w_{k-1}+w_k}{k},\frac{w_k}{k}-\sum_{j=1}^{k-1}\frac{w_j}{k}\right].
\end{align*}
We refer to Corollary~\ref{corollary:FRestrictionSubspace} for an explanation for the additional combinatorial factor. Note that we can recover $\mathcal{F}(\GW(\mu))$ from $\F(\mu)$ in the following way:
\begin{align*}
	&\mathcal{F}(\GW(\mu))[w]=\frac{(-1)^k}{k!k^{2k-2}}\F(\mu)\left[kw_1-\sum_{j=1}^kw_j,\dots,kw_{k-1}-\sum_{j=1}^kw_j,\sum_{j=1}^{k}w_j\right].
\end{align*}

As observed in \cite{KnoerrUnitarilyinvariantvaluations2021}, the Fourier--Laplace transform of Goodey--Weil distributions is intimately tied to the pushforward of valuations along projections to lower dimensional subspaces, defined as follows: Given a $k$-dimensional subspace $E\in\Gr_k(\R^n)$, let $\pi_E:\R^n\rightarrow E$ denote the orthogonal projection. Given $\mu\in\VConv_k(\R^n)$, we define a valuation $\pi_{E*}\mu\in\VConv_k(E)$ by
\begin{align*}
	\pi_{E*}\mu(f)=\mu(\pi_E^*f)\quad\text{for}~f\in\Conv(E,\R).
\end{align*}
Note that $\pi_{E*}\mu$ is a valuation of top degree if $\mu$ is $k$-homogeneous. These valuations were completely classified by Colesanti, Ludwig, and Mussnig \cite[Theorem 5]{ColesantiEtAlhomogeneousdecompositiontheorem2020}, and are all of the form \eqref{equation:ValuationsMAOperators}, which leads to the following relation to the Fourier--Laplace transform of the associated Goodey--Weil distributions.

\begin{lemma}[\cite{KnoerrUnitarilyinvariantvaluations2021} Lemma 2.6]
	\label{lemma:FourierRestrictionSubspace}
	Let $\mu\in\VConv_k(\R^n)$. For $E\in \Gr_k(\R^n)$ let $\phi_E\in C_c(E)$ be the unique function with $\pi_{E*}\mu=\int_{E}\phi_Ed\MA_E$. Then for $w_1,\dots,w_k\in E\otimes\C$,
	\begin{align*}
		\mathcal{F}(\GW(\mu))[w_1,\dots,w_k]=\frac{(-1)^k}{k!}\det(\langle w_i,w_j\rangle)_{i,j=1}^k \mathcal{F}_E(\phi_E)\left[\sum_{j=1}^{k}w_j\right].
	\end{align*}
	Here, $\mathcal{F}_E$ denotes the Fourier transform on $L^1(E)$.
\end{lemma}
The following is a simple calculation.
\begin{corollary}
	\label{corollary:FRestrictionSubspace}
	In the situation in Lemma~\ref{lemma:FourierRestrictionSubspace} we have for $w_1,\dots,w_k\in E\otimes\C$,
	\begin{align*}
		\F(\mu)[w_1,\dots,w_k]=\det(\langle w_i,w_j\rangle)_{i,j=1}^k \mathcal{F}_E(\phi_E)\left[w_k\right].
	\end{align*}
\end{corollary}
Corollary~\ref{corollary:estimateFourier_noDelta} implies the following estimate.
\begin{corollary}
	\label{corollary:EstimateFMu}
	For every compact and convex subset $A\subset\R^n$ there exists a constant $C(A,k)$ such that for all $\mu\in\VConv_k(\R^n)$ with $\supp\mu\subset A$,
	\begin{align*}
		|\F(\mu)[w_1,\dots,w_k]|\le& C(A,k)\|\mu\|_{A,1}\left(\prod_{j=1}^{k-1}\left(1+|w_j|\right)\right)^6\\
		 &(1+|w_k|)^{3k}e^{h_{A}(\Im(w_k))+\frac{1}{k}\sum_{j=1}^{k-1}h_A(\Im w_j)+h_A(-\Im w_j)}.
	\end{align*}
\end{corollary}
\begin{proof}
	Applying the coordinate change to the estimate in Corollary~\ref{corollary:estimateFourier_noDelta}, we can estimate the product involving the norms of $w_1,\dots,w_k$ and the exponential function separately. For the first term, we have
	\begin{align*}
		\prod_{j=1}^{k-1}\left(1+\frac{|w_j+w_k|}{k}\right)\left(1+\frac{|w_k-\sum_{j=1}^{k-1}w_j|}{k}\right)\le \prod_{j=1}^{k-1}\left(1+|w_j|\right)^2 (1+|w_k|)^k.
	\end{align*}
	We can estimate the argument of the resulting exponential function using the subadditivity of the support function $h_A$:
	\begin{align*}
		&\sum_{j=1}^{k-1}h_A\left(\frac{\Im(w_j)+\Im(w_k)}{k}\right)+h_A\left(\frac{\Im(w_k)}{k}-\sum_{j=1}^{k-1}\frac{\Im(w_j)}{k}\right)\\
		\le&\frac{1}{k}\sum_{j=1}^{k-1}\left(h_A(\Im(w_j))+h_A(\Im(w_k))\right)+\frac{1}{k}h_A(\Im(w_k))+\frac{1}{k}\sum_{j=1}^{k-1}h_A(-\Im(w_j))\\
		=&h_{A}(\Im(w_k))+\frac{1}{k}\sum_{j=1}^{k-1}h_A(\Im w_j)+h_A(-\Im w_j).
	\end{align*}
	Since the exponential function is monotone, this implies the desired result.
\end{proof}

\subsection{Relation to translation equivariant Monge--Amp\`ere operators}
For $\Psi\in \MAVal_k(\R^n)$ and $\phi\in C_c(\R^n)$, we denote by $\Psi[\phi]\in\VConv_k(\R^n)$ the valuation defined by
\begin{align*}
	\Psi[\phi](f)=\int_{\R^n}\phi(x)d\Psi(f;x)\quad\text{for}~f\in\Conv(\R^n,\R).
\end{align*}
Note that this is well defined due to Theorem~\ref{theorem:measuresDifferentialCycle}. In this section we will establish some properties of these valuations which are used in the next section. Starting point is the following observation.
\begin{theorem}[\cite{KnoerrMongeAmpereoperators2024} Theorem 6.16 and Corollary 6.17]
	\label{theorem:FourierMA}
	For every $\Psi\in \MAVal_k(\R^n)$ there exists a unique $Q(\Psi)\in \M_k^2$ such that for all $\phi\in C_c(\R^n)$ we have
	\begin{align*}
		\mathcal{F}(\GW(\Psi[\phi]))[w_1,\dots,w_k]=\frac{(-1)^k}{k!}Q(\Psi)[w_1,\dots,w_k] \mathcal{F}(\phi)\left[\sum_{j=1}^k w_j\right].
	\end{align*}
\end{theorem}
Note that this implies for $\Psi\in \MAVal_k(\R^n)$, $\phi\in C_c(\R^n)$
\begin{align}
	\label{equation:FforMAValuations}
	\F(\Psi[\phi])[w]=Q(\Psi)[w] \mathcal{F}(\phi)\left[w_k\right],
\end{align}
where we used that $Q(\Psi)$ is a sum of quadratic products of $k$-minors, i.e. of two skew symmetric multilinear functionals. We will also need the following result.
\begin{lemma}[\cite{KnoerrMongeAmpereoperators2024} Corollary 6.19]
	\label{lemma:QBijectiveMA}
	The map
	\begin{align*}
		\MAVal_k(\R^n)&\rightarrow \M^2_k\\
		\Psi&\mapsto Q(\Psi)
	\end{align*}
	is bijective.
\end{lemma}
We will need the following result to obtain suitable valuations with a given power series expansion for the Fourier--Laplace transform of their Goodey--Weil distributions.
\begin{proposition}
	\label{proposition:prescribedFourier}
	Let $P\in\P(\C^n)$ be a polynomial of degree at most $N$. There exists $\phi\in C^\infty_c(\R^n)$ such that the power series expansion of $\mathcal{F}(\phi)$ coincides with $P$ up to order $N$.
\end{proposition}	
\begin{proof}
	The coefficient in the power series expansion of $\mathcal{F}(\phi)[z]$ in front of $z^\alpha$ for $\alpha\in \mathbb{N}^n$ (with respect to the standard monomial basis) is given by 
	\begin{align*}
		\frac{1}{\alpha!}\frac{\partial^\alpha}{\partial z^\alpha}\Big|_0\mathcal{F}(\phi)[z]=\mathcal{F}((-i)^{|\alpha|}x^\alpha\phi)[0]=\frac{(-i)^{|\alpha|}}{\alpha!}\int_{\R^n}x^\alpha \phi(x)dx.
	\end{align*}
	Let $\P(\C^n)_{N}$ denote the space of polynomials of degree at most $N$. The linear map $C_c(\R^n)\rightarrow \left(\P(\C^n)_{N}\right)'$ given by
	\begin{align}
		\label{eq:propositionPrescribedFourierTransformFinite}	
		\phi&\mapsto \left[Q\mapsto \int_{\R^n}Q(x) \phi(x)dx\right]
	\end{align}
	is surjective, since for every $Q\in \P(\C^n)_{N}$ we may choose $\phi(x)=\overline{Q(x)}\psi(|x|)$ for some positive $\psi\in C_c([0,\infty))$ to obtain a linear functional that does not vanish on $Q$. Since the monomials $z^\alpha$, $\alpha\in\mathbb{N}^n$ with $|\alpha|\le N$, form a basis of $\P(\C^n)_{N}$, we may for $P(z)=\sum_{|\alpha|\le N}c_\alpha z^\alpha$ define a linear functional on $\P(\C^n)_{N}$ by
	\begin{align*}
		T(z^\alpha)=\frac{\alpha!}{(-i)^{|\alpha|}}c_\alpha \quad\text{for all}~ \alpha\in\mathbb{N}, ~ |\alpha|\le N.
	\end{align*}
	Any function $\phi$ that is mapped to $T$ under \eqref{eq:propositionPrescribedFourierTransformFinite} has the desired property.
\end{proof}

\begin{lemma}
	\label{lemma:PrescribedFourierTransformFinite}
	Let $P\in \P(\C^n)$ be a polynomial of degree at most $N\in\mathbb{N}$, $Q\in \M^2_k$, and let $\Psi_{Q}\in \MAVal_k(\R^n)$ be the unique element with $Q(\Psi_Q)=Q$. There exists $\phi\in C_c(\R^n)$ such that the power series expansion of $\F(\Psi_Q[\phi])$ coincides with $w\mapsto Q(w)P(w_k)$ up to order $2k+N$.
\end{lemma}
\begin{proof}
	For $\phi\in C_c(\R^n)$ and $\Psi\in\MAVal_k(\R^n)$, \eqref{equation:FforMAValuations} implies
	\begin{align*}
		\F(\Psi[\phi])[w]= Q(w)\mathcal{F}(\phi)[w_k].
	\end{align*}
	It is thus sufficient to find $\phi\in C_c(\R^n)$ such that the power series expansion of $\mathcal{F}(\phi)$ is given by $P$ up to order $N$, which exists due to Proposition~\ref{proposition:prescribedFourier}.
\end{proof}

\subsection{Decomposition of the Fourier--Laplace transform of Goodey--Weil distributions}
	\label{section:DecompositionFourierGW}
	
	Let $\M\subset \P(\Mat_{n,k}(\C))$ denote the subspace spanned by the lowest order terms of the power series expansion of $\F(\mu)$ for $\mu\in \VConv_k(\R^n)$. The goal of this section is to show that $\M$ coincides with the $\P(\C^n)$-module generated by $\M^2_k$ (where we use the module structure considered in Section \ref{section:Prelim_Minors}) and to show that every homogeneous term in the power series expansion of $\F(\mu)$ belongs to $\M$ in order to apply the results in Section \ref{section:DivisionAlg}. We start with the following observation.
	\begin{lemma}
		\label{lemma:SubmoduleGeneratedBySmoothValuations}
		$\M$ is spanned by the lowest order terms of $\F(\mu)$ for smooth valuations $\mu\in \VConv_k(\R^n)$.
	\end{lemma}
	\begin{proof}
		Let $\phi\in C^\infty_c(\R^n)$ be a function with $\int_{\R^n}\phi(x)dx=1$ and set $\phi_\epsilon(x)=\epsilon^{-n}\phi(\frac{x}{\epsilon})$ for $\epsilon>0$, $x\in\R^n$. For $\mu\in \VConv_k(\R^n)$ set
		\begin{align*}
			\mu_\epsilon(f):=\int_{\R^n}\phi_\epsilon(x)\mu(f(\cdot+x))dx.
		\end{align*}
		Then $\mu_\epsilon$ is a smooth valuation by Lemma~\ref{lemma:MollifiedValuations}. From the definition of the Fourier--Laplace transform of $\GW(\mu)$ we obtain
		\begin{align*}
			&\mathcal{F}(\GW(\mu_\epsilon))[w]\\
			=&\int_{\R^n}\phi_\epsilon(x)\GW(\mu)[\exp(-i\langle w_1,\cdot+x\rangle)\otimes\dots\otimes \exp(-i\langle w_1,\cdot+x\rangle)]dx\\
			=&\int_{\R^n}\phi_\epsilon(x)\exp\left(-i\sum_{j=1}^k\langle w_j,x\rangle\right)dx\cdot\mathcal{F}(\GW(\mu))[w],
		\end{align*} so from the definition of $\F$, we obtain
		\begin{align}
			\label{equation:Fourier_convolutionSmoothFunction}
			\F(\mu_\epsilon)[w_1,\dots,w_k]=&\mathcal{F}(\phi_\epsilon)\left[w_k\right]\cdot \F(\mu)[w_1,\dots,w_k].
		\end{align}
		As $\mathcal{F}(\phi_\epsilon)[0]=\int_{\R^n}\phi(x)dx=1$, we see that the lowest degree terms of the power series expansion of $\F(\mu_\epsilon)$ and $\F(\mu)$ coincide, which shows the claim.
	\end{proof}
	\begin{corollary}
		\label{corollary:LowestOrderTermsModuleOverPolynomials}
		$\M$ is a module over $\P(\C^n)$. In particular, $\M$ is generated by homogeneous polynomials.
	\end{corollary}
	\begin{proof}
		Since $\M$ is spanned by the lowest order terms of smooth valuations by Lemma~\ref{lemma:SubmoduleGeneratedBySmoothValuations}, it is sufficient to show that the product of such a term with $w_{j,k}$, $1\le j\le n$, is a lowest order term. Any such term coincides with the lowest order term of $\F(\mu_\epsilon)$ for $\mu_\epsilon$ defined for some $\mu\in\VConv_k(\R^n)$ as in the proof of Lemma~\ref{lemma:SubmoduleGeneratedBySmoothValuations}. Then
		\begin{align*}
			w_{j,k}\F(\mu_\epsilon)[w]=&\left(w_{j,k}\mathcal{F}(\phi_\epsilon)\left[w_k\right]\right)\cdot \F(\mu)[w_1,\dots,w_k]\\
			=&\mathcal{F}(-i\partial_j\phi_\epsilon)[w_k]\cdot \F(\mu)[w_1,\dots,w_k]\\
			=&\F(\tilde{\mu}_\epsilon)
		\end{align*}
		for the valuation $\tilde{\mu}_\epsilon\in\VConv_k(\R^n)$ defined by
		\begin{align*}
			\tilde{\mu}_\epsilon(f)=-i\int_{\R^n}\partial_j\phi_\epsilon(x)\mu(f(\cdot+x))dx.
		\end{align*}
		In particular, the lowest order term of the  power series expansion of $\F(\tilde{\mu}_\epsilon)$ is given by the product of $w_{j,k}$ with the lowest order term in the power series expansion of $\F(\mu_\epsilon)$.
	\end{proof}

	Recall that $\M_k$ denotes the space spanned by the $k$-minors. For $\Delta\in \M_k$ consider the sets $U_\Delta=\{E\in {}^\C\Gr_k(\C^n): \Delta|_{E^k}\ne 0\}$ from Section \ref{section:Prelim_Minors}.
	\begin{lemma}
		\label{lemma:RestrictionLowestOrderTermsSubspaces}
		Let $P\in\M$. For every $\Delta\in \M_k$, $E_0\in U_\Delta\cap \{E\otimes \C: E\in\Gr_k(\R^n)\}$ there exists a polynomial $P_{\Delta,E_0}$ such that 
		\begin{align*}
			P|_{E_0}(w_1,\dots,w_k)=\Delta^2(w_1,\dots,w_k) P_{\Delta,E_0}(w_k).
		\end{align*}
	\end{lemma}
	\begin{proof}
		Since $M$ is a homogeneous module over $\P(\C^n)$ by Corollary~\ref{corollary:LowestOrderTermsModuleOverPolynomials}, we can restrict ourselves to the case that $P$ is homogeneous and the lowest order term of $\F(\mu)$ for some $\mu\in\VConv_k(\R^n)$. Let $E\in\Gr_k(\R^n)$ with $E\otimes \C\in U_\Delta$ be given. In this case, Corollary~\ref{corollary:FRestrictionSubspace} shows that for $w_1,\dots,w_k\in E\otimes\C$,
		\begin{align*}
			\F(\mu)[w]=\det(\langle w_i,w_j\rangle)_{i,j=1}^k \mathcal{F}_{E}(\phi_{E})\left[w_k\right],
		\end{align*}
		where $\phi_E\in C_c(E)$ denotes the unique function with $\pi_{E*}\mu=\int_{E}\phi_{E}d\MA_{E}$. In particular, since $P$ is homogeneous,
		\begin{align*}
			P|_{E\otimes\C}(w)=\det(\langle w_i,w_j\rangle)_{i,j=1}^k Q_{E}(w_k)\quad\text{for}~w_1,\dots,w_k\in E\otimes\C,
		\end{align*}
		where $Q_{E}$ denotes the lowest order term of $\mathcal{F}_{E}(\phi_{E})\left[w_k\right]$. If we choose a basis of $E$, then the polynomial $\Delta^2$ restricts to the square of the determinant of the coordinate matrix of $w\in E^k$ with respect to this basis. Since we assume that $E\otimes\C\in U_\Delta$, this multiple is nontrivial . Similarly, the restriction of the polynomial
		$G(w):= \det(\langle w_i,w_j\rangle)_{i,j=1}^k$ to $(E\otimes \C)^k$ vanishes on the zero set of this determinant, since the Gram matrix is singular if its arguments are linearly dependent. Since the determinant is an irreducible polynomial, $G|_{(E\otimes \C)^k}$ is therefore divisible by the determinant. However, since the determinant is skew-symmetric and multilinear, the quotient is a skew-symmetric and multilinear, and thus a multiple of the determinant as well. Thus $\Delta^2$ and $G$ both restrict to nontrivial scalar multiples of the same polynomial on $(E\otimes \C)^k$. Rescaling $Q_{E}$, we see that $P$ has the desired property.
	\end{proof}

\begin{corollary}
	\label{corollary:LowestOrderTermsKMinors}
	$\M$ coincides with the $\P(\C^n)$-module generated by $\M^2_k$. 
\end{corollary}
\begin{proof}
	Lemma~\ref{lemma:PrescribedFourierTransformFinite} and Corollary~\ref{corollary:LowestOrderTermsModuleOverPolynomials} show that $\P(\C^n)\M^2_k\subset \M$. Combining Lemma~\ref{lemma:RestrictionLowestOrderTermsSubspaces} with Lemma~\ref{lemma:PolynomialFromComplexifiedSpaces} and Proposition~\ref{proposition:CharacterizationModuleSquareMinors}, we see that $\M\subset \P(\C^n)\M^2_k$.
\end{proof}
\begin{lemma}
	\label{lemma:HomogeneousComponentsPowerSeriesSquaresMinors}
	Let $\mu\in\VConv_k(\R^n)$. Every homogeneous component of the power series expansion of $\F(\mu)$ belongs to $\P(\C^n)\M^2_k$.
\end{lemma}
\begin{proof}
	Assume that this is not the case. Then we can find a maximal $N\in\mathbb{N}$ such that power series expansion of $\F(\mu)$ up to order $N$ belongs to $\P(\C^n)\M^2_k$, but the power series expansion up to order $N+1$ does not belong to $\P(\C^n)\M^2_k$. Note that $N\ge 2k$ by Corollary~\ref{corollary:LowestOrderTermsKMinors}, since the lowest order term of the expansion belongs to $\M=\P(\C^n)\M^2_k$ by definition. Let $Q\in \P(\Mat_{n,k}(\C))$ denote the power series expansion of $\F(\mu)$ up to order $N$. Fix a basis $Q_j$, $1\le j\le N_{n,k}$, of $\M^2_k$. By assumption, we may write $Q(w)=\sum_{j=1}^{N_{n,k}}P_{j}(w_k)Q_j(w)$ with $P_{j}\in \P(\C^n)$ of order at most $N-2k$. By Lemma~\ref{lemma:PrescribedFourierTransformFinite} we find valuations $\mu_{j}\in\VConv_k(\R^n)$ such that power series expansion of $\F(\mu_{j})$ coincides with $Q_j(w)P_{j}(w_k)$ up to order $N+1$. In particular, the lowest order term in the power series expansion of 
	\begin{align*}
		\F\left(\mu-\sum_{j=1}^{N_{n,k}}\mu_{j}\right)=\F(\mu)-\sum_{j=1}^{N_{n,k}}\F(\mu_{j})
	\end{align*}
	coincides with the homogeneous term in the power series expansion of $\F(\mu)$ of degree $N+1$. Thus this term belongs to $\M$, which coincides with $\P(\C^n)\M^2_k$ by Corollary~\ref{corollary:LowestOrderTermsKMinors}, and we obtain a contradiction.
\end{proof}
	
	\begin{theorem}
		\label{theorem:RepresentationFMu}
		Let $\Psi_j$, $1\le j\le N_{n,k}$, be a basis for $\MAVal_k(\R^n)$. For every $\mu\in\VConv_k(\R^n)$ there exist functions $g_{j}\in \mathcal{O}_{\C^n}$ such that
		\begin{align*}
			\F(\mu)[w_1,\dots,w_k]=\sum_{j=1}^{N_{n,k}}g_j(w_k)Q(\Psi_j)[w].
		\end{align*}
		Moreover, the functions $g_j$ can be chosen to satisfy the inequality
		\begin{align}
			\label{eq:estimateFourierCoefficients}
			|g_j(z)|\le C_\delta\left(1+|z|\right)^{(2+n)N_{n,k}}\sup_{\zeta\in D_{\delta(0,\dots,0,z)}} |\F(\mu)[\zeta]|,
		\end{align}
		where $C_\delta$ is a constant depending on the chosen basis, $n$, $k$, and $\delta>0$ only.
	\end{theorem}
	\begin{proof}
		Obviously, the claim is true for an arbitrary basis as soon as it holds for some basis of $\MAVal_k(\R^n)$. Using Lemma~\ref{lemma:QBijectiveMA}, we can choose the basis corresponding to the Gröbner basis of $\mathcal{P}(\C^n)\M^2_k$ used in Theorem~\ref{theorem:GroebnerAlgModuleMinor}. By Lemma~\ref{lemma:HomogeneousComponentsPowerSeriesSquaresMinors}, every homogeneous term of the power series expansion of $\F(\mu)$ belongs to $\P(\C^n)\M^2_k$, so in this case, the claim follows directly from Theorem~\ref{theorem:GroebnerAlgModuleMinor}.
	\end{proof}

	\begin{remark}
		\label{remark:Decomposition}
		For a given basis $\Psi_j$, $1\le j\le N_{n,k}$, of $\MAVal_k(\R^n)$ and $\mu\in\VConv_k(\R^n)$ with decomposition  $\F(\mu)[w]=\sum_{j=1}^{N_{n,k}}g_{j}(w_k)Q(\Psi_j)[w]$, there do in general not exist valuations $\mu_{j}\in\VConv_k(\R^n)$ with $\F(\mu_{j})[w]=g_j(w_k)Q(\Psi_j)[w] $. In particular, the decomposition of $\F(\mu)$ in Theorem~\ref{theorem:RepresentationFMu} does not reflect a decomposition of $\VConv_k(\R^n)$.
	\end{remark}

	\begin{proof}[Proof of Theorem~\ref{maintheorem:FourierTransformBelongsToModule}]
		For $\mu\in \VConv_k(\R^n)$ and a basis $\Psi_j$, $1\le j\le N_{n,k}$, of $\MAVal_k(\R^n)$, Theorem~\ref{theorem:RepresentationFMu} implies that there are functions $g_j\in \mathcal{O}_{\C^n}$ such that
		\begin{align*}
			\F(\mu)[w]=\sum_{j=1}^{N_{n,k}}g_j(w_k)Q(\Psi_j)[w].
		\end{align*}
		Thus
		\begin{align*}
			\mathcal{F}(\GW(\mu))[w]=&\frac{(-1)^k}{k!k^{2k-2}}\F(\mu)\left[kw_1-\sum_{j=1}^kw_j,\dots,kw_{k-1}-\sum_{j=1}^kw_j,\sum_{j=1}^kw_j\right]\\
			=&\sum_{j=1}^{N_{n,k}}g_j\left(\sum_{j=1}^kw_j\right)Q(\Psi_j)[w],
		\end{align*}
		which belongs to $\widehat{\M^2_k}$.
	\end{proof}

\section{Paley--Wiener--Schwartz Theorem for smooth valuations}
	\label{section:PWSValuations}
	In this section we prove Theorem~\ref{maintheorem:PWSSmoothValuations}. The main ingredient is the following estimate for smooth valuations.	
	\begin{lemma}
		\label{lemma:PWS_SatisfiedBySmoothValuations}
		Let $A\subset \R^n$ be compact and convex and let $\mu\in\VConv_k(\R^n)$ be a smooth valuation with $\supp\mu\subset A$. For every $N\in\mathbb{N}$ there exists a constant $C_N>0$ such that 
		\begin{align*}
			&|\F(\mu)[w_1,\dots,w_k]|\\ 
			&\le C_N\prod_{j=1}^{k-1}(1+|w_j|)^6(1+|w_k|)^{-N}e^{h_{A}(\Im(w_k))+\frac{1}{k}\sum_{j=1}^{k-1}h_A(\Im w_j)+h_A(-\Im w_j)}.
		\end{align*}
	\end{lemma}
	\begin{proof}
		By definition, the map 
		\begin{align*}
			\R^n&\rightarrow \VConv_k(\R^n)\\
			x&\mapsto \left[f\mapsto \mu(f(\cdot+x))\right]
		\end{align*} is smooth.  Set $\mu_x(f):=\mu(f(\cdot+x))$. Then
		\begin{align*}
			\mathcal{F}(\GW(\mu_x))[w]=&\GW(\mu)[\exp(-i\langle w_1,\cdot+x\rangle)\otimes\dots\otimes\exp(-i\langle w_k,\cdot+x\rangle)]\\
			=&\exp\left(-i\left\langle\sum_{j=1}^{k}w_j,x\right\rangle \right)\mathcal{F}(\GW(\mu))[w].
		\end{align*}	
		In particular,
		\begin{align*}
			&\mathcal{F}(\GW(\mu_x))\left[\frac{w_1+w_k}{k},\dots,\frac{w_{k-1}+w_k}{k},\frac{w_k}{k}-\sum_{j=1}^{k-1}\frac{w_j}{k}\right]\\
			&\qquad=\exp\left(-i\left\langle w_k,x\right\rangle \right)\mathcal{F}(\GW(\mu))\left[\frac{w_1+w_k}{k},\dots,\frac{w_{k-1}+w_k}{k},\frac{w_k}{k}-\sum_{j=1}^{k-1}\frac{w_j}{k}\right],
		\end{align*}
		so
		\begin{align*}
			\exp\left(-i\left\langle w_k,x\right\rangle \right)\F(\mu)[w]=\F(\mu_x)[w].
		\end{align*}
		Applying $\partial_{x}^{\alpha}$ for $\alpha\in\mathbb{N}^n$ to both sides and using that $\mathcal{F}\circ \GW$ is continuous by Proposition \ref{proposition:ContinuityFourier}, this implies
		\begin{align*}
			(-i)^{|\alpha|}w_k^\alpha\exp\left(-i\left\langle w_k,x\right\rangle \right)\F(\mu)[w]=&\F(\partial_{x}^{\alpha}\mu_x)[w].
		\end{align*}
		If we evaluate this expression in $x=0$ and sum over the appropriate indices $\alpha$ with $|\alpha|= N$, we obtain 
		\begin{align*}
			|w_k|^{N}\cdot |\F(\mu)[w]|\le |w_k|_1^{N}\cdot |\F(\mu)[w]|\le& n^{N}\max_{|\alpha|= N}\left|\F(\partial^\alpha\mu_x|_0)\left[w\right]\right|.
		\end{align*}
		Using the binomial theorem to write $(1+|w_k|)^N=\sum_{j=0}^N\binom{N}{j}|w_k|^j$, this implies
		\begin{align*}
			(1+|w_k|)^N|\F(\mu)[w]|
			\le& 2^{N}n^{N}\max_{|\alpha|\le N}\left|\F(\partial_{x}^{\alpha}\mu_x|_0)\left[w\right]\right|.
		\end{align*}
		From Corollary~\ref{corollary:EstimateFMu}, we obtain a constant $C(A,k)$ depending on $k$ and $A\subset\R^n$ only such that for $N\in\mathbb{N}$,
		\begin{align*}
			|\F(\mu)[w]|\le& (1+|w_k|)^{-N}(2n)^{N}\max_{|\alpha|\le N}|\F(\partial_{x}^{\alpha}\mu_x|_0)\left[w\right]|\\
			\le&(1+|w_k|)^{-N+3k} (2n)^{N}C(A,k)\prod_{j=1}^{k-1}(1+|w_j|)^6  \cdot \max_{|\alpha|\le N}\|\partial_{x}^{\alpha}\mu_x|_0\|_{A,1}\\
			&e^{h_{A}(\Im(w_k))+\frac{1}{k}\sum_{j=1}^{k-1}h_A(\Im w_j)+h_A(-\Im w_j)},
		\end{align*}
		which shows the desired estimate for the exponent $N-3k$.
	\end{proof}
	\begin{remark}
		The proof shows that the constant $C_N$ can be chosen to be
		\begin{align*}
			C_N=(2n)^{N+3k}C(A,k)\max_{|\alpha|\le N+3k}\|\partial_{x}^{\alpha}\mu_x|_0\|_{A,1}
		\end{align*}
		for the constant $C(A,k)$ in Corollary~\ref{corollary:EstimateFMu}.	With this constant, the estimate in Lemma~\ref{lemma:PWS_SatisfiedBySmoothValuations} holds as soon as the map $x\mapsto \pi(x)\mu$ is of class $C^{N+3k}$.
	\end{remark}
	The following may be seen as a version of Theorem \ref{maintheorem:PWSSmoothValuations} for the functions $\F(\mu)$ for smooth valuations $\mu\in\VConv_k(\R^n)$.
	\begin{theorem}
		\label{theorem:PWSValuations}
		Let $F\in\O_{\C^n}\M^2_k$, $A\subset \R^n$ compact and convex, and assume that for every $N\in\mathbb{N}$ there exists $\delta_N>0$ and $C_N\in\mathbb{N}$ such that for $w_1,\dots,w_k\in \C^n$ with $|w_1|\le \delta_N$, \dots, $|w_{k-1}|\le\delta_N$,
		\begin{align}
			\label{equation:theoremPWSValuations}
			|F(w_1,\dots,w_k)|\le C_N(1+|w_k|)^{-N}e^{h_A(\Im(w_k))}.
		\end{align}
		Then $F=\F(\mu)$ for a unique smooth valuation $\mu\in\VConv_k(\R^n)$ with $\supp\mu\subset A$. More precisely, given any basis $\Psi_j$, $1\le j\le N_{n,k}$, of $\MAVal_k(\R^n)$ there exist functions $\phi_j\in C_c^\infty(\R^n)$ with $\supp \phi_j\subset A$ such that $\mu$ is given by
		\begin{align*}
			\mu(f)=\sum_{j=1}^{N_{n,k}}\int_{\R^n}\phi_jd\Psi_j(f)\quad\text{for all}~f\in\Conv(\R^n,\R).
		\end{align*}
	\end{theorem}
	\begin{proof}
		Assume that $F\in\O_{\C^n}\M^2_k$ satisfies these estimates and let $\Psi_j$, $1\le j\le N_{n,k}$,  be a basis of $\MAVal_k(\R^n)$. Theorem~\ref{theorem:GroebnerAlgModuleMinor} allows us to find $g_j\in \mathcal{O}_{\C^n}$ such that
		\begin{align*}
			F(w_1,\dots,w_k)=\sum_{j=1}^{N_{n,k}}g_j(w_k)Q(\Psi_j)[w],
		\end{align*}
		where $g_j\in \mathcal{O}_{\C^n}$ satisfies for $\delta>0$
		\begin{align*}
			|g_j(z)|\le C_{\delta}(1+|z|)^{(2+n)N_{n,k}}\sup_{D_\delta(0,\dots,0,z)}|F(\zeta)|
		\end{align*}
		for some constant $C_{\delta}$ depending on $\delta>0$ and the chosen basis only. Since $F$ satisfies \eqref{equation:theoremPWSValuations}, we obtain for every $N\in\mathbb{N}$ a constant $\tilde{C}_N>0$ such that
		\begin{align*}
			|g_j(z)|\le \tilde{C}_N(1+|z|)^{-N}e^{h_A(\Im(z))}.
		\end{align*}
		In particular, the functions $g_j\in \mathcal{O}_{\C^n}$ satisfy the condition of the Paley--Wiener--Schwartz Theorem~\ref{theorem:PaleyWienerSchwartz}, so there exist $\phi_j\in C^\infty_c(\R^n)$ with $\supp\phi_j\subset A$ and $g_j=\mathcal{F}(\phi_j)$. Consider the valuation $\mu\in\VConv_k(\R^n)$ given by
		\begin{align*}
			\mu(f)=\sum_{j=1}^{N_{n,k}}\int_{\R^n}\phi_jd\Psi_j(f)\quad\text{for}~f\in\Conv(\R^n,\R).
		\end{align*}
		Then $\mu$ is a smooth valuation by Lemma~\ref{lemma:SmoothValTranslationSmooth}, and it is easy to see that $\supp\mu\subset A$. Moreover, from Theorem~\ref{theorem:FourierMA} (and \eqref{equation:FforMAValuations}) and the definition of $\F(\mu)$ we deduce that
		\begin{align*}
			\F(\mu)[w]=&\sum_{j=1}^{N_{n,k}}\mathcal{F}(\phi_j)[w_k]Q(\Psi_j)[w]=F(w).
		\end{align*}
		The claim follows.
	\end{proof}

\begin{proof}[Proof of Theorem~\ref{maintheorem:PWSSmoothValuations}]
	Let us first show that smooth valuations satisfy the desired inequality. If $\mu\in\VConv_k(\R^n)$, then $\F(\mu)\in \O_{\C^n}\M^2_k$, and thus for $\delta>0$ 
	\begin{align*}
		|\F(\mu)[w]\le C_\delta \left(\prod_{j=1}^{k}|w_j|^2\right)\left(1+|w_k|\right)^{(2+n)N_{n,k}}\sup_{\zeta\in D_{\delta(0,\dots,0,w_k)}}|\F(\mu)[\zeta]|
	\end{align*}
by Corollary 
\ref{corollary:EstimateFourierDiagonalCoordinates}, where $C_\delta>0$ is a constant depending on $\delta>0$ only. If $\mu$ is in addition a smooth valuation, we can combine this estimate with the estimate in Lemma~\ref{lemma:PWS_SatisfiedBySmoothValuations} to obtain for any $N\in\mathbb{N}$, $\delta>0$ a constant $C_{N,\delta}>0$  such that 
\begin{align*}
	|\F(\mu)[w_1,\dots,w_k]\le C_{N,\delta}\left(\prod_{j=1}^{k-1}|w_j|^2\right) (1+|w_k|)^{-N}e^{h_A(\Im(w_k))}.
\end{align*} 
Since we can recover $\mathcal{F}(\GW(\mu))$ from $\F(\mu)$ using the relation
\begin{align*}
	&\mathcal{F}(\GW(\mu))[w]=\frac{(-1)^k}{k!k^{2k-2}}\F(\mu)\left[kw_1-\sum_{j=1}^kw_j,\dots,kw_{k-1}-\sum_{j=1}^kw_j,\sum_{j=1}^kw_j\right],
\end{align*}
this implies the desired inequality.\\
Now assume that $F\in\widehat{\M^2_k}$ satisfies the given estimates for every $N\in\mathbb{N}$. Then $\tilde{F}\in\mathcal{O}_{\Mat_{n,k}(\C)}$ given by
\begin{align*}
	\tilde{F}(w_1,\dots,w_k)=F\left(\frac{w_1+w_k}{k},\dots,\frac{w_{k-1}+w_k}{k},\frac{w_k}{k}-\sum_{j=1}^{k-1}\frac{w_j}{k}\right)
\end{align*}
belongs to $\O_{\C^n}\M^2_k$, and for every $N\in\mathbb{N}$ we obtain a constant $\tilde{C}_N>0$ such that
\begin{align*}
	|\tilde{F}(w_1,\dots,w_k)|\le &\tilde{C}_N\prod_{j=1}^{k-1}(1+|w_j|)^2(1+|w_k|)^{-N}e^{h_A(\Im(w_k))}.
\end{align*}
In particular, $\tilde{F}$ satisfies the estimate in Theorem~\ref{theorem:PWSValuations}, and we find a smooth valuation $\mu\in\VConv_k(\R^n)$ with $\supp\mu\subset A$ and $\tilde{F}=\F(\mu)$. Unraveling the definitions, this shows that $F=\mathcal{F}(\GW(\mu))$, which completes the proof.
\end{proof}

Theorem~\ref{maintheorem:DescriptionsSmoothValuations} follows from the following slightly refined result.
\begin{theorem}
	\label{theorem:DescriptionSmoothValuations_support}
	Let $A\subset \R^n$ be compact and convex. The following are equivalent for $\mu\in\VConv_k(\R^n)$: 
	\begin{enumerate}
		\item $\mu$ is a smooth valuation and $\supp\mu\subset A$.
		\item For a basis $\Psi_j$ of $\MAVal_k(\R^n)$ there exists functions $\phi_j\in C^\infty_c(\R^n)$ such that $\supp\phi_j\subset A$ and
		\begin{align*}
			\mu(f)=\sum_{j=1}^{N_{n,k}}\int_{\R^n}\phi_jd\Psi_j(f).
		\end{align*}
		\item There exists a differential form $\omega\in \Omega^{n-k}_c(\R^n)\otimes \Lambda^k((\R^n)^*)^*$ with $\supp\omega\subset \pi^{-1}(A)$ such that
		\begin{align*}
			\mu(f)=D(f)[\omega].
		\end{align*}
	\end{enumerate}
\end{theorem}
\begin{proof}
	(2)$\Rightarrow$ (3) is true by Theorem~\ref{theorem:ClassificationMAVal}, while (3) $\Rightarrow$ (1) follows from Lemma~\ref{lemma:SmoothValTranslationSmooth}. Let us show the implication (1) $\Rightarrow$ (2). If (1) holds, then Theorem~\ref{theorem:RepresentationFMu} and Lemma~\ref{lemma:PWS_SatisfiedBySmoothValuations} show that $\F(\mu)$ satisfies the assumptions in Theorem~\ref{theorem:PWSValuations}, which implies the claim.
\end{proof}

\section{Affine invariant subspaces}
	\label{section:AffInariantSubspaces}
	\subsection{Sequentially dense subspaces of $\VConv_k(\R^n)$}
	\label{section:denseSubspaces}
	In this section we prove the following stronger version of Theorem~\ref{maintheorem:WeakIrreducibility} (3). The first two parts are proved in Section \ref{section:affineInvSubspaces}.
	\begin{theorem}
		\label{theorem:weakIrred}
		Let $W\subset \VConv_k(\R^n)$ be an affine invariant subspace such that there exists a positive semi-definite quadratic form $q$ on $\R^n$ and $\mu\in W$ with $\mu(q)\ne 0$. Then $W\subset\VConv_k(\R^n)$ is sequentially dense.
	\end{theorem}
	Let us remark that the version in Theorem~\ref{maintheorem:WeakIrreducibility} (3) admits a rather short proof based on the classification of all closed affine invariant subspaces in Section \ref{section:affineInvSubspaces}, however, the version stated above requires a slightly more careful treatment. The proof presented here relies on the following representation theoretic description of the space $\MAVal_k(\R^n)$.\\
	We define a representation of $\GL(n,\R)$ on $\MAVal_k(\R^n)$ in the following way: For $g\in \GL(n,\R)$ and $\Psi\in\MAVal_k(\R^n)$, $\tilde{\pi}(g)\Psi\in \MAVal_k(\R^n)$ is defined by
	\begin{align*}
		\tilde{\pi}(g)\Psi (f)[B]=\Psi(f\circ g)[g^{-1}B]
	\end{align*}
	for $f\in\Conv(\R^n,\R)$, $B\subset\R^n$ bounded Borel set. Equivalently, for $\phi\in C_c(\R^n)$,
	\begin{align*}
		[\tilde{\pi}(g)\Psi] (f)[\phi]=\Psi(f\circ g)[\phi\circ g].
	\end{align*}
	If we equip $C_c(\R^n)$ with the operation of $\GL(n,\R)$ defined by $g\cdot \phi:=\phi\circ g^{-1}$, this implies that the map
	\begin{align*}
		C_c(\R^n)\otimes \MAVal_k(\R^n)\rightarrow\VConv_k(\R^n)
	\end{align*}
	induced by $(\phi,\Psi)\mapsto \Psi[\phi]$ is $\GL(n,\R)$-equivariant.	\begin{theorem}[\cite{KnoerrMongeAmpereoperators2024} Theorem 1.3]
		\label{theorem:MAVal_Irreducible}
		$\MAVal_k(\R^n)$ is an irreducible representation of $\GL(n,\R)$.
	\end{theorem}
	In the following proof, $\Hess_k\in\MAVal_k(\R^n)$ denotes the $k$th Hessian measure, which is given by $d\Hess_k(f)=[D^2f(x)]_kd\vol_n$ for $f\in \Conv(\R^n,\R)\cap C^2(\R^n)$, where $[D^2f(x)]_k$ denotes the $k$th elementary symmetric function of the eigenvalues of the Hessian of $f$.

	\begin{proof}[Proof of Theorem \ref{theorem:weakIrred}]
		Let us first introduce some notation. For $A\subset\R^n$ compact, set $W_A:=W\cap \VConv_A(\R^n)$, i.e. $W_A$ consists of the valuations in $W$ supported on $A$. Since $\VConv_{k,A}(\R^n)$ is a Banach space, the closure $\overline{W_A}\subset\VConv_{k,A}(\R^n)$ coincides with its sequential closure, so $\overline{W_A}$ is contained in the sequential closure of $W$ for every $A\subset\R^n$ compact. Since every valuation is compactly supported, it is now sufficient to show that $\overline{W_A}=\VConv_{k,A}(\R^n)$ for every compact and convex subset with nonempty interior.\\
		
		Before we continue with the proof, observe that if $g\in\Aff(n,\R)$ and $A,B\subset\R^n$ are compact such that $gA\subset B$, then
		\begin{align*}
			\pi(g):\overline{W_A}\rightarrow\overline{W_B}
		\end{align*}
		is well defined since $W$ is affine invariant and $\Aff(n,\R)$ acts continuously on $\VConv_k(\R^n)$, compare Lemma \ref{lemma:continuityActionAff_VConv}. We will use this observation throughout the proof without explicitly mentioning it. Let us now turn to the proof.\\
		If $\mu\in W$ is a valuation with $\mu(q)\ne 0$, then there exists $\epsilon>0$ such that $\mu(q+\epsilon|\cdot|^2)\ne 0$ by continuity. Since $W$ is invariant under $\GL(n,\R)$, this implies that there exists a valuation $\mu\in W$ with $\mu(|\cdot|^2)\ne0$. Applying a dilation to $\mu$, we may assume that $\supp\mu\subset B_{1/2}(0)$.\\	
		Let $\phi\in C^\infty_c(\R^n)$ be a nonnegative function with $\int_{\R^n}\phi(x)dx=1$ and $\supp\phi\subset B_{1/2}(0)$. Then as in the proof of Lemma \ref{lemma:MollifiedValuations}, the Lebesgue--Bochner intergral
		\begin{align*}
			\mu_\epsilon:=\epsilon^{-n}\int_{\R^n}\phi\left(\frac{y}{\epsilon}\right)\pi(-y)\mu dy
		\end{align*}
		belongs to $\overline{W_{B_1(0)}}$ for $0<\epsilon\le 1$ and converges to $\mu$ for $\epsilon\rightarrow0$, compare Corollary \ref{corollary:ApproxTranslationSmoothValuations}. In particular, there exists $\epsilon\in (0,1)$ such that $\mu_\epsilon(|\cdot|)\ne 0$. By Lemma \ref{lemma:MollifiedValuations}, $\mu_\epsilon$ is a smooth valuation. For this choice of $\epsilon$, we average $\mu_\epsilon$ with respect to the Haar probability measure on $\SO(n)$,
		\begin{align*}
			\mu_0:=\int_{\SO(n)}\pi(g)\mu_\epsilon dg,
		\end{align*}
		to obtain a smooth and $\SO(n)$-invariant valuation $\mu_0$, which satisfies $\mu_0(|\cdot|)=\mu_\epsilon(|\cdot|)\ne 0$. By construction, the integrand belongs to $\overline{W_{B_1(0)}}$ for every $g\in \SO(n)$. Thus $\mu_0\in\overline{W_{B_1(0)}}$ as well.\\
		For $\delta>0$ consider the valuation
		\begin{align*}
			\mu^\delta:=\delta^{-2k}\pi(d_\delta)\mu_0,
		\end{align*}
		where $d_\delta\in\GL(n,\R)$ is the dilation $d_\delta(x)=\delta x$. Then $\mu^\delta\in \overline{W_{B_{\delta}(0)}}$ for every $\delta>0$.\\
		
		By \cite[Theorem 1.3]{KnoerrSingularvaluationsHadwiger2022} there exists $\phi_0\in C_c([0,\infty))$ such that $\mu_0(f)=\int_{\R^n}\phi_0(|x|)d\Hess_k(f)$ for all $f\in\Conv(\R^n,\R)$. Moreover, $\supp\phi_0\subset B_1(0)$ by \cite[Theorem 1.4]{KnoerrSingularvaluationsHadwiger2022}. We obtain 
		\begin{align*}
			0\ne \mu_0(|\cdot|^2)= \int_{\R^n}\phi_0(|x|)d\Hess_k(|\cdot|^2)=c_{n,k}\int_{\R^n}\phi_0(|x|)d\vol_n(x),
		\end{align*}
		so we may assume that $\int_{\R^n}\phi_0(|x|)d\vol_n(x)=1$. For $f\in\Conv(\R^n,\R)\cap C^2(\R^n)$, this implies
		\begin{align*}
			\mu^\delta(f)=&\delta^{-2k}\int_{\R^n}\phi_0(x)d\Hess_k(f\circ d_\delta)=\int_{\R^n}\phi_0(x)[D^2f(\delta x)]_kdx\\
			=&\int_{\R^n}\delta^{-n}\phi_0\left(\frac{x}{\delta}\right)[D^2f(x)]_kdx=\int_{\R^n}\phi_\delta(x)d\Hess_k(f)
		\end{align*}
		for $\phi_\delta(x):=\delta^{-n}\phi_0\left(\frac{x}{\delta}\right)$. By continuity, this holds for every $f\in\Conv(\R^n,\R)$. Note that this implies that $\phi_\delta$ is a smooth approximation of the $\delta$-distribution.
		
		For $A\subset \R^n$ compact and convex with nonempty interior, let $C_A(\R^n)\subset C_c(\R^n)$ denote the subspace of all functions with support contained in $A$. Consider the space
		\begin{align*}
			F_A=\left\{\psi\in C_A(\R^n):\int_{\R^n}\psi d\Hess_k(\cdot)\in \overline{W_A}\right\}.
		\end{align*}
		We claim that $F_A= C_A(\R^n)$. Due to Corollary \ref{corollary:trivialBoundMAOperators}, $F_A\subset C_A(\R^n)$ is closed. Since $A$ is convex with nonempty interior, it is thus sufficient to show that every $\psi\in C_A(\R^n)$ with support contained in the interior of $A$ belongs to $F_A$. Let $\psi$ be such a function and choose $\delta_0>0$ such that $\supp\phi+B_{\delta_0}(0)\subset A$. Then the valuation
		\begin{align*}
			\mu_{\psi,\delta}:=\int_{\R^n} \psi(y)\pi(y)\mu^\delta dy
		\end{align*}
		belongs to $\overline{W_A}$ for every $0<\delta<\delta_0$, since the integrand belongs to $\overline{W_A}$. Using the representation of $\mu^\delta$, and the equivariance of $\Hess_k$ with respect to translations, we have
		\begin{align*}
			\mu_{\psi,\delta}(f)=&\int_{\R^n}\psi(y)\int_{\R^n}\phi_\delta(x)d\Hess_k(f(\cdot+y))dy\\
			=&\int_{\R^n}\left[\int_{\R^n}\psi(y)\phi_\delta(x-y)dy\right]d\Hess_k(f).
		\end{align*}
		By construction, the function $x\mapsto \int_{\R^n}\psi(y)\phi_\delta(x-y)dy$ is supported on $\supp\psi+B_{\delta}(0)\subset A$ for $0<\delta<\delta_0$. In particular, this function belongs to $F_A$ for every $0<\delta<\delta_0$. Since $\phi_\delta$ is an approximation of the $\delta$-distribution, this function converges uniformly to $\psi$ for $\delta\rightarrow0$. As $F_A$ is closed, this implies $\psi\in F_A$.\\
		
		We thus obtain that $F_A=C_A(\R^n)$ for every compact and convex set $A$ with nonempty interior. Now consider the subspace of $\MAVal_k(\R^n)$ given by
		\begin{align*}
			E=\bigg\{\Psi\in\MAVal_k(\R^n):& \int_{\R^n}\psi d\Psi\in \overline{W_A}\quad\text{for every}~\psi\in C_A(\R^n),\\
			&A\subset\R^n~\text{compact, convex with nonempty interior}\bigg\}.
		\end{align*}
		Then $\Hess_k\in E$ by the previous discussion. We claim that $E$ is $\GL(n,\R)$-invariant. Thus let $g\in\GL(n,\R)$ be given. Fix a compact and convex set $A$ with nonempty interior. For $\Psi\in\MAVal_k(\R^n)$, we will use the notation $\Psi[\phi]$ for $\phi\in C_c(\R^n)$ for the valuation
		\begin{align*}
			\Psi[\phi](f)=\int_{\R^n}\phi d\Psi(f).
		\end{align*}
		From the definition of the action, we obtain
		\begin{align*}
			(\tilde{\pi}(g)\Psi)[\phi]=\pi(g)[\Psi[\phi\circ g]].
		\end{align*}
		If $\phi\in C_A(\R^n)$, then $\phi\circ g\in C_{g^{-1}A}(\R^n)$. Since $\Psi\in E$ by assumption, $\Psi[\phi\circ g]\in \overline{W_{g^{-1}A}}$ for every $\phi\in C_A(\R^n)$. But then
		\begin{align*}
			(\tilde{\pi}(g)\Psi)[\phi]=\pi(g)[\Psi[\phi\circ g]]\in \pi(g)(\overline{W_{g^{-1}A}})\subset \overline{W_A}
		\end{align*}
		for every $\phi\in C_A(\R^n)$, so $\tilde{\pi}(g)\Psi\in E$.\\
		We thus see that $E$ is a nontrivial and $\GL(n,\R)$-invariant subspace of $\MAVal_k(\R^n)$, so Theorem \ref{theorem:MAVal_Irreducible} implies $E=\MAVal_k(\R^n)$. This implies in particular, that for any compact and convex set $A\subset\R^n$, the valuations $\Psi[\phi]$ for $\Psi\in\MAVal_k(\R^n)$ and $\phi\in C_A(\R^n)$ are all contained in $\overline{W_A}$. However, since $A$ is convex, Theorem \ref{theorem:DescriptionSmoothValuations_support} shows that every smooth valuation in $\VConv_{k,A}(\R^n)$ is a sum of valuations of this form. Since smooth valuations are dense in $\VConv_{k,A}(\R^n)$ by Theorem \ref{theorem:approxBySmoothInConvexSupport}, $\overline{W_A}$ is a dense subspace of $\VConv_{k,A}(\R^n)$. Since it is closed, this shows $\overline{W_A}=\VConv_{k,A}(\R^n)$.\\
		As discussed in the beginning, the spaces $\overline{W_A}$ are all contained in the sequential closure of $W$. Since their union is $\VConv_k(\R^n)$, we obtain that $W$ is sequentially dense in $\VConv_k(\R^n)$.
	\end{proof}
	\begin{remark}
		The proof actually shows a stronger result: For every compact and convex set $A\subset\R^n$ with nonempty interior, $W\cap \VConv_{k,A}(\R^n)$ is (sequentially) dense in $\VConv_{k,A}(\R^n)$.
	\end{remark}

	Theorem~\ref{theorem:weakIrred} can be used to obtain a variety of density results for different types of valuations. We present two simple examples.
	\begin{corollary}
		\label{corollary:densityMixedMA}
		Let $\mathcal{F}\subset \Conv(\R^n,\R)$ and $\mathcal{C}\subset C_c(\R^n)$ be affine invariant subsets. If there is $f_0\in\mathcal{F}$ and $\phi_0\in \mathcal{C}$ such that 
		\begin{align*}
			\int_{\R^n}\phi_0(x)d\Hess_{n-k}(f_0)\ne 0,
		\end{align*}
		then valuations of the form
		\begin{align*}
			f\mapsto \int_{\R^n}\phi(x)d\MA(f[k],f_1,\dots,f_{n-k})
		\end{align*}
		for $f_1,\dots,f_{n-k}\in\mathcal{F}$, $\phi\in\mathcal{C}$, span a sequentially dense subspace of $\VConv_k(\R^n)$.
	\end{corollary}
	\begin{proof}
		Obviously the space spanned by these valuations is affine invariant. Note that
		\begin{align*}
			\int_{\R^n}\phi_0(x)d\Hess_{n-k}(f_0)=c_{n,k}\int_{\R^n}\phi_0(x)d\MA(|\cdot|^2[k],f_0[n-k])
		\end{align*}
		for some combinatorial constant $c_{n,k}$. Thus the valuation \begin{align*}
			\mu:=\int_{\R^n}\phi_0(x)d\MA(\cdot[k],f_0[n-k])
		\end{align*} is contained in the space spanned by these valuations and  satisfies $\mu(|\cdot|^2)\ne0$, so the claim follows from Theorem~\ref{theorem:weakIrred}.
	\end{proof}
	Note that Corollary~\ref{corollary:densityMixedMA} applies in particular to the following sets of functions:
	\begin{itemize}
		\item $\{h_\Delta(\cdot-y):\Delta~(n-k)\text{-dimensional simplex},~y\in\R^n\}$,
		\item $\{h_\mathcal{E}(\cdot-y):\mathcal{E}~(n-k)\text{-dimensional ellipsoid},~y\in\R^n\}$,
		\item $\{h_\mathcal{P}(\cdot-y):\mathcal{P}~(n-k)\text{-dimensional parallelotope},~y\in\R^n\}$,
		\item $\{\sum_{j=1}^{n-k}\lambda_j\exp(\langle y_j,\cdot\rangle):\lambda_j\ge 0,~y_j\in\R^n\}$,
		\item $\{q: q~ \text{convex polynomial of degree less or equal}~l\}$, $l\ge 2$.
	\end{itemize}
	It also applies to the space spanned by valuations of the form \eqref{equation:ValuationsMAOperators}.
	\begin{corollary}
		Every valuation in $\VConv_1(\R^n)$ can be approximated uniformly on compact subsets of $\Conv(\R^n,\R)$ by a sequence of valuations of the form
		\begin{align*}
			\mu(f)=\sum_{j=1}^{N}c_j\left(f(x_j)+f(y_j)-2f\left(\frac{x_j+y_j}{2}\right)\right),
		\end{align*}
		where $x_j,y_j\in\R^n$, $c_j\in\C$.
	\end{corollary}
	\begin{proof}
		These valuations form an affine invariant subspace and include the valuations $\tilde{\mu}(f):=f(x)+f(-x)-2f(0)$, $x\ne 0$, which satisfy $\tilde{\mu}(|\cdot|^2)=2|x|^2\ne0$. The claim follows from Theorem~\ref{theorem:weakIrred}.
	\end{proof}

\subsection{Proof of Theorem~\ref{maintheorem:WeakIrreducibility}}
\label{section:affineInvSubspaces}
In this section we complete the proof of Theorem~\ref{maintheorem:WeakIrreducibility}. We start by associating to any closed and translation invariant subspace of $\VConv_k(\R^n)$ a submodule of the $\P(\C^n)$-module $\P(\C^n)\M^2_k$, where $\P(\C^n)$ operates on $\P(\Mat_{n,k}(\C))$ as in Section \ref{section:Prelim_Minors}, i.e. $p\in \P(\C^n)$ acts on $P\in \P(\Mat_{n,k}(\C))$ by
\begin{align*}
	(p\cdot P)[w]=p(w_k)P(w)\quad\text{for}~w=(w_1,\dots,w_k)\in\Mat_{n,k}(\C).
\end{align*}
We call a $\P(\C^n)$-submodule of $\P(\C^n)\M^2_k$ a homogeneous submodule if it is generated by homogeneous elements, where $P\in\P(\Mat_{n,k}(\C))$ is called homogeneous of degree $h\in \mathbb{N}^k$ if $P(t_1w_1,\dots,t_kw_k)=t_1^{h_1}\dots t_k^{h_k}P(w_1,\dots,w_k)$ for $t_1,\dots, t_k\in \C$. For a subspace $W\subset \VConv_k(\R^n)$, let $\I(W)$ denote the subspace of $\P(\C^n)\M^2_k$ generated by the homogeneous terms in the power series expansion of $\F(\mu)$ for $\mu\in W$.

\begin{lemma}
	\label{lemma:propertiesAssociatedIdeal}
	Let $W$ be a closed and translation invariant subspace
	\begin{enumerate}
		\item $\mathcal{I}(W)$ is spanned by the lowest order terms of smooth valuations in $W$.
		\item $\mathcal{I}(W)$ is a homogeneous submodule of $\P(\C^n)\M^2_k$.
		\item If $W$ is $\GL(n,\R)$-invariant, then $\mathcal{I}(W)$ is $\GL(n,\C)$-invariant with respect to the action on $\P(\Mat_{n,k}(\C))$ given by
		\begin{align*}
			(g\cdot P)[w]=P(g^T\cdot w)\quad \text{for}~g\in \GL(n,\C),~P\in\P(\Mat_{n,k}(\C)).
		\end{align*}
	\end{enumerate}
\end{lemma}
\begin{proof}
	Note that for $\phi\in C^\infty_c(\R^n)$ and $\mu\in W$, the valuation
	\begin{align*}
		\mu_\phi(f)=\int_{\R^n}\phi(x)\mu(f(\cdot+x))dx
	\end{align*}
	belongs to $W$ by Lemma \ref{lemma:MollifiedValuations}. Thus (1) and (2) follow with the same argument as the proofs of Lemma \ref{lemma:SubmoduleGeneratedBySmoothValuations} and Corollary \ref{corollary:LowestOrderTermsModuleOverPolynomials}. In order to see that (3) holds, observe that the equivariance properties of the Goodey--Weil distributions and the Fourier--Laplace transform imply that $\mathcal{I}(W)$ is $\GL(n,\R)$-invariant with respect to the given action. Since $\P(\Mat_{n,k}(\C))$ is a rational representation of $\GL(n,\C)$ and $\GL(n,\R)\subset\GL(n,\C)$ is Zariski dense, this implies that $\mathcal{I}(W)$ is $\GL(n,\C)$-invariant.
\end{proof}

The next result is the main technical tool in the proof of Theorem \ref{maintheorem:WeakIrreducibility}. It relies on the behavior of elements in $\MAVal_k(\R^n)$ with respect to the dilations $d_t\in\GL(n,\R)$ given by $d_t(x)=tx$, $x\in\R^n$, $t>0$: Using the characterization of $\MAVal_k(\R^n)$ in Theorem \ref{theorem:ClassificationMAVal}, it is easy to see that 
	\begin{align*}
		\tilde{\pi}(d_t)\Psi=t^{2k-n}\Psi\quad\text{for}~\Psi\in\MAVal_k(\R^n).
	\end{align*}
	Here, $\tilde{\pi}$ denotes the action of $\GL(n,\R)$ on $\MAVal_k(\R^n)$ defined in the previous section.
\begin{proposition}
	\label{proposition:valuationFromHighestWeightVector}
	Let $W\subset \VConv_k(\R^n)$ be a closed and affine invariant subspace. If $Q=w_{1,k}^d\Delta^2_k\in \mathcal{I}(W)$, then $W$ contains all smooth valuations $\mu\in \VConv_k(\R^n)$ with $\F(\mu)\in \mathcal{O}_{\C^n}Q$.
\end{proposition}
\begin{proof}
	Lemma \ref{lemma:propertiesAssociatedIdeal} implies that there exists a smooth valuation $\mu\in W$ such that the lowest order term of $\F(\mu)$ is given by $w_{1,k}^d\Delta^2_k$. We denote the ordered Gröbner basis of $\M^2_k$ from Lemma \ref{lemma:productsMinorsGroebnerBasis} by $P_1,\dots,P_{N_{n,k}}$, so in particular $P_1=\Delta_k^2$. Combining Theorem \ref{theorem:DivisionAlg} Lemma \ref{lemma:productsMinorsGroebnerBasis} we obtain unique functions $g_j\in\mathcal{O}_{\C^n}$, $1\le j\le N_{n,k}$, such that
	$\F(\mu)=\sum_{j=1}^N g_j(w_k)P_j(w)$ and such that no monomial in the power series expansion of
	\begin{align}
		\label{equation:decomposition_Mu_lowestorder}	
		\F(\mu)-\sum_{j=1}^lg_j(w_k)P_j(w)=\sum_{j=l+1}^{N_{n,k}} g_j(w_k)P_j(w)
	\end{align}
	is divisible by the initial term of $P_1,\dots,P_l$ for any $1\le l\le N$. Moreover,
	\begin{align*}
		|g_j(z)|\le C (1+|z|)^{m_0} \sup_{\zeta\in D_1(0,\dots,0,z)}|\F(\mu)[\zeta]|
	\end{align*}
	for some $m_0>0$, $C>0$. Since $\mu$ is a smooth valuation, Lemma~\ref{lemma:PWS_SatisfiedBySmoothValuations} implies that there exist constants $C_m>0$ such that
	\begin{align*}
		\sup_{\zeta\in D_1(0,\dots,0,z)}|\F(\mu)[\zeta]|\le C_m(1+|z|)^{-m}e^{h_A(\Im(z))}
	\end{align*}
	for every $m\in\mathbb{N}$, where $A\subset \R^n$ is a compact convex set containing $\supp\mu$. Thus $g_j$ satisfies the assumptions of the Paley--Wiener--Schwartz Theorem~\ref{theorem:PaleyWienerSchwartz}. In particular, $g_j=\mathcal{F}(\phi_j)$ for some $\phi_j\in C_c^\infty(\R^n)$, which implies
	\begin{align}
		\label{eq:formulaMuHighestWeightTerm}
		\mu=\sum_{j=1}^{N_{n,k}}\int_{\R^n}\phi_j d\Psi_j(\cdot),
	\end{align} 
	where $\Psi_j\in \MAVal_k(\R^n)$ is the unique element with $Q[\Psi_j]=P_j$, compare Lemma \ref{lemma:QBijectiveMA}. For $l=1$, \eqref{equation:decomposition_Mu_lowestorder} implies that no monomial in the power series expansion of the function 
	\begin{align*}
			\F(\mu)-g_1(w_k)\Delta_k^2(w)=\sum_{j=2}^{N_{n,k}} g_j(w_k)P_j(w)
	\end{align*}
	is divisible by the initial term of $\Delta_k^2$. Since the lowest order term on $\F(\mu)$ is $w_{1,k}^d\Delta_k^2$ by assumption, this is only possible if the lowest order term of $g_1(w_k)\Delta_k^2(w)$ is given by the same term, i.e. the lowest order term of $\mathcal{F}(\phi_1)[z]=g_1(z)$ is given by $z_1 ^d$. This implies that the lowest order term of the function on the right hand side is of order at least $2k+d+1$. Since no monomial in the power series expansion of
	\begin{align*}
			\F(\mu)-\sum_{j=1}^lg_j(w_k)P_j(w)=\sum_{j=l+1}^{N_{n,k}} g_j(w_k)P_j(w)
	\end{align*}
	is divisible by the initial term of $P_1,\dots,P_l$ for any $1\le l\le N$, it now easily follows by induction that the lowest order terms of $g_j=\mathcal{F}(\phi_j)$ are of order at least $d+1$ for $2\le j\le N_{n,k}$.\\
	
	Let $\psi\in C^\infty_c(\R^n)$ and $\epsilon>0$ be given. Then the valuations $\mu_{\psi,\epsilon}$ defined by
	\begin{align}
		\label{equation:definition_muJ}
		\mu_{\psi,\epsilon}(f):=&\epsilon^{-d-2k}\int_{\R^n}\psi(y)\mu(f(\epsilon(\cdot)+y))dy
	\end{align}
	 belong to $W$ since $W$ is translation invariant and closed. Moreover, using the translation equivariance of $\Psi_j$, the behavior under dilations, as well as Fubini's Theorem, \eqref{eq:formulaMuHighestWeightTerm} implies that
	 \begin{align}
	 	\notag
	 	\mu_{\psi,\epsilon}(f)=&\sum_{j=1}^{N_{n,k}}\epsilon^{-d-2k}\int_{\R^n}\psi(y)\int_{\R^n}\phi_j(x) d\Psi_j(f(\epsilon(\cdot)+y;x))dy\\
	 	\notag
	 	=&\sum_{j=1}^{N_{n,k}}\epsilon^{-d-2k}\int_{\R^n}\psi(y)\int_{\R^n}\phi_j\left(\frac{x-y}{\epsilon}\right) \epsilon^{2k-n}d\Psi_j(f;x)dy\\
	 	\label{eq:defMuRefined}
	 	=&\sum_{j=1}^{N_{n,k}}\epsilon^{-d}\int_{\R^n}\left[\int_{\R^n}\psi\left(x-\epsilon y\right)\phi_j(y)dy\right] d\Psi_j(f;x).
	 \end{align}
 	We claim that these valuations converge for $\epsilon\rightarrow0$ to 
 	\begin{align*}
 		\mu_\psi(f):=\int_{\R^n}(-i\partial_1)^d\psi\left(x\right) d\Psi_1(f;x).
 	\end{align*}
 	If this holds, $\mu_\psi\in W$ for every $\psi\in C^\infty_c(\R^n)$ since $W$ is closed. As \begin{align*}
 		\F(\mu_\psi)[w]=&\mathcal{F}((-i\partial_1)^d\psi)[w_k] \Delta_k^2(w)=w_{1,k}^d\mathcal{F}(\psi)[w_k] \Delta_k^2(w)
 	\end{align*}
 	by Theorem \ref{theorem:FourierMA}, this implies the claim.\\
 	
 	Let us therefore show that the valuations defined by \eqref{equation:definition_muJ} converge to the stated limit. As discussed before, the lowest order term of $\mathcal{F}(\phi_j)$ is of order at least $d+1$ for $2\le j\le N_{n,k}$. By \eqref{eq:defMuRefined}, we have 
 	\begin{align*}
 		\mu_{\psi,\epsilon}(f)=\sum_{j=1}^{N_{n,k}}\int_{\R^n}\psi_{j,\epsilon}(x)d\Psi_j(f;x)
 	\end{align*}
 	for the functions $\psi_{j,\epsilon}\in C^\infty_c(\R^n)$ given by
 	\begin{align*}
 		\psi_{j,\epsilon}(x):=\epsilon^{-d}\int_{\R^n}\psi\left(x-\epsilon y\right)\phi_j(y)dy.
 	\end{align*}
 	We claim that the support of $\psi_{j,\epsilon}$ is uniformly bounded for $\epsilon\in(0,1]$ and that it converges uniformly to $(-i\partial_1)^d\psi$ for $\epsilon\rightarrow0$ if $j=1$ and $0$ otherwise. Once this is established, the convergence follows directly from Corollary~\ref{corollary:trivialBoundMAOperators}.\\
 	First note that the support of $\psi_{j,\epsilon}$ is contained in $\supp\psi -\epsilon\supp\phi_j$, which is uniformly bounded for $\epsilon\in(0,1]$. Next, we consider the Taylor expansion of $\psi$ up to order $d$. We have
 	\begin{align*}
 		\psi(x-\epsilon y)=\sum_{|\alpha|\le d} (-1)^{|\alpha|}\frac{\partial^\alpha \psi(x)}{ \alpha!}\epsilon^{|\alpha|}y^\alpha+ R_{d}(x,y,\epsilon),
 	\end{align*}
 	where $|R_{d}(x,y,\epsilon)|\le C\epsilon^{d+1}|y|^{d+1}$ for a constant $C>0$ independent of $x,y\in\R^n$, $\epsilon>0$, since $\psi$ has compact support. In particular,
 	\begin{align*}
 		\psi_{j,\epsilon}(x)=&\epsilon^{-d}\sum_{|\alpha|\le d} (-1)^{|\alpha|}\frac{\partial^\alpha \psi(x)}{ \alpha!}\epsilon^{|\alpha|}\int_{\R^n}y^\alpha\phi_j(y)dy\\
 		&+ \epsilon^{-d}\int_{\R^n}R_{d}(x,y,\epsilon) \phi_j(y)dy.
 	\end{align*}
 	By assumption, the lowest order term of $\mathcal{F}(\phi_1)[z]$ is $ z_1^d$, whereas the lowest order terms of $\mathcal{F}(\phi_j)$ are of order at least $d+1$ for $2\le j\le N_{n,k}$. As in the proof of Proposition~\ref{proposition:prescribedFourier}, this implies for $|\alpha|\le d$ 
 	\begin{align*}
 		\int_{\R^n}y^\alpha\phi_j(y)dy=\begin{cases}
 			\frac{d!}{(-i)^d}&\text{for}~ j=1, \alpha=(d,0,\dots,0),\\
 			0 &\text{else}.
 		\end{cases}
 	\end{align*}
 	For $j=1$, this implies
 	\begin{align*}
 		\psi_{1,\epsilon}(x)=&\frac{(-1)^{d}}{(-i)^d}\partial_1^d\psi(x)+ \epsilon^{-d}\int_{\R^n}R_{d}(x,y,\epsilon) \phi_1(y)dy,
 	\end{align*}
 	and we obtain
 	\begin{align*}
 		|\psi_{1,\epsilon}(x)-(-i\partial_1)^d\psi(x)|\le C\epsilon \int_{\R^n}|y|^{d+1}|\phi_j(y)|dy.
 	\end{align*}
 	Since $C$ is independent of $x\in\R^n$ and $\epsilon>0$, this shows that $\psi_{1,\epsilon}$ converges uniformly to the desired limit. For $2\le j\le N_{n,k}$, we similarly obtain
 	\begin{align*}
 		\psi_{j,\epsilon}(x)=& \epsilon^{-d}\int_{\R^n}R_{d}(x,y,\epsilon) \phi_j(y)dy,
 	\end{align*}
 	and the same estimate shows that this function converges uniformly to $0$.
\end{proof}

Consider the closed subspaces of $\VConv_k(\R^n)$ defined for $d\in\mathbb{N}$ by
\begin{align*}
	W_d=&\overline{\mathrm{span}\left\{\int_{\R^n}\partial^\alpha\phi d\Psi: \alpha\in\mathbb{N}^n,~|\alpha|=d,~\phi\in C_c^\infty(\R^n),\Psi\in\MAVal_k(\R^n)\right\}}.
\end{align*}
\begin{lemma}
	\label{lemma:description_Wd}
	The following are equivalent for $\mu\in\VConv_k(\R^n)$:
	\begin{enumerate}
		\item $\mu\in W_d$.
		\item The power series expansion of $\F(\mu)$ vanishes up to order $2k+d-1$.
		\item For all real polynomials $P_1,\dots,P_k$ with $\sum_{j=1}^k\deg P_j\le 2k+d-1$:
		\begin{align*}
			\frac{\partial^k}{\partial\lambda_1\dots\partial\lambda_k}\Big|_0\mu\left(|x|^2+|x|^{2(k+d)}+\sum_{j=1}^k\lambda_j P_j(x)\right)=0.
		\end{align*}
		\item  For all real polynomials $P_1,\dots,P_k$ with $\sum_{j=1}^k\deg P_j\le 2k+d-1$:
			\begin{align*}
				\GW(\mu)[P_1\otimes\dots\otimes P_k]=0.
			\end{align*}
	\end{enumerate}
\end{lemma}
\begin{proof}
	First, the defining property of $\GW(\mu)$ implies that (3) and (4) are equivalent. Moreover, (2) and (4) are equivalent since the Goodey--Weil distributions have compact support.\\
	It remains to see that (1) and (2) are equivalent. Let us show that (1) implies (2).
	Note that the maps $L_\alpha:\mathcal{O}_{\Mat_{n,k}(\C)}\rightarrow\C$, $F\mapsto \frac{\partial ^\alpha F(0)}{\alpha!}$ are continuous for $\alpha\in \mathbb{N}^{n\times k}$ and that $\F(\mu)$ vanishes up to order $2k+d-1$ if and only if it is in the joint kernel of these maps for $|\alpha|\le 2k+d-1$. For valuations of the form $\int_{\R^n}\partial^\alpha\phi d\Psi$ with $\alpha\in\mathbb{N}^n$, $|\alpha|=d$, $\phi\in C_c^\infty(\R^n)$, $\Psi\in\MAVal_k(\R^n)$, the implication (1) $\Rightarrow$ (2) follows from Theorem \ref{theorem:FourierMA}. Since $\F:\VConv_k(\R^n)\rightarrow\mathcal{O}_{\Mat_{n,k}(\C)}$ is continuous by Proposition \ref{proposition:ContinuityFourier}, and valuations of this form are a dense subspace of $W_d$ by construction, this shows that (1) implies (2) for every valuation.\\
	In order to see that (2) implies (1), note that the space $W$ of all valuations satisfying (2) is closed in $\VConv_k(\R^n)$ due to Proposition \ref{proposition:ContinuityFourier} as well as affine invariant. In order to show that $W\subset W_d$, it is thus sufficient to show that every smooth valuation in $W$ belongs to $W_d$, since smooth valuations are dense in $W$ by Theorem \ref{theorem:approxBySmoothInConvexSupport} and $W_d$ is closed. Thus let $\mu$ be a smooth valuation satisfying (2). If we take the Gröbner basis $P_1,\dots,P_{N_{n,k}}$ of $\M^2_k$ from Lemma \ref{lemma:productsMinorsGroebnerBasis}, we can apply Theorem \ref{theorem:DivisionAlg} to obtain unique functions $g_j$ such that no term of the power series expansion of 
	\begin{align*}
		\F(\mu)[w]-\sum_{j=1}^{l} g_j(w_k)P_j(w)=\sum_{j=l+1}^{N_{n,k}} g_j(w_k)P_j(w)
	\end{align*}
	contains a monomial divisible by the initial term of $P_j$, $1\le j\le l$, for each $1\le l\le N_{n,k}$. As in the proof of Lemma \ref{proposition:valuationFromHighestWeightVector}, a simple induction argument shows that the power series expansion of $g_j$ vanishes up to order $d-1$. Now consider the ideal of $\P(\C^n)$ generated by the homogeneous polynomials of degree $d$. The monomials of degree $d$ form a Gröbner basis for this module, so since every homogeneous term of the power series expansion of $g_j$ belongs to this module, we may in particular apply Theorem \ref{theorem:DivisionAlg} to this module with this Gröbner basis to obtain functions $h_{j\alpha}$, $\alpha\in \mathbb{N}^n$ with $|\alpha|\le d$, such that $g_j=\sum_{|\alpha|=d} h_{j\alpha}z^\alpha$, where
	\begin{align*}
		|h_{j\alpha}(z)|\le C(1+|z|)^m \sup_{\zeta\in D_1(z)}|g_j(z)|
	\end{align*}
	for some $m\in\mathbb{N}$ and a constant $C>0$ independent of $z\in\C^n$. Since the functions $g_j\in\mathcal{O}_{\C^n}$ satisfy the estimates in Theorem \ref{theorem:DivisionAlg} and $\F(\mu)$ satisfies the estimates in Lemma \ref{lemma:PWS_SatisfiedBySmoothValuations}, it is easy to see that $h_{j\alpha}$ satisfies the conditions of the Paley--Wiener--Schwartz Theorem \ref{theorem:PaleyWienerSchwartz}, so $h_{j\alpha}=\F(\phi_{j\alpha})$ for some $\phi_{j\alpha}\in C^\infty_c(\R^n)$. Then
	\begin{align*}
		\mu(f)=\sum_{j=1}^{N_{n,k}}\sum_{|\alpha|=d}\int_{\R^n} (-i\partial)^\alpha\phi_{j\alpha}d\Psi_j(f)
	\end{align*}
	since both sides map to the same function under the injective map $\F$. Thus $\mu\in W_d$.\\
	This completes the proof that (2) implies (1).
\end{proof}

\begin{corollary}
	$W_d$ has finite codimension in $\VConv_k(\R^n)$.
\end{corollary}
\begin{proof}
	Consider the linear maps $L_\alpha:\mathcal{O}_{\Mat_{n,k}(\C)}\rightarrow\C$, $F\mapsto \frac{\partial ^\alpha F(0)}{\alpha!}$ for $\alpha\in \mathbb{N}^{n\times k}$. By Lemma~\ref{lemma:description_Wd} (2),
	\begin{align*}
		W_d=\bigcap_{|\alpha|\le 2k+d-1}\ker\left(L_\alpha\circ \F\right).
	\end{align*}
	Thus $W_d$ has finite codimension.
\end{proof}

The following provides the proof of the first two parts of Theorem~\ref{maintheorem:WeakIrreducibility}. Note that Theorem~\ref{maintheorem:WeakIrreducibility} (3) was already proved in Theorem~\ref{theorem:weakIrred}, however, we present a short alternative proof after the next result. 
\begin{theorem}
	\label{theorem:ClassificationAffineInvariantSubspaces}
	Let $W\subset\VConv_k(\R^n)$ be a closed and affine invariant subspace. Then there exists $d\in\mathbb{N}$ such that $W$ is equal to $W_d$. In particular, $W$ has finite codimension.
\end{theorem}
\begin{proof}
	If $W\subset\VConv_k(\R^n)$ is a closed and affine invariant subspace, then $\I(W)\subset\P(\C^n)\M^2_k$ is a $\GL(n,\C)$-invariant submodule by Lemma \ref{lemma:propertiesAssociatedIdeal} and thus a representation of $\GL(n,\C)$. If $Q\in \I(W)$ is a highest weight vector, then $Q(w)=\Delta_k^2(w)w_{1,k}^d$ for some $d\in\mathbb{N}$ by Corollary~\ref{corollary:lowestWeightvectorsquaredMinors}. Since $\I(W)$ is a $\P(\C^n)$-submodule, it thus contains the elements $\Delta_k^2(w)w_{1,k}^{d'}$ for all $d'\ge d$, which implies that $\mathcal{I}(W)$ is the direct sum of the irreducible $\GL(n,\C)$-subrepresentations of $\P(\C^n)\M^2_k$ corresponding to these highest weight vectors.\\
	Let $d\in\mathbb{N}$ be the smallest number such that $w_{1,k}^d\Delta_k^2\in \mathcal{I}(W)$. Then $\mathcal{I}(W)$ does not contain any homogeneous polynomials of order less than $2k+d$, so the power series expansion of any valuation in $W$ vanishes up to order $2k+d-1$. In particular, $W\subset W_d$ by Lemma \ref{lemma:RestrictionLowestOrderTermsSubspaces}.\\
	For the converse inclusion, we will show that $W$ contains the valuations $\int_{\R^n}\partial^\alpha\phi d\Psi$ for $\phi\in C^\infty_c(\R^n)$, $\Psi\in \MAVal_k(\R^n)$ and $\alpha\in \mathbb{N}^n$, $|\alpha|=d$. Let $\P_d(\C^n)$ denote the space of homogeneous polynomials on $\C^n$ of degree $d$ and consider the linear map
	\begin{align*}
		\Phi:\P_d(\C^n)\otimes \MAVal_k(\R^n)&\rightarrow \P_d(\C^n)\M_k^2\\
		P\otimes \Psi&\mapsto P(w_k)Q(\Psi),
	\end{align*}
	where $Q:\MAVal_k(\R^n)\rightarrow\M_k^2$ is the map in Theorem \ref{theorem:FourierMA}. We similarly have a bilinear map
	\begin{align*}
		\tilde{\Phi}:\P_d(\C^n)\otimes \MAVal_k(\R^n)\times C^\infty_c(\R^n)&\rightarrow W_d\subset\VConv_k(\R^n)\\
		(P\otimes \Psi,\phi)&\mapsto \int_{\R^n}P(-i\partial)\phi(x)d\Psi.
	\end{align*}
	It follows directly from the construction that
	\begin{align*}
		\F(\tilde{\Phi}(P\otimes \Psi,\phi))[w]=\Phi(P\otimes\Psi)[w]\cdot \mathcal{F}(\phi)[w_k].
	\end{align*}
	In particular, if we set $M:=\P_d(\C^n)\otimes \MAVal_k(\R^n)/\ker\Phi$, then $\tilde{\Phi}$ induces a well defined bilinear map
	\begin{align*}
		\overline{\Phi}:M\times C^\infty_c(\R^n)&\rightarrow W_d.
	\end{align*}
	If we let $\GL(n,\C)$ act on $\P_d(\C^n)$ by $(g\cdot P)[z]=P(g^{T}z)$, then $\Phi$ and $\tilde{\Phi}$ become $\GL(n,\R)$-equivariant with respect to the natural operation of $\GL(n,\R)$ on these spaces. In particular, $M$ is isomorphic to $\P_d(\C^n)\M^2_k$, which is generated by a single $\GL(n,\C)$-highest weight vector by Corollary \ref{corollary:lowestWeightvectorsquaredMinors} and thus an irreducible representation of $\GL(n,\C)$. Since $\GL(n,\R)\subset \GL(n,\C)$ is Zariski dense and since this is a rational representation, this implies that $M$ is an irreducible representation of $\GL(n,\R)$. Consider the subspace of $M$ given by
	\begin{align*}
		E=\{Q\in M: \overline{\Phi}(Q, \phi)\in W~\text{for all}~\phi\in C^\infty_c(\R^n)\}.
	\end{align*}
	Since $w_{1,k}^d\Delta_k^2\in \mathcal{I}(W)$ by assumption, Proposition \ref{proposition:valuationFromHighestWeightVector} shows that $w_{1,k}^d\Delta_k^2\in E$, so $E$ is a nontrivial subspace. We claim that $E$ is a $\GL(n,\R)$-invariant subspace of $M$. Since $\overline{\Phi}$ is $\GL(n,\R)$-equivariant, we have for $Q\in M$, $\phi\in C^\infty_c(\R^n)$, and $g\in\GL(n,\R)$,
	\begin{align*}
		\overline{\Phi}(g\cdot Q, \phi)=&\overline{\Phi}(g\cdot Q,g\cdot g^{-1}\cdot \phi)\\
		=&\pi(g)\overline{\Phi}(Q, g^{-1}\cdot\phi),
	\end{align*}
	where $\pi(g)$ denotes the action of $g\in \GL(n,\R)$ on $\VConv_k(\R^n)$. Since $Q\in E$, $\Phi(Q, g^{-1}\cdot\phi)\in W$ for every $\phi\in C^\infty_c(\R^n)$. As $W$ is $\GL(n,\R)$-invariant, this implies that 
	\begin{align*}
		\overline{\Phi}(g\cdot Q, \phi)\in W\quad \text{for every}~\phi\in C^\infty_c(\R^n),
	\end{align*}
	i.e. $g\cdot Q\in E$. Since $M$ is a $\GL(n,\R)$-irreducible representation, we deduce that $E=M$. But this implies that $W$ contains every valuation of the form 
	\begin{align*}
		\int_{\R^n}\partial^\alpha \phi(x)d\Psi
	\end{align*}
	for $\alpha\in\mathbb{N}^k$, $|\alpha|=d$, $\phi\in C^\infty_c(\R^n)$, and $\Psi\in\MAVal_k(\R^n)$. Since $W$ is closed, we obtain $W_d\subset W$ from the definition of $W_d$, which completes the proof.
\end{proof}
Let us briefly discuss how this result implies Theorem~\ref{maintheorem:WeakIrreducibility} (3).
\begin{proof}[Alternative proof of Theorem~\ref{maintheorem:WeakIrreducibility} (3)]
	If $W\subset\VConv_k(\R^n)$ is a closed affine invariant subspace, then $W=W_d$ for some $d\ge0$ by Theorem~\ref{theorem:ClassificationAffineInvariantSubspaces}. If there is a positive semi-definite quadratic form $q$ on $\R^n$ and a valuation $\mu\in  W$ with $\mu(q)\ne 0$, then $\mu$ cannot be an element of $W_{d'}$ for $d'>0$ due to Lemma~\ref{lemma:description_Wd} (3). Thus $d=0$, i.e. $W=W_0=\VConv_k(\R^n)$. 
\end{proof}

We note the following two consequences of Theorem~\ref{theorem:ClassificationAffineInvariantSubspaces}.
\begin{corollary}
	Let $1\le k\le n$. The set of $\Aff(n,\R)$-invariant closed subspaces of $\VConv_k(\R^n)$ is countable.
\end{corollary}
\begin{corollary}
	The representation of $\Aff(n,\R)$ on $\VConv_k(\R^n)$ is Noetherian: Every increasing sequence $V_0\subset V_1\subset \dots$ of closed $\Aff(n,\R)$-invariant subspaces becomes stationary.
\end{corollary}

\bibliographystyle{plain}
\bibliography{../../library/library.bib}

\Addresses
\end{document}